\documentclass[reqno, 11pt]{amsart}  

\usepackage{amssymb}
\usepackage{mathtools}

\usepackage{graphicx}
\usepackage[usenames,dvipsnames,svgnames]{xcolor}
\usepackage[shortlabels]{enumitem} 

\usepackage[colorlinks, urlcolor=Navy, linkcolor=Navy, citecolor=Navy]{hyperref} 

\usepackage[all]{xy}
\usepackage{tikz}
\usepackage{tikz-cd}
\usetikzlibrary{arrows,shapes,chains,matrix,positioning,scopes,cd,patterns}
\pgfdeclarelayer{bg}    
\pgfsetlayers{bg,main}

\DeclareFontFamily{U}{mathx}{\hyphenchar\font45}
\DeclareFontShape{U}{mathx}{m}{n}{
      <5> <6> <7> <8> <9> <10>
      <10.95> <12> <14.4> <17.28> <20.74> <24.88>
      mathx10
      }{}
\DeclareSymbolFont{mathx}{U}{mathx}{m}{n}
\DeclareFontSubstitution{U}{mathx}{m}{n}
\DeclareMathAccent{\widecheck}{0}{mathx}{"71}
\DeclareMathAccent{\wideparen}{0}{mathx}{"75}

\usepackage[top=30truemm,bottom=30truemm,left=20truemm,right=20truemm]{geometry}

\theoremstyle{plain}
\newtheorem{thm}{Theorem}[section]
\newtheorem{lem}[thm]{Lemma}

\newtheorem{propdef}[thm]{Definition-Proposition}
\newtheorem{prop}[thm]{Proposition}
\newtheorem{cor}[thm]{Corollary}
\newtheorem{thm*}{Theorem}
\theoremstyle{definition}
\newtheorem{dfn}[thm]{Definition}
\newtheorem{exa}[thm]{Example}
\newtheorem{rem}[thm]{Remark}

\newtheorem{assum}[thm]{Assumption}
\newtheorem{ques}[thm]{Question}
\numberwithin{equation}{section}

\DeclareMathOperator{\Hom}{Hom}
\DeclareMathOperator{\Aut}{Aut}
\DeclareMathOperator{\End}{End}
\DeclareMathOperator{\Ext}{Ext}

\DeclareMathOperator{\Tor}{Tor}

\renewcommand{\mod}{\mathop{\mathsf{mod}}}
\DeclareMathOperator{\Mod}{\mathsf{Mod}}

\DeclareMathOperator{\Add}{\mathsf{Add}}
\DeclareMathOperator{\add}{\mathsf{add}}
\DeclareMathOperator{\Gen}{\mathsf{Gen}}
\DeclareMathOperator{\gen}{\mathsf{gen}}
\DeclareMathOperator{\proj}{\mathsf{proj}}
\DeclareMathOperator{\Proj}{\mathsf{Proj}}

\DeclareMathOperator{\Fac}{\mathsf{gen}}

\DeclareMathOperator{\fp}{\mathsf{fp}}
\DeclareMathOperator{\Fl}{\mathsf{fl}}
\DeclareMathOperator{\Filt}{\mathsf{Filt}}

\newcommand{\xto}[1]{\xrightarrow{#1}}

\DeclareMathOperator{\rad}{rad}

\DeclareMathOperator{\Tr}{Tr}
\DeclareMathOperator{\Ker}{Ker} 
\DeclareMathOperator{\Cok}{Cok}
\renewcommand{\Im}{\mathop{\mathrm{Im}}}

\DeclareMathOperator{\ann}{Ann}
\DeclareMathOperator{\Gal}{Gal}

\DeclareMathOperator{\thick}{\mathsf{thick}}

\DeclareMathOperator{\Spec}{Spec}
\newcommand{\ass}{\operatorname{Ass}\nolimits}
\newcommand{\supp}{\operatorname{Supp}\nolimits}
\newcommand{\MSpec}{\operatorname{Max}\nolimits}

\newcommand{\ideal}[1]{\left\langle #1 \right\rangle}

\newcommand{\ov}[1]{\overline{#1}}

\DeclareMathOperator{\id}{id} 
\renewcommand{\epsilon}{\varepsilon}
\renewcommand{\tilde}{\widetilde}

\newcommand{\la}{\lambda}

\renewcommand{\L}{\Lambda}

\newcommand{\bF}{\mathbb{F}}

\newcommand{\bQ}{\mathbb{Q}}

\newcommand{\bS}{\mathbb{S}}
\newcommand{\bT}{\mathbb{T}}

\newcommand{\bZ}{\mathbb{Z}}

\newcommand{\m}{\mathfrak{m}}
\newcommand{\n}{\mathfrak{n}}
\newcommand{\p}{\mathfrak{p}}
\newcommand{\q}{\mathfrak{q}}

\newcommand{\mfm}{\mathfrak{m}}

\newcommand{\mfp}{\mathfrak{p}}
\newcommand{\mfq}{\mathfrak{q}}

\newcommand{\cA}{\mathcal{A}}
\newcommand{\cB}{\mathcal{B}}
\newcommand{\cC}{\mathcal{C}}

\newcommand{\cF}{\mathcal{F}}

\newcommand{\cS}{\mathcal{S}}
\newcommand{\cT}{\mathcal{T}}
\newcommand{\cU}{\mathcal{U}}

\newcommand{\cX}{\mathcal{X}}
\newcommand{\cY}{\mathcal{Y}}
\newcommand{\cZ}{\mathcal{Z}}

\newcommand{\sD}{\mathsf{D}}

\newcommand{\sK}{\mathsf{K}}

\newcommand{\sM}{\mathsf{M}}

\newcommand{\sP}{\mathsf{P}}

\newcommand{\sT}{\mathsf{T}}

\newcommand{\serre}{\mathop{\mathsf{serre}}}
\newcommand{\tors}{\mathop{\mathsf{tors}}}
\newcommand{\torf}{\mathop{\mathsf{torf}}}
\newcommand{\ftors}{\mathop{\mathsf{f}\text{-}\mathsf{tors}}}
\newcommand{\stors}{\mathop{\mathsf{s}\text{-}\mathsf{tors}}}

\newcommand{\simple}{\mathop{\mathsf{sim}}}
\newcommand{\Simple}{\mathsf{Sim}}

\newcommand{\psilt}{\mathop{\mathsf{psilt}}}
\newcommand{\silt}{\mathop{\mathsf{silt}}}
\newcommand{\siltm}{\mathop{\mathsf{m}\text{-}\mathsf{silt}}}
\newcommand{\twopsilt}{\mathop{\mathsf{2}\text{-}\mathsf{psilt}}}
\newcommand{\twosilt}{\mathop{\mathsf{2}\text{-}\mathsf{silt}}}

\begin{document}
\title[Classifying subcategories of modules over Noetherian algebras]{Classifying subcategories of \\ modules over Noetherian algebras}
\dedicatory{Dedicated to Professor Henning Krause on the occasion of his 60th birthday.}
\author[O. Iyama]{Osamu Iyama}
\author[Y. Kimura]{Yuta Kimura}
\address{Osamu Iyama : Graduate School of Mathematical Sciences, The University of Tokyo, 3-8-1 Komaba Meguro-ku Tokyo 153-8914, Japan}
\email{iyama@ms.u-tokyo.ac.jp}
\address{Yuta Kimura : Graduate School of Mathematical Sciences, The University of Tokyo, 3-8-1 Komaba Meguro-ku Tokyo 153-8914, Japan}
\email{ykimura@ms.u-tokyo.ac.jp}
\makeatletter
\@namedef{subjclassname@2020}{%
  \textup{2020} Mathematics Subject Classification}
\makeatother
\begin{abstract}
The aim of this paper is to unify classification theories of torsion classes of finite dimensional algebras and commutative Noetherian rings.
For a commutative Noetherian ring $R$ and a module-finite $R$-algebra $\L$, we study the set $\tors \L$ (respectively, $\torf\L$) of torsion (respectively, torsionfree) classes of the category of finitely generated $\L$-modules.
We construct a bijection from $\torf\L$ to $\prod_\p \torf(\kappa(\p)\otimes_R \L)$, and an embedding $\Phi_{\rm t}$ from $\tors\L$ to $\bT_R(\L):=\prod_\p \tors(\kappa(\p)\otimes_R \L)$, where $\p$ runs all prime ideals of $R$.
When $\L=R$, these give classifications of torsionfree classes, torsion classes and Serre subcategories of $\mod R$ due to Takahashi, Stanley-Wang and Gabriel.
To give a description of $\Im\Phi_{\rm t}$, we introduce the notion of compatible elements in $\bT_R(\L)$, and prove that all elements in $\Im\Phi_{\rm t}$ are compatible.
We give a sufficient condition on $(R, \L)$ such that all compatible elements belong to $\Im\Phi_{\rm t}$ (we call $(R, \L)$ compatible in this case).
For example, if $R$ is semi-local and $\dim R\leq 1$, then $(R, \L)$ is compatible.
We also give a sufficient condition in terms of silting $\L$-modules.
As an application, for a Dynkin quiver $Q$, $(R, RQ)$ is compatible and we have a poset isomorphism $\tors RQ \simeq \Hom_{\rm poset}(\Spec R, \mathfrak{C}_Q)$ for the Cambrian lattice $\mathfrak{C}_Q$ of $Q$.
\end{abstract}
\maketitle
\tableofcontents
\section{Introduction}\label{section-introduction}

A \emph{torsion class} (respectively, \emph{torsionfree class}) is a full subcategory of an abelian category $\cA$ which is closed under extensions and factor objects (respectively, subobjects).
It appears naturally in several branches of mathematics and plays important roles.
Among others, torsion classes in $\cA$ are closely related to derived equivalences of $\cA$.
In fact, a torsion class often gives rise to a \emph{torsion pair} $(\cT, \cF)$, which corresponds bijectively with \emph{intermediate $t$-structure} \cite{BBD, BY} in the derived category $\sD^{\rm b}(\cA)$.

When $\cA$ is the category of (finitely generated) modules over a finite dimensional algebra $A$ over a field, then a connection between torsion classes and classical tilting modules was well understood in the last century, see \cite[Chapter VI. 6]{ASS}.
The recent development of tilting theory is often called as \emph{silting theory}, where the notion of $\tau$-tilting modules (also called as \emph{silting modules}) plays a central role as a completion of the classical tilting modules from a point of view of mutation, and various applications have been found, see e.g. \cite{AHMW, Adachi-Iyama-Reiten, AHIKM, ALS, Asai-semi, Asai-Iyama, BY, CD, Koshio-Kozakai, Ingalls-Thomas, MP, STV, Sentieri, Yurikusa}.
Silting theory was developed also for arbitrary rings \cite{Iyama-Jorgensen-Yang} even in the setting of infinitely generated modules \cite{AMV, An}.

The notion of classical tilting modules is important for a number of rings, e.g. preprojective algebras and Calabi-Yau algebras \cite{Iyama-Reiten, BIRS}, and plays a crucial role in the study of cluster tilting for Cohen-Macaulay modules and non-commutative crepant resolutions of singularities \cite{Iyama-Wemyss-maximal}.
Therefore it will be very important to generalize silting theory to a more general class of rings.
In this paper we continue to develop silting theory for Noetherian algebras recently studied in \cite{Gnedin, GIK, Kimura}.

The main aim of this paper is to unify classification theories of torsion classes of finite dimensional algebras and commutative Noetherian rings by studying torsion classes of Noetherian algebras.
Recall that a pair $(R, \L)$ is a \emph{Noetherian algebra} if  $R$ is a commutative Noetherian ring and $\L$ is (not necessary commutative) an $R$-algebra which is finitely generated as an $R$-module.
We study the set $\tors\L$ (respectively, $\tors(\Fl\L)$) of torsion classes of the category $\mod\L$ of finitely generated $\L$-modules (respectively, the category $\Fl\L$ of finite length $\L$-modules),
and $\torf\L$ (respectively, $\torf(\Fl\L)$) the set of all torsionfree classes of $\mod\L$ (respectively, $\Fl\L$).
These sets are partially ordered by inclusion.

Our first result gives an isomorphism from $\torf\L$ to a product of $\torf(\Fl\L_\p)$, and an embedding of $\tors\L$ into a product of $\tors(\Fl\L_\p)$, where $\p$ runs all prime ideals of $R$.
Let $\kappa(\p):=R_\p/\p R_\p$ be the residue field of $R$ at $\p\in\Spec R$, and we consider a finite dimensional $\kappa(\p)$-algebra
\[\L(\p):=\L_\p/\p\L_\p=\kappa(\p)\otimes_{R_\p}\L_\p=\kappa(\p)\otimes_R\L.\]
By \cite[Theorem 5.4]{Kimura} (see Proposition \ref{prop-torsion-fl}), there exist isomorphisms of posets
\[
\tors(\Fl\L_\p) \simeq \tors\L(\p), \qquad \torf(\Fl\L_\p) \simeq \torf\L(\p).
\]
Therefore in results below, we can describe the statements by using both of $\tors(\Fl\L_\p)$ and $\tors\L(\p)$ (respectively, $\torf(\Fl\L_\p)$ and $\torf\L(\p)$).
We use the former one.
\begin{thm}[Theorem \ref{thm-torsion-free-iso}]\label{intro-thm-embedding}
For a Noetherian algebra $(R, \L)$, there is a commutative diagram
\[
\begin{tikzcd}
\tors\L \arrow[rr, hookrightarrow, "\Phi_{\rm t}"] \arrow[d, hookrightarrow, "(-)^{\perp}"] & & \bT_R(\L)  :=\prod_{\p\in\Spec R}\tors(\Fl \L_\p) \arrow[d, "(-)^{\perp}"]\\
\torf\L \arrow[rr, "\Phi_{\rm f}", "\simeq"'] & & \bF_R(\L):=\prod_{\p\in\Spec R}\torf(\Fl \L_\p),
\end{tikzcd}
\]
where $\Phi_{\rm t}$ and the left $(-)^{\perp}$ are anti-embeddings of posets, $\Phi_{\rm f}$ and the right $(-)^{\perp}$ are anti-isomorphisms of posets.
Both of $\Phi_{\rm t}$ and $\Phi_{\rm f}$ are given by $\cC \mapsto (\cC_\p\cap\Fl\L_\p)_{\p\in\Spec R}$, where $\cC_\p:=\{X_\p \mid X\in\cC\}$.
The map $(-)^{\perp}$ means taking perpendicular categories with respect to homomorphisms.
\end{thm}
We apply Theorem \ref{intro-thm-embedding} to a classification of Serre subcategories of $\mod\L$.
We denote by $\simple\L_\p$ the set of isomorphism classes of simple $\L_\p$-modules and regard
\[\Simple_R\L:=\bigsqcup_{\p\in\Spec R}\simple\L_\p\]
as a poset, where for $S\in\simple\L_\p$ and $T\in\simple\L_\q$, we write $S\leq T$ if and only if $\p \supseteq \q$ and $S$ is a subfactor of $T$ as a $\L_\p$-module.
For a set $X$, we denote by $\sP(X)$ the poset of all subsets of $X$ ordered by inclusion. When $X$ is a poset, we denote by $\sP_{\rm down}(X)$ the subposet of $\sP(X)$ of all down-sets of $X$.
\begin{cor}[Theorem \ref{thm-serre-poset-bijection}]\label{intro-cor-serre}
For a Noetherian algebra $(R, \L)$, there is an isomorphism of posets
\[
\serre \L \simeq \sP_{\rm down}(\Simple_R\L)\ \mbox{ given by }\ \cS \mapsto \bigsqcup_{\p\in\Spec R}(\cS_\p \cap \simple \L_\p).
\]
\end{cor}
This result also follows from Kanda's theory of atom spectrum of an abelian category, see Remark \ref{rem-serre-kanda}.

As an application of Theorem \ref{intro-thm-embedding} and Corollary \ref{intro-cor-serre}, we prove the following result, where we denote by $\sP_{\rm spcl}(\Spec R)$ the subposet of $\sP(\Spec R)$ of all specialization closed subsets of $\Spec R$.
\begin{cor}[Corollaries \ref{cor-local-free-subset}, \ref{cor-local-sp-closed}, Proposition \ref{prop-tors-serre}]\label{intro-cor-Gabriel}
For a Noetherian algebra $(R, \L)$, assume that $\L_\p$ is Morita equivalent to a local ring for each $\p\in\Spec R$.
Then
\begin{enumerate}[{\rm (a)}]
\item $\tors\L=\serre \L$ holds, and there is a canonical isomorphism $\tors\L\simeq\sP_{\rm spcl}(\Spec R)$ of posets.
\item There is a canonical isomorphism $\torf\L\simeq\sP(\Spec R)$ of posets.
\end{enumerate}
\end{cor}
Applying Corollary \ref{intro-cor-Gabriel} to the special case $\Lambda=R$, we recover the famous classification results of Serre subcategories by Gabriel \cite{Gabriel}, and its analogues for torsion classes\footnote{In \cite{Hovey} and \cite{GP}, the authors called Serre subcategories closed under coproducts  torsion classes, and studied them in $\Mod R$. 
In our language, they are {\it hereditary} torsion classes (also known as \emph{localizing subcategories}), thus they did not study torsion classes in our sense.}
by Stanley and Wang \cite{Stanley-Wang} (see also \cite[Proposition 2.5]{APST}), and for torsionfree classes by Takahashi \cite{Takahashi}.
See also \cite{Enomoto, IMST, Saito} for more recent results.

For a general $\L$, the structure of $\tors\L$ is much richer and hence harder to understand.
For example, when $\L$ is the path algebra $kQ$ of a Dynkin quiver $Q$ over a field $k$, $\tors kQ$ parametrizes the clusters in the corresponding cluster algebra.
The second aim of this paper is to characterize the elements of $\bT_R(\L)$ corresponding to elements in $\tors\L$.
For each pair $\p\supseteq\q$ in $\Spec R$ and $\cX^\p\in\tors(\Fl\L_\p)$, we consider
\[
\psi_\p(\cX^\p):=\{X\in\mod\L_\p \mid X/\p X\in \cX^\p\} \in \tors\L_\p
\]
and the map
\[
{\rm r}_{\p, \q} : \tors(\Fl\L_\p) \longrightarrow \tors(\Fl\L_\q), \qquad \cX^\p \mapsto \psi_\p(\cX^\p)_\q\cap\Fl\L_\q.
\]
An element $(\cX^\p)_{\p\in\Spec R}$ in $\bT_R(\L)=\prod_{\p\in\Spec R}\tors(\Fl\L_\p)$ is called \emph{compatible} if ${\rm r}_{\p, \q}(\cX^\p)\supseteq \cX^\q$ holds for each pair $\p\supseteq \q$ in $\Spec R$.
We denote by $\bT_R^{\rm c}(\L)$ the subposet of $\bT_R(\L)$ consisting of all compatible elements in $\bT_R(\L)$, that is,
\[
\bT_R^{\rm c}(\L)=\left\{(\cX^\p)_\p\in \prod_{\p\in\Spec R}\tors(\Fl\L_\p) \ \middle| \ \forall\p \supseteq \q \in \Spec R,\ {\rm r}_{\p, \q}(\cX^\p) \supseteq \cX^\q\right\}.
\]
\begin{prop}[Proposition \ref{prop-wcomp-comp-easy}]\label{intro-prop-comp}
For a Noetherian algebra $(R, \L)$, we have $\Im \Phi_{\rm t}\subseteq\bT_R^{\rm c}(\L)$. Thus we have an embedding of posets
\begin{align}\label{intro-map-comp}
	\Phi_{\rm t} : \tors\L \hookrightarrow \bT_R^{\rm c}(\L).
\end{align}
\end{prop}
We call $(R, \L)$ \emph{compatible} if $\Im \Phi_{\rm t}=\bT_R^{\rm c}(\L)$, that is, we have an isomorphism of posets
\[
\Phi_{\rm t} :\tors\L \simeq \bT_R^{\rm c}(\L),
\]
which translates the classification of torsion classes of Noetherian algebras $\L$ into the classification of torsion classes of finite dimensional algebras $\L(\p)$.
We give sufficient conditions for $(R, \L)$ to be compatible.
\begin{thm}[Corollary \ref{cor-semi-local}]\label{intro-thm-semilocal}
Let $(R, \L)$ be a Noetherian algebra.
If $R$ is semi-local and $\dim R=1$, then $(R, \L)$ is compatible.
Thus we have an isomorphism $\Phi_{\rm t} : \tors\L \simeq \bT_R^{\rm c}(\L)$ of posets.
\end{thm}
We give other sufficient conditions of quite different type in terms of silting theory.
In Appendix \ref{subsection-silting-ch}, we give characterizations of silting modules (that is, support $\tau$-tilting pairs in \cite{Iyama-Jorgensen-Yang}, and finitely presented silting modules in \cite{AMV}).
We denote by $\twosilt\L$ the set of additive equivalence classes of two-term silting complexes of $\L$, and by $\siltm\L$ the set of additive equivalence classes of silting $\L$-modules (Definition \ref{dfn-two-term-silting}).
Each $M\in\siltm\L$ gives a torsion class $\Fac M\in\tors\L$, and let
\begin{align}\label{intro-stors-storsfl}\notag
	\stors(\Fl\L)&:=\{\Fac M \cap \Fl\L \mid M \in \siltm\L\} \subseteq \tors(\Fl\L)
\end{align}
and give the following criterion for compatibility.
\begin{thm}[Theorem \ref{thm-r-isom}]\label{thm-intro-r-isom}
Let $(R, \L)$ be a Noetherian algebra such that $R$ is ring indecomposable and the following conditions are satisfied.
\begin{enumerate}[\rm(i)]
	\item $(-)_\p : \twosilt\L \to \twosilt \L_\p$ is an isomorphism of posets for any $\p\in\Spec R$.
	\item $\stors(\Fl\L_\p)=\tors(\Fl\L_\p)$ holds for any $\p$.
\end{enumerate}
Then $(R, \L)$ is compatible and we have an isomorphism of posets
\[
\tors\L \simeq \Hom_{\rm poset}(\Spec R, \twosilt\L).
\]
\end{thm}
Notice that in view of Theorem \ref{thm-intro-r-isom}, it is important to understand when the condition $\stors(\Fl\L_\p)=\tors(\Fl\L_\p)$ is satisfied.
The following result shows that this condition is a generalization of $\tau$-tilting finiteness \cite{DIJ} for finite dimensional algebras over a field.
\begin{thm}[Theorem \ref{thm-silting-finite}]\label{intro-thm-silting-finite}
Let $(R, \L)$ a Noetherian algebra such that $\L$ is semi-perfect (e.g.\ $R$ is complete local).
Then the following conditions are equivalent.
\begin{enumerate}[{\rm (i)}]
	\item $\sharp\twosilt\L < \infty$.
	\item $\sharp\tors(\Fl\L) < \infty$.
	\item $\stors(\Fl\L)=\tors(\Fl\L)$.
\end{enumerate}
\end{thm}
As a corollary of Theorems \ref{thm-intro-r-isom} and \ref{intro-thm-silting-finite} we have the following one.
\begin{cor}[Corollary \ref{cor-r-isom}]\label{intro-cor-r-isom}
Let $k$ be a field and let $A$ be a finite dimensional $k$-algebra.
Assume that $A$ is $\tau$-tilting finite and any simple $A$-module is $k$-simple.
Then for arbitrary commutative Noetherian ring $R$ which contains the field $k$, we have an isomorphism of posets
\[
\tors(R\otimes_k A) \simeq \Hom_{\rm poset}(\Spec R, \tors A).
\]
\end{cor}
We have the following example by using Corollary \ref{intro-cor-r-isom}.
\begin{exa}[Example \ref{example-Dynkin}]\label{intro-example-Dynkin}
Let $R$ be a commutative Noetherian ring which contains a field $k$ and $Q$ a Dynkin quiver.
Since $\tors kQ$ is isomorphic to $\mathfrak{C}_Q$, there is an isomorphism of posets
\[
\tors RQ \simeq \Hom_{\rm poset}(\Spec R, \mathfrak{C}_Q),
\]
where $\mathfrak{C}_Q$ is the Cambrian lattice of $Q$.
\end{exa}
\begin{rem}
Our Example \ref{intro-example-Dynkin} is analogous to a result of \cite[Corollary 5.11]{AS}.
For a commutative Noetherian ring $R$ and a Dynkin quiver $Q$, they constructed an isomorphism of posets between the set of thick subcategories of $\sK^{\rm b}(\proj RQ)$ and $\Hom_{\rm poset}(\Spec R, \mathrm{NC}(Q))$, where $\mathrm{NC}(Q)$ is the poset of non-crossing partitions of $Q$.
Since there is an order preserving bijection $\mathrm{NC}(Q) \to \mathfrak{C}_Q$ which does not reflect the order, we obtain an order preserving map from the set of thick subcategories of $\sK^{\rm b}(\proj RQ)$ to $\tors RQ$, which does not reflect the order.
\end{rem}
We pose the following question.
\begin{ques}\label{intro-ques-comp}
Characterize Noetherian algebras which are compatible.
\end{ques}
Theorems \ref{intro-thm-semilocal} and \ref{thm-intro-r-isom} give partial answers to this question.
So far we do not know any Noetherian algebra which is \emph{not} compatible.

We end the introduction with the following remark about a connection between our results and classifications of torsion classes of $\Mod\L$.
\begin{rem}
Let $\L$ be a Noetherian $R$-algebra. A torsion class $\cT$ of $\Mod\L$ is said to be \emph{finitely generated} if there exists a subcategory $\cC$ of $\mod\L$ such that $\cT=\Gen \cC$ holds.
This is equivalent to that there exists a cosilting module $U$ such that $\cT={}^{\perp}U$ \cite{AH-com, BL}.
By \cite[(4.4) Lemma]{CB-locally}, it is easy to see that there exists a bijection
\[\lim:\tors\L\simeq\{\mbox{finitely generated torsion classes of $\Mod\L$}\}:(-)\cap(\mod\L).\]
Therefore our result gives a classification of finitely generated torsion classes of $\Mod\L$. Moreover the bijection above clearly restricts to a bijection
\[\lim:\serre\L\simeq\{\mbox{finitely generated Serre subcategories of $\Mod\L$}\}:(-)\cap(\mod\L).\]
Note that not all torsion classes are finitely generated, e.g.\ the category of injective $\bZ$-modules.
\end{rem}
\subsection{Notations and conventions}\label{subsection-notation}
In this paper, all categories are assumed to be essentially small and all subcategories are assumed to be full and closed under isomorphisms. The composition of morphisms $f:X\to Y$ and $g:Y\to Z$ is denoted by $gf:X\to Z$.

Let $A$ be a ring.
We denote by $\Mod A$ (respectively, $\mod A$, $\fp A$, $\proj A$, $\Fl A$) the category of (respectively, finitely generated, finitely presented, finitely generated projective, finite length) left $A$-modules and denote by $\sK^{\rm b}(\proj A)$ the bounded homotopy category of $\proj A$.
Let $\sD(\L)=\sD(\Mod\L)$.

Let $\cA$ be an additive category.
For a subcategory  $\cB$ (or a collection of objects) of $\cA$, we denote by $\add\cB$ the full subcategory of $\cA$ consisting of direct summands of finite direct sums of objects in $\cB$.
Let $\cB^{\perp_{\cA}}=\{X\in\cA \mid \Hom_{\cA}(B,X)=0,\,{}^{\forall}B\in\cB\}$.
Similarly, ${}^{\perp_\cA}\cB$ is defined.
If the category $\cA$ is clear from the context, then we use the notion $\cB^{\perp}$ and ${}^{\perp}\cB$ for simplicity.
Two objects $X, Y$ in an additive category are said to be \emph{additive equivalent} if $\add X=\add Y$ holds.

Let $\cA$ be an abelian category and $\cB$ a subcategory (or a collection of objects) of $\cA$.
We denote by $\Fac\cB$ the full subcategory of $\cA$ consisting of factor objects of objects in $\add\cB$.
We denote by $\Filt\cB$ the subcategory of $\cA$ consisting objects $X$ such that there exists a finite filtration $0\subseteq X_1 \subseteq\dots\subseteq X_{\ell}=X$ and $X_i/X_{i-1}$ is an object of $\add\cB$ for each $i$.

Let $\cC$ be an additive category and $\cB$ be a subcategory of $\cC$.
For an object $C$ of $\cC$, a \emph{right $\cB$-approximation of $C$} is a morphism $f:B_0 \to C$ with $B_0\in\cB$ such that the map $\Hom_{\cC}(B,f)$ is surjective for any object $B$ of $\cB$,
and a \emph{left $\cB$-approximation of $C$} is a morphism $g:C \to B^0$ with $B^0\in\cB$ such that the map $\Hom_{\cC}(g, B)$ is surjective for any object $B$ of $\cB$.

Let $\cA$ be an abelian category and $\cC$ a subcategory of $\cA$.
We say that $\cC$ is \emph{closed under extensions} in $\cA$ if for each short exact sequence $0 \to X \to Y \to Z \to 0$ in $\cA$ with $X, Z\in\cC$, we have $Y\in\cC$.
We say that $\cC$ is \emph{closed under factor objects (respectively, subobjects)} in $\cA$ if for each objects $C$ of $\cC$, any factor objects of $C$ in $\cA$ (respectively, any subobjects of $C$ in $\cA$) belongs to $\cC$.

Let $R$ be a commutative Noetherian ring.
An $R$-algebra $\L$ is a ring $\L$ with a ring homomorphism $\phi : R \to \L$ such that the image of $\phi$ is contained in the center of $\L$.
A \emph{Noetherian $R$-algebra} (or, \emph{module-finite $R$-algebra}, \emph{Noetherian algebra}) is an $R$-algebra $\L$ which is finitely generated as an $R$-module.
For a prime ideal $\p$ of $R$, let $\kappa(\p)=R_\p/\p R_\p$.

A ring $\L$ is said to be \emph{semi-perfect} if it admits a decomposition $\L=P_1\oplus\dots\oplus P_n$ as a left $\L$-module such that each $P_i$ has a local endomorphism ring.
It is well-known that $\L$ is semi-perfect if and only if $\proj \L$ is a Krull-Schmidt category.
For more details of semi-perfect rings, see \cite{Curtis-Reiner, Krause}.

For a posets $X, Y$, we denote by $\Hom_{\rm poset}(X, Y)$ the set of all morphisms of posets, that is, a map $f : X \to Y$ such that $a\leq b$ in $X$ always implies $f(a)\leq f(b)$.
A morphism $f : X \to Y$ of posets is called an \emph{embedding of posets} if $f$ is injective and a partial order on $X$ coincides with the restriction of a partial order on $Y$.
Namely, $f$ is an embedding of posets if and only if $f(a) \leq f(b)$ implies $a\leq b$ for any $a, b\in X$.

The \emph{Hasse quiver} of a poset $X$ has vertices corresponding to elements of $X$, and we draw an arrow $a \to b$ between two elements $a,b$ if $a> b$ and there are no intermediate element between $a$ and $b$.

For a background of Noetherian algebras we refer \cite{Curtis-Reiner}.
For a background of commutative ring theory we refer \cite{Eisenbud, Matsumura}.
\section{Noetherian algebras and silting complexes}
\subsection{Noetherian algebras}
Let $R$ be a commutative Noetherian ring and $\L$ a Noetherian $R$-algebra.
Let $S$ be a multiplicative subset of $R$. We denote by $R_S$ and $\L_S:=R_S\otimes_R\L$ the localizations of $R$ and $\L$ by $S$, where $\L_S$ is a Noetherian $R_S$-algebra. For a $\L$-module $M$, 
we have an exact functor
\[
(-)_S = R_S\otimes_{R}(-) : \mod\L \to \mod\L_S.
\]
Since the functor is exact, we have $\Im f_S=(\Im f)_S$, $\Cok f_S=(\Cok f)_S$ and $\Ker f_S=(\Ker f)_S$ for a morphism $f$ in $\mod\L$.
Let $M, N\in\mod\L$ and $n\geq 0$ an integer. By \cite[(8.18) Corollary]{Curtis-Reiner}, the functor $(-)_S$ induces an isomorphism
\begin{align}\label{iso-localization}
\Ext_{\L}^{n}(M, N)_S \simeq \Ext_{\L_S}^{n}(M_S, N_S).
\end{align}
For $n=0$, it sends an element $1_{R_S}\otimes f$ in the left term to a morphism $[M_S\ni 1_{R_S}\otimes m \mapsto 1_{R_S}\otimes f(m)\in N_S]$. 
For $n=1$, it sends an element $1_{R_S}\otimes g$ in the left term corresponding to a short exact sequence $g=[0\to N\xto{a} L\xto{b} M\to0]$ to its localization $[0\to N_S\xto{a_S} L_S\xto{b_S} M_S\to0]$.

We denote by $\Spec R$ (respectively, $\MSpec R$) the set of all prime (respectively, maximal) ideals of $R$.
For an ideal $I$ of $R$, we denote by $V(I)$ the set of all prime ideals of $R$ containing $I$.
The set $\Spec R$ is a poset by inclusion, that is, $\q \leq \p$ if $\q \subseteq \p$.
If $S=R \setminus \p$ for $\p\in\Spec R$, then we write 
\[\L_{\mfp}:=\L_S=R_{\mfp}\otimes_{R}\L,\]
which is a Noetherian $R_\p$-algebra. For $X\in\mod\L$, let $\add X_\p:=\add(X_\p)$. Then $\add X_\p\supseteq(\add X)_\p$ holds, and the equality does not hold in general (e.g. $X:=\L$ in Example \ref{exa-L-local}).

For a $\L$-module $M$, we denote by $\supp_R M$ the support of $M$ regarded as an $R$-module, that is,
\begin{align*}
\supp_R M=\{\mfp\in\Spec R \mid M_{\mfp}\neq 0\}.
\end{align*}
We also recall that an \emph{associated prime ideal} of $M\in\mod R$ is a prime ideal $\mfp$ such that $\p=\ann(x):=\{a \in R \mid ax=0\}$ holds for some $x\in M$.
Thus $\p$ is an associated prime ideal of $M$ if and only if there exists an injective morphism from $R/\mfp$ to $M$.
We denote by $\ass_R M$ the set of all associated prime ideals of $M$.
It is known that $\ass_R M \subseteq \supp_R M$ holds, and the minimal elements of $\ass_R M$ coincide with the minimal elements of $\supp_R M$, see \cite{Matsumura} for instance.

We collect basic facts in commutative ring theory, which will be used frequently without reference.
\begin{prop}\label{prop-comm-basic}
Let $R$ be a commutative Noetherian ring, and $M\in\mod R$.
\begin{enumerate}[\rm(a)]
\item $\supp_RM=V(R/\ann_RM)$ holds.
\item The set of minimal elements of $\supp_RM$ and of $\ass_RM$ coincide.
\item $M=0$ if and only if $\supp_RM=\emptyset$ if and only if $\ass_RM=\emptyset$.
\item $M$ has finite length if and only if $\supp_RM\subseteq\MSpec R$.
\item Any $M\in\mod R$ has a finite filtration such that each subfactor module is isomorphic to $R/\p$ for some $\p\in\Spec R$.
\item $\supp_{R_\p}M_\p=\supp_RM\cap\Spec R_\p$ and $\ass_{R_\p}M_\p=\ass_RM\cap\Spec R_\p$.
\end{enumerate}
\end{prop}
In the rest, let $(R,\L)$ be a Noetherian algebra.
It is clear that a $\L$-module $X$ is finitely generated if and only if it is finitely generated as an $R$-module. We observe some properties of simple and length finite $\L$-modules.
\begin{lem}\label{lem-finite-length}
Let $M\in\mod\L$. The following statements hold.
\begin{enumerate}[{\rm (a)}]
\item
$M$ has finite length as a $\L$-module if and only if	$M$ has finite length as an $R$-module.
\item
If $(R, \m)$ is a local ring and $M$ has finite length as a $\L$-module, then there exists a positive integer $\ell$ such that $\mfm^{\ell}M=0$.
\item $\Fl\L$ is a product of $\Fl\L_\m$ for all $\m\in\MSpec R$.
\end{enumerate}
\end{lem}
\begin{proof}
(a)
We show that a simple $\L$-module $M$ has finite length as an $R$-module.
The center $Z$ of the division ring $\End_{\L}(M)$ is a field containing $\overline{R}:=R/\ann_RM$. Since $Z$ is a finitely generated $\overline{R}$-module, $\overline{R}$ is also a field by \cite[\S9, Lemma 1]{Matsumura}, and $\ann_{R}M$ is a maximal ideal of $R$. Thus $M\in\mod\overline{R}$, and hence $M$ has finite length as an $R$-module.

(b)
By (a), $M$ has finite length as an $R$-module.
In particular, if $(R, \m)$ is local, then there exists a positive integer $\ell$ such that $\mfm^{\ell}M=0$ by Nakayama's lemma.
\end{proof}
\begin{lem}\label{lem-local-global}
Let $S$ be a multiplicative subset of $R$, and $X\in\mod\L_S$.
\begin{enumerate}[\rm(a)]
\item There exists a $\L$-submodule $M\in\mod\L$ of $X$ such that $M_S\simeq X$. Moreover, for a $\L_S$-submodule $Y$ of $X$, there exists a $\L$-submodule $N$ of $M$ such that $M_S\simeq X$ induces $N_S\simeq Y$.
\item Assume $S=R\setminus\p$ for $\p\in\Spec R$. If $X$ has finite length, then $M$ in \emph{(a)} satisfies $\supp_R M \subseteq V(\p)$ and $\ass_R M \subseteq \{\p\}$.
\end{enumerate}
\end{lem}
\begin{proof}
(a) The assertions follow from \cite[(23.13), (23.12) Proposition]{Curtis-Reiner}.

(b) For any $\q\in\ass_R M$, take $x\in M$ such that $R/\q\simeq Rx$. Since $x$ is non-zero in $M_\p$, we have $(R/\q)_\p\neq0$ and hence $\p\supseteq \q$.
Since $(R/\q)_\p\subseteq M_\p\simeq X$ has finite length as an $R_\p$-module, $\p=\q$ holds. Thus $\ass_R M\subseteq\{\p\}$ and hence $\supp_RM\subseteq V(\p)$.
\end{proof}
\begin{lem}\label{lem-local-gen}
Let $M\in\mod\L$. The following statements hold.
\begin{enumerate}[{\rm (a)}]
\item $(\gen M)_\p=\gen(M_\p)$ holds for any $\p\in\Spec R$. Thus we denote it by $\gen M_\p$ for simplicity.
\item For $X\in\mod\L$, $X\in \gen M$ if and only if $X_\m\in\gen M_\m$ for any maximal ideal $\m$ of $R$.
\end{enumerate}
\end{lem}
\begin{proof}
(a) It is clear that $(\gen M)_\p \subseteq \gen(M_\p)$ holds.
The converse holds by Lemma \ref{lem-local-global}(a), since each submodule of $M_\p^{\oplus n}$ is isomorphic to $N_\p$ for some submodule $N$ of $M^{\oplus n}$.

(b) It suffices to show ``if'' part.
Let $f : M' \to X$ be a right $(\add M)$-approximation of $X$, that is, 
$\Hom_\L(M, f)$ is surjective.
For each maximal ideal $\m$ of $R$, 
$\Hom_{\L_\m}(M_\m, f_\m) \simeq \Hom_\L(M, f)_\m$ holds by \eqref{iso-localization}, and hence $\Hom_{\L_\m}(M_\m, f_\m)$ is surjective.
So $f_\m$ is a right $(\add M_\m)$-approximation.
Since $X_\m\in\gen M_\m$, $f_m$ is surjective and $(\Cok f)_\m=\Cok(f_\m)=0$.
Thus $\Cok f=0$ and $f$ is surjective.
\end{proof}
\begin{lem}\label{rxAx}
Let $X\in\mod\L$. For each $x\in X$, we have $\supp_R(Rx)=\supp_R(\L x)$.
\end{lem}
\begin{proof}
Since $Rx\subseteq\L x$, it suffices to prove ``$\supseteq$''.
Take a surjection $R^{\oplus\ell}\to\L$ of $R$-modules. Applying $-\otimes_RRx$, we have a surjection $(Rx)^{\oplus\ell}\to\L\otimes_RRx$. Composing with the surjection  $\L\otimes_RRx\to\Lambda x$ given by $\lambda\otimes rx\mapsto \lambda rx$, we have a surjection $(Rx)^{\oplus\ell}\to\L x$ of $R$-modules. Thus the assertion follows.
\end{proof}
\begin{prop}\label{prop-sp-local}
Let $(R, \L)$ be a Noetherian algebra such that $\L$ is a faithful $R$-module.
\begin{enumerate}[\rm(a)]
\item We have a surjection $\{\mbox{maximal left ideals of $\L$}\}\to\MSpec R$ given by $I\mapsto I\cap R$. 
\item If $\L$ is local, then $R$ is local
\item If $\L$ is semi-perfect, then $R$ is a finite product of local rings.
\end{enumerate}
\end{prop}
\begin{proof}
(a) We show that the map is well-defined.
For a maximal left ideal $I$ of $\L$, take $\m\in\MSpec R$ such that $\m\in\supp_R(\L/I)$. Then $I\subseteq \m\L+I\subseteq\L$. If the right equality holds, then $\m\L_\m+I_\m=\L_\m$ implies $I_\m=\L_\m$ by Nakayama's Lemma, a contradiction. Thus the maximality of $I$ implies $\m\L+I=I$ and $\m\L\subseteq I$. Hence $\m\subseteq I\cap R$ and hence $\m=I\cap R$ holds.

We show that the map is surjective.
For $\m\in\MSpec R$, we have $\m\L\subsetneq\L$. Otherwise, we have $\m\L_\m=\L_\m$ and hence $\L_\m=0$ by Nakayama's Lemma. This contradicts to the faithfulness of the $R$-module $\L$. Take a maximal right ideal $I$ of $\L$ containing $\m\L$. Then $\m\subseteq\m\L\cap R\subseteq I\cap R\subsetneq R$ implies $\m=I\cap R$.

(b) Immediate from (a).

(c) 
For each primitive idempotent $e\in\L$, since $e\L e$ is local, there exists unique $\m\in\MSpec R$ containing $\ann_R(e\L e)$ by (c). 
Thus there exist pairwise distinct maximal ideals $\m_1,\ldots,\m_n$ of $R$ and primitive orthogonal idempotents $1_\L=\sum_{i=1}^n\sum_{j=1}^{\ell_i}e_{ij}$ such that $\m_i$ is a unique maximal ideal of $R$ containing $K_{ij}:=\ann_R(e_{ij}\L e_{ij})$.
Then $\m_i$ is a unique maximal ideal of $R$ containing $K_i:=\bigcap_{j=1}^{\ell_i}K_{ij}$. We have $K_i+K_j=R$ for each $1\le i\neq j\le n$ since $\m_i\neq\m_j$.
Moreover $\bigcap_{i=1}^nK_i=0$ holds since $\L$ is a faithful $R$-module. By \cite[Theorem 1.4]{Matsumura}, we have $R\simeq\prod_{i=1}^n(R/K_i)$, as desired.
\end{proof}
The following is a basic lemma.
\begin{lem}\label{lem-Spec-conn}
Let $R$ be a commutative Noetherian ring which is ring indecomposable.
\begin{enumerate}[\rm(a)]
\item\cite[Exercise 2.25]{Eisenbud} $\Spec R$ is connected in the Zariski topology.
\item Let $\emptyset \neq \cS\subseteq\Spec R$.
If $\cS$ and $\cS^{\rm c}$ are specialization closed, then $\cS = \Spec R$.
\end{enumerate}
\end{lem}
\begin{proof}
(b) Let $\{\p_1, \dots, \p_m\}$ (respectively, $\{\q_1,\dots,\q_n\}$) be the minimal prime ideals such that $\p_i \in\cS$ (respectively, $\q_i\in\cS^{\rm c}$).
Then $\cS=\bigcup_{i=1}^m V(\p_i)$ and $\cS^{\rm c}=\bigcup_{i=1}^n V(\p_i)$ hold.
Thus the assertion holds since $\Spec R$ is connected by (a).
\end{proof}
The following fact is well-known.
\begin{prop}\label{Spec R finite}
$\Spec R$ is a finite set if and only if $\dim R\le 1$ and $R$ is semi-local.
\end{prop}
\begin{proof}
The ``if'' part is clear since $\dim R=1$ implies that each prime ideal of $R$ is either minimal or maximal, and there are only finitely many minimal prime ideals in general \cite[Theorem 6.5]{Matsumura}. 
The ``only if part'' follows from Krull's principal ideal theorem, see \cite[Theorem 144]{Kaplansky}.
\end{proof}
\subsection{Silting complexes}
In this subsection, we recall the definition of silting complexes and modules, and observe when base changes preserve silting complexes.
A \emph{thick closure} $\thick_{\cT} P$ of an object $P$ in a triangulated category $\cT$ is the smallest triangulated subcategory of $\cT$ which is closed under direct summands and contains $P$.
For a ring $\L$, we have $\sK^{\rm b}(\proj\L)=\thick_{\sD(\L)}\L$. Thus for $\cT:=\sK^{\rm b}(\proj\L)$ and $P\in\cT$, $\cT=\thick_{\cT}P$ holds if and only if $\L\in\thick_{\cT}P$ holds.
\begin{dfn}\label{dfn-presilting}
Let $P$ be an object in a triangulated category $\cT$.
\begin{enumerate}[{\rm (1)}]
\item We say that $P$ is \emph{presilting} if $\Hom_{\cT}(P, P[i])=0$ holds for any $i>0$.
\item We say that $P$ is \emph{silting} if $P$ is presilting and satisfies $\thick P =\cT$.
\end{enumerate}
\end{dfn}
\begin{dfn}\label{dfn-two-term-silting}
Let $\L$ be a ring.
\begin{enumerate}[{\rm (1)}]
	\item An object in $\sK^{\rm b}(\proj\L)$ is said to be \emph{two-term} if it is isomorphic to an object $P=(P^i,d^i)$ 
in $\sK^{\rm b}(\proj\L)$ such that $P^i=0$ for $i\neq -1, 0$.
	\item A (pre)silting complex of $\sK^{\rm b}(\proj\L)$ which is two-term is called a \emph{two-term (pre)silting complex} of $\L$.
	We denote by $\twosilt \L$ (respectively, $\twopsilt\L$) the set of additive equivalence classes of two-term silting (respectively, presilting) complexes of $\L$.
	\item A $\L$-module $M$ is said to be \emph{silting module} (respectively, \emph{presilting module}) if there is a two-term silting (respectively, presilting) complex $P$ of $\L$ such that $H^0(P)=M$.
	We denote by $\siltm \L$ the set of additive equivalence classes of silting $\L$-modules.
\end{enumerate}
\end{dfn}
For $P, Q\in\twosilt\L$, we write $Q \leq P$ if $\Hom_{\sD(\L)}(P, Q[i])=0$ for any $i >0$.
Then $(\twosilt \L, \leq)$ is a poset by \cite{AI}.
For $M, N\in\siltm\L$, we write $M\leq N$ if $\gen M \subseteq \gen N$. Then $(\siltm \L, \leq)$ is a poset which is isomorphic to $(\twosilt \L, \leq)$, see Proposition \ref{prop-ap-twosilt-silt-stors}.

We first see that taking factor ring preserves two-term silting complexes.
For each factor ring $\ov{\L}$ of $\L$, $\ov{\L}\otimes_{\L}(-)$ induces a triangle functor
\begin{align}\label{tri-func-L-ovL}
\ov{(-)}=\ov{\L} \otimes_{\L}(-) :\sK^{\rm b}(\proj\L) \to\sK^{\rm b}(\proj\ov{\L}).
\end{align}
\begin{lem}\label{lem-twosilt-fac}
Let $\L$ be a ring, $\ov{\L}$ a factor ring of $\L$ and $P, Q$ two-term complexes in $\sK^{\rm b}(\proj\L)$.
\begin{itemize}
\item[{\rm (a)}]
There exists a canonical surjective map $\Hom_{\sD(\L)}(P, Q[1]) \to \Hom_{\sD(\ov{\L})}(\ov{P}, \ov{Q}[1])$.
\item[{\rm (b)}]
The functor $\ov{(-)}$ give maps, where the latter is a morphism of posets:
\[\twopsilt\L \to \twopsilt\ov{\L}, \qquad  \twosilt\L \to \twosilt\ov{\L}.\]
\end{itemize}
\end{lem}
\begin{proof}
(a)
Let $P=(P^{-1} \xto{d_P} P^0)$ and $Q=(Q^{-1} \xto{d_Q} Q^0)$.
We have the following commutative diagram
\[
\begin{tikzcd}
\Hom_{\L}(P^{-1}, Q^{-1})\oplus\Hom_{\L}(P^0, Q^0) \arrow[r, "\alpha"] \arrow[d] & \Hom_{\L}(P^{-1}, Q^0) \arrow[d] \arrow[r] & \Hom_{\sD(\L)}(P, Q[1]) \arrow[d] \arrow[r] & 0 \\
\Hom_{\ov{\L}}(\ov{P^{-1}}, \ov{Q^{-1}})\oplus\Hom_{\ov{\L}}(\ov{P^0}, \ov{Q^0}) \arrow[r, "\ov{\alpha}"] & \Hom_{\ov{\L}}(\ov{P^{-1}}, \ov{Q^0}) \arrow[r] & \Hom_{\sD(\ov{\L})}(\ov{P}, \ov{Q}[1]) \arrow[r] & 0,
\end{tikzcd}
\]
where each horizontal sequence is exact, vertical maps are given by the functor $\ov{(-)}$, and $\alpha(f, g)=d_Q\circ f + g \circ d_P$ (similar for $\ov{\alpha}$).
Since $P^{-1}$ is a projective $\L$-module, the middle vertical map is surjective.
Thus the right vertical map is surjective.

(b)
By (a) the functor $\ov{(-)}$ gives a map from $\twopsilt\L$ to $\twopsilt\ov{\L}$.
If $P$ is a silting complex, then since (\ref{tri-func-L-ovL}) is a triangle functor, $\L\in\thick P$ implies $\ov{\L} \in \thick \ov{P}$.
Thus $\ov{(-)}$ induces a map from $\twosilt\L$ to $\twosilt\ov{\L}$.
\end{proof}
In general $\ov{(-)}$ does not induce a map between presilting complexes with length more than two as the following example shows.
\begin{exa}
Let $k$ be a field and $A=kQ$ the path algebra of a quiver $Q=(1\to 2)$.
Let $I$ be the ideal of $A$ generated by the arrow of $Q$ and $B=A/I=k\times k$.
We denote by $P(i)$ and $S(i)$ a projective (respectively, simple) $A$-module associated to a vertex $i=1,2$.
Let $X=(P(2) \to P(1))$ be a minimal projective resolution of a simple $A$-module $S(1)$ and we regard $X$ as a two-term complex.
Then we have a silting complex $P=X\oplus (S(2)[2])$ which is concentrated in degree $0$, $-1$ and $-2$.
For $i=1,2$, $\ov{P(i)}=S_B(i)$ holds, where $S_B(i)$ is the $i$-th simple $B$-module.
We have $\ov{P}=S(2)[1] \oplus S(1) \oplus (S(2)[2])$, which is not presilting.
\end{exa}
Let $(R, \L)$ be a Noetherian algebra and $S$ a commutative Noetherian ring with a morphism of rings $R\to S$.
We have a Noetherian algebra $(S, S\otimes_R \L)$ and have a triangle functor
\[S\otimes_R (-) : \sK^{\rm b}(\proj\L) \to \sK^{\rm b}(\proj S\otimes_R \L).\]

First we give one basic lemma.
\begin{lem}\label{lem-base-hom}
Assume that $\Tor_i^R(S, \L)=0$ for any $i>0$.
For $P, Q\in\sK^{\rm b}(\proj\L)$, 
\begin{enumerate}[\rm(a)]
\item If $\Hom_{\sD(\L)}(P, Q[i])=0$ for any $i>0$, then we have an isomorphism
\[
\Hom_{\sD(S\otimes_R \L)}(S\otimes_R P, S\otimes_R Q) \simeq S\otimes_R \Hom_{\sD(\L)}(P, Q).
\]
In particular,  we have $\Hom_{\sD(S\otimes_R \L)}(S\otimes_R P, S\otimes_R Q[i])=0$ for any $i>0$.
\item Assume that each $0\neq X\in\mod R$ satisfies $S\otimes_RX\neq0$. If $\Hom_{\sD(S\otimes_R \L)}(S\otimes_R P, S\otimes_R Q[i])=0$ for any $i>0$, then $\Hom_{\sD(\L)}(P, Q[i])=0$ for any $i>0$.
\end{enumerate}
\end{lem}
\begin{proof}
(a) Let $\mathcal{H}om_{\L}(P, Q)$ be the complex of morphisms from $P$ to $Q$.
We have the following isomorphism of complexes of $S$-modules
\[
\mathcal{H}om_{S\otimes_R \L}(S\otimes_R P, S\otimes_R Q) \simeq S\otimes_R\mathcal{H}om_{\L}(P, Q).
\]
Because $H^i(\mathcal{H}om_{\L}(P, Q))=0$ and $\Tor_i^R(S, \L)=0$ hold for any $i>0$, $Z^0(S\otimes_R\mathcal{H}om_{\L}(P, Q))=S\otimes_R Z^0(\mathcal{H}om_{\L}(P, Q))$ holds.
Thus we have
\[
H^0(S\otimes_R\mathcal{H}om_{\L}(P, Q)) \simeq S\otimes_R H^0(\mathcal{H}om_{\L}(P, Q)) \simeq S\otimes_R\Hom_{\sD(\L)}(P, Q).
\]
Therefore we have the desired isomorphism.
By applying the isomorphism for $P$ and $Q[j]$, we have the latter assertion.

(b) Since $P$ is a bounded complex, $I:=\{i\in\bZ\mid\Hom_{\sD(\L)}(P, Q[i])\neq0\}$ is a finite set. It suffices to show that $\ell:=\max I$ satisfies $\ell\le0$. Assume $\ell>0$. Applying (a) to $(P,Q):=(P,Q[\ell])$, we obtain
\[0=\Hom_{\sD(S\otimes_R \L)}(S\otimes_R P, S\otimes_R Q[\ell]) \simeq S\otimes_R \Hom_{\sD(\L)}(P, Q[\ell]).\]
By our assumption on $S$, we have $\Hom_{\sD(\L)}(P, Q[\ell])=0$, a contradiction to $\ell\in I$.
\end{proof}
Now we are ready to prove the following result.
\begin{thm}\label{prop-base-silt}
Let $(R,\L)$ be a Noetherian algebra such that $\Tor_i^R(S, \L)=0$ for any $i>0$.
\begin{enumerate}[\rm(a)]
\item The functor $S\otimes_R(-)$ gives a map $\psilt\L \to \psilt(S\otimes_R \L)$ and a morphism of posets $\silt\L \to \silt(S\otimes_R \L)$ which is restricted to $\twosilt\L \to \twosilt(S\otimes_R \L)$.
\item Assume that each $0\neq X\in\mod R$ satisfies $S\otimes_RX\neq0$.
\begin{enumerate}[\rm(i)]
\item $S\otimes_R(-):\silt\L \to \silt(S\otimes_R \L)$ is an embedding of posets.
\item Let $P\in\sK^{\rm b}(\proj\L)$. If $S\otimes_RP\in \psilt(S\otimes_R\L)$, then $P\in\psilt\L$.
\end{enumerate}
\end{enumerate}
\end{thm}
\begin{proof}
(a) If $P$ is presilting, then $S\otimes_R P$ is presilting by Lemma \ref{lem-base-hom}(a).
Assume that $P$ is silting.
Since $S\otimes_R(-)$ is a triangle functor, $\L \in\thick P$ implies $S\otimes_R\L\in\thick (S\otimes_R P)$. So $S\otimes_R P$ is silting.
Clearly $S\otimes_R(-)$ preserves two-term complexes. So the last assertion holds.

(b) Both assertions are immediate from Lemma \ref{lem-base-hom}(b).
\end{proof}
The following special case will be used later.
Note that without any assumption, for an arbitrary factor algebra $\ov{\L}$ of $\L$, we have a morphism of posets $\twosilt\L \to \twosilt\ov{\L}$ by Lemma \ref{lem-twosilt-fac}.
\begin{prop}\label{prop-quo-silt}
Let $(R,\L)$ be a Noetherian algebra such that $\L$ is projective as an $R$-module. For an ideal $I$ of $R$ contained in the Jacobson radical of $R$, let $\ov{(-)}:=(R/I)\otimes_R(-):\sK^{\rm b}(\proj\L)\to\sK^{\rm b}(\proj\ov{\L})$.
\begin{enumerate}[{\rm (a)}]
\item The functor $\ov{(-)}$ gives a map $\psilt\L\to\psilt\ov{\L}$ and an embedding of posets $\silt\L\to\silt\ov{\L}$ which restricted to $\twosilt\L \to \twosilt\ov{\L}$.
\item Let $P\in\sK^{\rm b}(\proj\L)$. If $\ov{P}\in\psilt\ov{\L}$, then $P\in\psilt\L$.
\end{enumerate}
\end{prop}
\begin{proof}
Since $\L$ is projective as an $R$-module, $\Tor_i^R(R/I, \L)=0$ holds for any $i>0$. By Nakayama's Lemma, each $0\neq X\in\mod R$ satisfies $(R/I)\otimes_RX\neq0$.
Thus each assertion follows from Theorem \ref{prop-base-silt}.
\end{proof}
Next we observe some properties of silting complexes with respect to localization functors.
We denote by $\sD^{\rm b}(\mod\L)$ the bounded derived category of $\mod\L$. It is elementary that $\Hom_{\sD(\L)}(X,Y)$ is a finitely generated $R$-module for each $X,Y\in\sD^{\rm b}(\mod\L)$ (e.g.\ \cite[Lemma 2.3]{Kimura}).

Let $S$ be a multiplicative subset of $R$.
Since $(-)_S$ is exact, it induces a triangle functor
\[
(-)_S : \sD^{\rm b}(\mod \L) \to \sD^{\rm b}(\mod\L_S)
\]
and this restricts to homotopy categories: $\sK^{\rm b}(\proj \L) \to \sK^{\rm b}(\proj\L_S)$.
Let $X, Y\in\sD^{\rm b}(\mod\L)$, $i$ an integer and put $\sD(\L)=\sD(\Mod\L)$.
We have
\begin{align}\label{iso-localization-derived}
\Hom_{\sD(\L)}(X, Y[i])_S \simeq \Hom_{\sD(\L_S)}(X_S, Y_S[i]),
\end{align}
since $\Hom_{\sD(\L)}(X, Y[i])= H^i(\mathcal{H}{\rm om}_{\L}(X, Y))$, $(-)_S$ is exact and (\ref{iso-localization}).
\begin{prop}\label{prop-presilt-local}
Let $(R, \L)$ be a Noetherian algebra. The following statements hold.
\begin{enumerate}[{\rm (a)}]
\item For any multiplicative subset $S$ of $R$, the functor $(-)_S$ gives a map $\psilt\L\to\psilt\L_S$ and a morphism of posets $\silt\L\to\silt\L_S$ which restricted to $\twosilt\L \to \twosilt\L_S$.
\item 
\begin{enumerate}[\rm(i)]
\item We have an embedding of posets \[\silt\L \to \prod_{\p\in\Spec R}\silt\L_\p, \quad P \mapsto (P_\p)_{\p\in\Spec R}.\]
\item Let $P\in\sK^{\rm b}(\proj\L)$. If $P_\p \in\psilt\L_\p$ for any $\p\in\Spec R$, then $P\in\psilt\L$.
\end{enumerate}
\end{enumerate}
\end{prop}
We prove the silting version of (b)(ii) in Theorem \ref{thm-silting-local}.
\begin{proof}
(a) Since $R_S$ is a flat $R$-module, $\Tor_i^R(R_S, \L)=0$ for any $i>0$. Thus the assertion follows from Theorem \ref{prop-base-silt}(a).

(b) Both assertions follow from the following claim: For $P, Q\in\sK^{\rm b}(\proj\L)$, $\Hom_{\sD(\L)}(P, Q[i])=0$ holds for any $i>0$ if and only if $\Hom_{\sD(\L_\p)}(P_\p, Q_\p[i])=0$ holds for any $i>0$ and any $\p\in\Spec R$.
This follows from Proposition \ref{prop-comm-basic}(c) and \eqref{iso-localization-derived} since $\Hom_{\sD(\L)}(P, Q[i])$ is a finitely generated $R$-module.
\end{proof}
As an application of the results above, we obtain the following observation.
\begin{prop}
Let $(R, \L)$ be a Noetherian algebra such that $\L$ is projective as an $R$-module.
\begin{enumerate}[{\rm (a)}]
\item For any $\p\in\Spec R$, the functor $\kappa(\p)\otimes_R(-)$ gives a map $\psilt\L\to\psilt\L(\p)$ and a morphism of posets $\silt\L\to\silt\L(\p)$ which restricted to $\twosilt\L \to \twosilt\L(\p)$.
\item
\begin{enumerate}[\rm(i)]
\item We have an embedding of posets
\[\silt\L \to \prod_{\p\in\Spec R}\silt\L(\p), \quad P \mapsto (\kappa(\p)\otimes_R P)_{\p\in\Spec R}.\]
\item Let $P\in\sK^{\rm b}(\proj\L)$. If $\kappa(\p)\otimes_RP \in\psilt\L(\p)$ for any $\p\in\Spec R$, then $P\in\psilt\L$.
\end{enumerate}
\end{enumerate}
\end{prop}
\begin{proof}
All assertions follow from Propositions \ref{prop-quo-silt} and \ref{prop-presilt-local}.
\end{proof}
We give one proposition giving a Noetherian algebra with trivial silting complexes.
\begin{prop}\label{prop-comm-silt}
Let $(R, \L)$ be a Noetherian algebra.
Assume that $R$ is ring indecomposable and that $\L_\p$ is Morita equivalent to a local ring for any prime ideal $\p$ of $R$.
Then any silting complex $T$ satisfies $\add T=\add \L[i]$ in $\sK^{\rm b}(\proj \L)$ for some integer $i$.
\end{prop}
\begin{proof}
For each $\p$, $T_\p$ is a silting complex by Proposition \ref{prop-presilt-local}.
Since $\L_\p$ is Morita equivalent to a local ring, there is an integer $i_\p$ such that $\add T_\p = \add \L_\p[i_\p]$ holds \cite[Theorem 2.26]{AI}.
For two prime ideals $\p\subseteq\q$ of $R$, we have $i_\p = i_\q$.
This implies that $i_\p$ is constant for any $\p$ since $\Spec R$ is connected by Lemma \ref{lem-Spec-conn}(a).
Namely, there is an integer $i$ such that $i=i_\p$ for any $\p$.
Therefore if $j\neq -i$, then $H^{j}(T)=0$ and so $T\simeq H^{-i}(T)[i]$ holds.
We have $\add T = \add H^{-i}(T)[i] = \add \L[i]$.
\end{proof}

\subsection{Bass-type Theorem on silting/perfect complexes}
The reader can skip this subsection since the main result will not be used in other parts of this paper. We will prove the result below, where the part (a) for the case $\L=R$ is a generalization of a result due to Bass, see Corollary \ref{bass result}.
\begin{thm}\label{thm-silting-local}
Let $(R, \L)$ be a Noetherian algebra.
For a complex $P\in\sD^{\rm b}(\mod\L)$, the following statements hold.
\begin{enumerate}[{\rm (a)}]
	\item $P\in\sK^{\rm b}(\proj\L)$ if and only if $P_\p\in\sK^{\rm b}(\proj\L_\p)$ for any $\p\in\Spec R$.
	\item $P\in\silt\L$ if and only if $P_\p\in\silt\L_\p$ for any $\p\in\Spec R$.
\end{enumerate}
\end{thm}

To show the theorem, we recall the following notion.

Let $\cC$ be an additive category and $\cB$ be a subcategory of $\cC$.
We say that $\cB$ is \emph{contravariantly finite} in $\cC$ if every object of $\cC$ admits a right $\cB$-approximation.
A \emph{covariantly finite} subcategory is defined dually.
A subcategory of $\cC$ is said to be \emph{functorially finite} if it is both contravariantly and covariantly finite in $\cC$.
For instance, if $\cC=\mod\L, \sK^{\rm b}(\proj \L)$ or $\sD^{\rm b}(\mod\L)$, then $\add C$ is a functorially finite subcategory of $\cC$ for any object $C$ \cite[Lemmas 2.2,  2.3]{Kimura}.

For two subcategories (or collections of objects) $\cX, \cY$ of a triangulated category $\cT$, we denote by $\cX\ast\cY$ a subcategory of $\cT$ consisting of objects $T$ admitting a triangle $X \to T \to Y \to X[1]$ with $X\in\cX$, $Y\in\cY$.
The operator $\ast$ is associative by the octahedral axiom: $(\cX \ast \cY) \ast \cZ = \cX \ast (\cY \ast \cZ)$.
\begin{assum}\label{ass-silt}
Let $P$ be a presilting object in a triangulated category $\cT$ such that $\add P$ is a contravariantly finite subcategory of $\cT$.
For $X=X_0\in\cT$ such that $\Hom_{\cT}(P[i], X)=0$ for $i<0$, define $X_{n+1}$ ($n \geq 0$) by taking a right $(\add P[n])$-approximation $f_n : P_n[n] \to X_n$ and $X_{n+1}:=\mathrm{cone}(f_n)$.
\end{assum}
\begin{lem}\label{lem-silting-app}
For objects $P$ and $X_n$ in $\cT$ as Assumption \ref{ass-silt}, the following statements hold.
\begin{enumerate}[{\rm (a)}]
	\item If $X_n\in\add P[n]$, then $X_{n+1}\in\add P[n+1]$ holds.
	\item $X$ belongs to $\thick P$ if and only if $X_n\in\add P[n]$ holds for some integer $n$.
\end{enumerate}
\end{lem}
\begin{proof}
(a)
If $X_n$ belongs $\add P[n]$, then $f_n$ is a split epimorphism.
Thus $\mathrm{cone}(f_n)\in \add P_n[n+1]$.

(b)
The "if" part is clear since $X \in P_0 \ast P_1[1] \ast \dots \ast P_{n-1}[n-1] \ast X_n$ holds.
Assume that $X$ belongs to $\thick P$.
Since $P$ is presilting, $(\add P) \ast (\add P[1]) \ast \dots \ast (\add P[n])$ is closed under direct summands \cite{Iyama-Yang}.
Thus there is an integer $n$ and $Q_i\in\add P$ such that $X\in Q_0 \ast Q_1[1] \ast \dots \ast Q_n[n]$ \cite[Lemma 2.15]{AI}.
Then there is $Y\in Q_1[1] \ast \dots \ast Q_n[n]$ with the following commutative diagram
\[
\begin{tikzcd}
Q_0 \arrow[r] \arrow[d] & X \arrow[r] \arrow[d, equal] & Y \arrow[r] \arrow[d] & Q_0[1] \arrow[d] \\
P_0 \arrow[r] & X \arrow[r] & X_1 \arrow[r] & P_0[1]
\end{tikzcd}
\]
where each horizontal sequence is a triangle.
The right side square gives a triangle $Y \to X_1\oplus Q_0[1] \to P_0[1] \to Y[1]$ \cite[proof of Lemma 1.4.3]{N}, which splits by $\Hom_{\cT}(P_0, Y)=0$.
Therefore we have $X_1\in\add(Y\oplus P_0[1]) \subseteq (\add P[1]) \ast \dots \ast (\add P[n])$.
Repeat the same argument, we have $X_n\in\add P[n]$.
\end{proof}
\begin{lem}\label{lem-thick-local}
Let $P \in \psilt\L$ and $X\in\sD^{\rm b}(\mod\L)$.
Then $X \in \thick P$ if and only if $X_\p\in\thick P_\p$ for any $\p\in\Spec R$.
\end{lem}
\begin{proof}
It is suffices to show "if" part.
Let $\cT=\sD^{\rm b}(\mod\L)$.
Without loss of generality, we may assume that $\Hom_{\cT}(P[i], X)=0$ for $i<0$.
Then $P, X$ and $\cT$ fit into Assumption \ref{ass-silt}.
Let
\[
I_n=\{\p\in\Spec R \mid (X_n)_\p \notin \add P_\p[n] \}.
\]
Consider the morphism of $R$-modules $f_n^{\ast} : \Hom_{\cT}(X_n, P_n[n]) \to \End_{\cT}(X_n)$ and let $C_n=\Cok(f_n^{\ast})$.
Since $X_n$ belongs to $\sD^{\rm b}(\mod\L)$, $\End_{\cT}(X_n)$ and $C_n$ belong to $\mod R$.
Then $I_n=\supp_R C_n$ holds.
In fact for a prime ideal $\p$, $(f_n)_\p$ is a right $(\add P_\p[n])$-approximation of $(X_n)_\p$. Thus $(C_n)_\p=0$ if and only if $(f_n)_\p$ is a split epimorphism if and only if $(X_n)_\p\in\add P_\p[n]$ holds.
Therefore $I_n$ is a Zariski closed subset of $\Spec R$.
By applying Lemma \ref{lem-silting-app}(b) to $P_\p$ and $X_\p$, we have $\bigcap_{n\geq 0}I_n = \emptyset$.
Lemma \ref{lem-silting-app}(a) implies that $I_{n+1} \subseteq I_n$ for any $n\geq 0$.
Since $\Spec R$ is a Noetherian space, there is an integer $n$ such that $I_n=\emptyset$.
Thus $C_n=0$ and $f_n$ is a split epimorphism.
We have $X_n\in\add P[n]$ and again by Lemma \ref{lem-silting-app}(b), the assertion holds.
\end{proof}
\begin{proof}[Proof of Theorem \ref{thm-silting-local}]
(a)
By applying Lemma \ref{lem-thick-local} to a silting complex $\L\in\sK^{\rm b}(\proj\L)$ and a complex $P\in\sD^{\rm b}(\mod\L)$, we have that $P\in\thick \L$ if and only if $P_\p\in\thick\L_\p$ for any prime ideal $\p$ of $R$.

(b)
Let $P\in\sK^{\rm b}(\proj\L)$.
If $P$ is silting, then by Proposition \ref{prop-presilt-local}(a), $P_\p$ is silting for any $\p\in\Spec R$.
Assume that $P_\p$ is silting for any $\p\in\Spec R$.
So $\L_\p\in\thick P_\p$ holds for any $\p\in\Spec R$.
By Proposition \ref{prop-presilt-local}(b), $P$ is presilting.
By applying Lemma \ref{lem-thick-local} to $P$ and a complex $\L\in\sD^{\rm b}(\mod\L)$, we have $\L\in\thick P$.
Thus $P$ is silting.
\end{proof}
As a corollary, we get the following result \cite[Theorem 2.16]{Iyama-Wemyss-sing} which is a non-commutative analogue of a result due to Bass.
\begin{cor}\label{bass result}
For $X\in\mod\L$, $X$ has a finite projective dimension over $\L$ if and only if $X_\p$ has a finite projective dimension over $\L_\p$ for any prime ideal $\p$ of $R$.
\end{cor}
\begin{proof}
The assertion is immediate from Theorem \ref{thm-silting-local}(a).
\end{proof}
We use the following lemma in Subsection \ref{subsection-preliminaries-subcategory}.
\begin{lem}\label{lem-derived-ext-local}
For subcategories $\cX, \cY$ of $\sD^{\rm b}(\mod\L)$ and a multiplicative subset $S$ of $R$, we have $(\cX \ast \cY)_S = \cX_S \ast \cY_S$.
\end{lem}
\begin{proof}
It suffices to show that $(\cX \ast \cY)_S \supseteq \cX_S \ast \cY_S$ holds.
Let $X_S \to E \to Y_S \xto{g} X_S[1]$ be a triangle in $\sD^{\rm b}(\mod\L_S)$.
By the isomorphism (\ref{iso-localization-derived}), there exist $f\in\Hom_{\sD^{\rm b}(\mod\L)}(Y, X[1])$ and $s\in S$ such that $s^{-1}f=g$.
Then we have a triangle $X \to C \to Y \xto{f} X[1]$ with $C\in\cX\ast\cY$, and $s$ induces an isomorphism $E \simeq C_S \in (\cX \ast \cY)_S$.
\end{proof}
\subsection{A Theorem on field extensions and (pre)silting complexes}
In the rest of this subsection, we give some properties of two-term presilting complexes of finite dimensional algebras with respect to a field extension. The reader can skip this subsection until the main Theorem \ref{thm-K-k-presilt} will be used in the proof of Proposition \ref{prop-ftors-bij}.

Let $k$ be a field, $K$ a field extension of $k$ and $A$ a finite dimensional $k$-algebra.
For an $A$-module $M$, let $M_K=K\otimes_k M$.
We have a finite dimensional $K$-algebra $A_K=K\otimes_k A$.
We denote by \emph{$\Aut_k(K)$} the group of $k$-algebra automorphisms of $A$.
Any $g\in\Aut_k(K)$ gives an automorphism $g\otimes\id_A$ of $A_K$.
Then for an $A_K$-module $X$, we denote by ${}_gX$ an $A_K$-module with an action $b \cdot_g x = (g\otimes\id_A)(b)x$ for $b \in A_K$ and $x\in X$.

For a complex $X$ of $A$-modules, we have a complex $X_K$ of $A_K$-modules by applying the tensor functor $K\otimes_k(-)$.
For two complexes $P, Q$ in $\sK^{\rm b}(\proj A)$, we have $\Hom_{\sD(A_K)}(P_K, Q_K)=K\otimes_k\Hom_{\sD(A)}(P, Q)$.
Thus $K\otimes_k (-)$ induces a map from the set of presilting complexes over $A$ to the set of presilting complexes over $A_K$.
Moreover this map restricts to silting complexes.
We denote by $\sK_0(\proj A)$ the Grothendieck group of $\proj A$.
For a two-term complex $P=(P^{-1} \to P^0)$ in $\sK^{\rm b}(\proj A)$, we call $g(P)=[P^0]-[P^{-1}]\in\sK_0(\proj A)$ the \emph{g-vector} of $P$.

We first observe the following lemma.
\begin{lem}\label{lem-K-k-inj}
Let $K/k$ be an extension of fields and $A$ a finite dimensional $k$-algebra.
\begin{enumerate}[{\rm (a)}]
	\item The morphism $K\otimes_k(-) : \sK_0(\proj A) \to \sK_0(\proj A_K)$ is injective.
	\item For complexes $X, Y\in\sD^{\rm b}(\mod A)$, if $\add X \cap \add Y =0$, then $\add X_K \cap \add Y_K =0$ holds.
\end{enumerate}
\end{lem}
\begin{proof}
The statement (a) directly follows from (b).
Thus we show (b).
All morphisms in this proof are taken in the bounded derived category.
We denote by $[Y]$ the ideal of $\End_{\sD(A)}(X)$ consisting of morphisms factoring through an object in $\add Y$. Similarly we define the ideal $[X]$ of $\End_{\sD(A)}(Y)$. Then $\add X \cap \add Y =0$ holds if and only if $[Y]\subseteq\rad\End_{\sD(A)}(X)$ and $[X]\subseteq\rad\End_{\sD(A)}(Y)$ hold if and only if the ideals $[Y]$ and $[X]$ are nilpotent. In this case, the ideals $[Y_K]$ of $\End_{\sD(A_K)}(X_K)$ and $[X_K]$ of $\End_{\sD(A_K)}(Y_K)$ are nilpotent, and hence $\add X_K \cap \add Y_K =0$ holds, as desired.
\end{proof}
For projective $A$-modules $P, Q$, we say that \emph{$\Hom_A(P, Q)$ contains a presilting complex} if there is a morphism $f : P\to Q$ such that the two-term complex $(f : P \to Q)$ is a presilting complex.

The following is the main result in this subsection.
\begin{thm}\label{thm-K-k-presilt}
Let $A$ be a finite dimensional $k$-algebra, $K/k$ be a field extension, and $P, Q\in\proj A$.
Then $\Hom_A(P, Q)$ contains a presilting complex if and only if so does $\Hom_{A_K}(P_K, Q_K)$.
\end{thm}
The natural morphism $A \to A_K$ induces a functor $(-)|_A : \mod A_K \to \mod A$.
To prove the theorem, we remark the following lemma, we refer \cite[Lemma 2.5]{Li} for instance.
\begin{prop}\label{prop-Galois-ext}
Let $K/k$ be a finite Galois extension and $A$ a finite dimensional $k$-algebra.
\begin{enumerate}[{\rm (a)}]
	\item There exists a functorial isomorphism $(K \otimes_k X)|_A \simeq X^{\oplus [K : k]}$ for each $X\in\mod A$.
	\item There exists a functorial isomorphism $K \otimes_k (Y|_A) \simeq \bigoplus_{\sigma\in\Gal(K/k)}{}_{\sigma}Y$ for each $Y\in\mod A_K$.
\end{enumerate}
\end{prop}
\begin{proof}[Proof of Theorem \ref{thm-K-k-presilt}]
The ``only if" part is clear.
We show the ``if" part, that is, if there exists a presilting complex $(f : P_K \to Q_K)$, then there exists a presilting complex $(g : P \to Q)$.

Let $\tilde{f} : \End_{A_K}(P_K)\oplus\End_{A_K}(Q_K) \to \Hom_{A_K}(P_K, Q_K)$ be a map defined as $\tilde{f}(\alpha, \beta)=\beta\circ f + f\circ \alpha$.
Since $(f :P_K\to Q_K)$ is presilting, $\tilde{f}$ is surjective.
Since $\Hom_{A_K}(P_K, Q_K)=K\otimes_k\Hom_A(P, Q)$ holds, there are finitely many $g_i \in \Hom_A(P, Q)$ and $\la_i\in K$ such that $f=\sum \la_i\otimes g_i$.
We have $\tilde{f}=\sum \la_i \otimes \tilde{g_i}$, where $\tilde{g_i}$ is a morphism $\End_{A}(P)\oplus\End_{A}(Q) \to \Hom_{A}(P,Q)$ defined in the same way as $\tilde{f}$.

By fixing $k$-bases, we denote by $G_i$ a matrix with coefficients in $k$ which represents $\tilde{g_i}$.
Then $\tilde{f}$ is represented by a matrix $\sum \la_i\otimes G_i$.
Since $\tilde{f}$ is surjective, there is a maximal minor $\sM$ with $\sM(\sum \la_i\otimes G_i)\neq 0$.
Then we have a non-zero polynomial $\sM(\sum t_i G_i)$ in variables $t_i$ with coefficients in $k$.

We divide into two cases (i) $k$ is an infinite field, (ii) $k$ is a finite field.

(i) Assume that $k$ is an infinite field.
Then there are $\mu_i \in k$ such that $\sM(\sum \mu_i G_i)\neq 0$.
Let $g=\sum \mu_i  g_i$.
Then a two-term complex $(g : P \to Q)$ is presilting, since $\tilde{g}$ is surjective.

(ii) Assume that $k$ is a finite field.
Let $k'$ be an algebraic closure of $k$ and $\mu_i \in k'$ such that $\sM(\sum \mu_i G_i)\neq 0$.
Let $k''$ be a field generated by all $\mu_i$ over $k$, which is a finite extension of $k$.
Then we have a two-term presilting complex $(\sum \mu_i g_i : P_{k''} \to Q_{k''})$.
By replacing $K$ to $k''$, we may assume that $K|k$ is a finite extension.

Assume that $K/k$ is a finite extension.
Since $k$ is a finite field, this extension is a finite Galois extension.
Let $G=\Gal(K/k)$ be the Galois group of $K/k$ and $S:=(f : P_K \to Q_K)$.
We have a two-term presilting complex ${}_{\sigma}S=({}_{\sigma}f : {}_{\sigma}P_K \to {}_{\sigma}Q_K)$ for each $\sigma\in G$.
For each idempotent $e\in A$, $\sigma(1_K \otimes e)=1_K \otimes e$ holds.
Therefore by calculating the $g$-vectors, $S \simeq {}_{\sigma}S$ in $\sK^{\rm b}(\proj A)$ holds for each $\sigma\in G$.
Let $h:=f|_A$.
By Proposition \ref{prop-Galois-ext}(a), we have a two-term complex $T=S|_A=(h : P^{\oplus d} \to Q^{\oplus d})$ of $A$, where $d=[K : k]$.
By Proposition \ref{prop-Galois-ext}(b), we have $T_K \simeq S^{\oplus d}$ in $\sK^{\rm b}(\proj A_K)$.
This implies that $T$ is a presilting complex.
In fact, consider the map $\tilde{h} : \End_{A}(P^{\oplus d})\oplus\End_{A}(Q^{\oplus d}) \to \Hom_{A}(P^{\oplus d}, Q^{\oplus d})$.
Since $T_K$ is a presilting complex, the cokernel of $\tilde{h_K}$ is zero.
Therefore the cokernel of $\tilde{h}$ is also zero.
Namely, $T$ is a presilting complex.
By taking a completion \cite[Proposition 2.16]{Aihara}, there is a basic two-term silting complex $U=\bigoplus_{i=1}^n U_i$ of $A$ such that $d([Q]-[P])=[T] \in \sum_{i=1}^n\bZ_{\geq 0}[U_i]$.
Since $\{[U_i] \mid i=1,\dots,n\}$ is a $\bZ$-basis of $\sK_0(\proj A) \simeq \bZ^{\oplus n}$ \cite[Theorem 2.27]{AI}, $[Q]-[P]\in\sum_{i=1}^n\bZ_{\geq 0}[U_i]$ holds.
Therefore there exists a morphism $g : P \to Q$ such that $(g : P\to Q)$ is a presilting complex of $A$.
\end{proof}
For a finite dimensional algebra $A$, we denote by $\twopsilt_{\oplus} A$ the set of isomorphic classes of all (not necessarily basic) presilting complexes of $A$.
\begin{thm}\label{thm-twosilt-k-K}
Let $K/k$ be an extension of fields and $A$ a finite dimensional $k$-algebra.
Assume that $K\otimes_k(-) : \proj A \to \proj A_K$ preserves indecomposability.
Then $K\otimes_k(-)$ induces the following bijections and an isomorphism.
\begin{enumerate}[{\rm (a)}]
	\item A bijection from $\twopsilt_{\oplus} A$ to $\twopsilt_{\oplus} A_K$, and a bijection from $\twopsilt A$ to $\twopsilt A_K$.
	\item An isomorphism of posets from $\twosilt A$ to $\twosilt A_K$.
\end{enumerate}
The bijections above preserve the number of non-isomorphism indecomposable direct summands.
\end{thm}
\begin{proof}
For an indecomposable decomposition $A=\bigoplus_{i=1}^{\ell} P_i$ of $A$ as a left module, $A_K=\bigoplus_{i=1}^{\ell} (P_i)_K$ is an indecomposable decomposition of $A_K$ by assumption.
Thus $K\otimes_k(-)$ induces a surjective map $\sK_0(\proj A) \to \sK_0(\proj A_K)$.
By Lemma \ref{lem-K-k-inj}, this is injective.

(a)
Since there is a bijection from $\tau$-rigid pairs of $A$-modules to two-term presilting complexes of $A$ \cite[Section 3]{Adachi-Iyama-Reiten}, the map $g_A:\twopsilt_{\oplus} A\to\sK_0(\proj A)$ (respectively $g_{A_K}:\twopsilt_{\oplus} A_K\to\sK_0(\proj A_K)$) taking the $g$-vectors is an injection by \cite[Theorem 6.5]{DIJ}.
We have the following commutative diagram
\begin{align*}
\begin{tikzpicture}[baseline=30]
\node(twosiltL)at(0,1.5){$\twopsilt_{\oplus} A$};
\node(twosiltLI)at(4,1.5){$\sK_0(\proj A)$};
\node(sttiltL)at(0,0){$\twopsilt_{\oplus} A_K$};
\node(sttiltLI)at(4,0){$\sK_0(\proj A_K)$};
\draw[thick, ->] (sttiltL)--(sttiltLI) node[midway, above]{$g_{A_K}$};
\draw[thick, ->] (twosiltL)--(twosiltLI) node[midway, above]{$g_A$};
\draw[thick, ->] (twosiltL)--(sttiltL) node[midway, left]{$K\otimes_k(-)$};
\draw[thick, ->] (twosiltLI)--(sttiltLI) node[midway, right]{$K\otimes_k(-)$};
\end{tikzpicture}
\end{align*}
where the right vertical map is a bijection.
Therefore, the left vertical map is injective.
By Theorem \ref{thm-K-k-presilt}, the left vertical map is surjective.
Thus $K\otimes_k(-)$ induces a bijection from $\twopsilt_{\oplus} A$ to $\twopsilt_{\oplus} A_K$.
This bijection induces a bijection from $\twopsilt A$ to $\twopsilt A_K$.

Let $X, Y\in\twopsilt_{\oplus} A$.
By Lemma \ref{lem-K-k-inj}(b), if $X$ and $Y$ are not isomorphic, then $X_K$ and $Y_K$ are not isomorphic.
We show that $X$ is indecomposable if and only if so is $X_K$.
If $X_K$ is decomposable, then by (a), there exist non-zero $Y,Z\in\twopsilt_{\oplus}A$ such that $X_K \simeq Y_K\otimes Z_K$.
Again by (a) we have $X \simeq Y\oplus Z$, that is, $X$ decomposable.
Conversely, if $X_K$ is indecomposable, then clearly $X$ is indecomposable.
Therefore the bijections preserve the number of non-isomorphism indecomposable direct summands.

(b)
It is well-known that $X\in\twopsilt_{\oplus} A$ is silting if and only if the number of isomorphism classes of indecomposable direct summands of $X$ is equal to the rank of $\sK_0(\proj A)$.
Since the bijections in (a) preserve the number, $K\otimes_k(-)$ induces a bijection from $\twosilt A$ to $\twosilt A_K$.
We show that this bijection is an isomorphism of posets.
Let $X, Y\in\twosilt A$.
Since $X, Y$ are two-term complexes and by Proposition \ref{prop-base-silt}, $\Hom_{\sD(A)}(X, Y[1])=0$ if and only if $\Hom_{\sD(A_K)}(X_K, Y_K[1])=0$.
Namely the bijection is an isomorphism of posets.
\end{proof}
The following gives a sufficient condition for $A$ to satisfy an assumption of Theorem \ref{thm-twosilt-k-K}.
\begin{lem}\label{lem-k-K-proj}
Let $K/k$ be an extension of fields and $A$ a finite dimensional $k$-algebra.
For an indecomposable projective $A$-module $P$ with the top $S$, assume that $S_K$ is a simple $A_K$-module.
Then $P_K$ is an indecomposable $A_K$-module.
\end{lem}
\begin{proof}
Clearly, $P_K$ is a projective $A_K$-module.
By \cite[(7.9) Theorem]{Curtis-Reiner} $(\rad A)_K=\rad A_K$ holds.
Then by applying the functor $K\otimes_k(-)$ to the short exact sequence $0 \to \rad P \to P \to P/\rad P \to 0$, we have that $P_K$ has a simple top.
Thus $P_K$ is an indecomposable projective $A_K$-module.
\end{proof}
\section{Classification of subcategories}
\subsection{Preliminaries}\label{subsection-preliminaries-subcategory}
We start with recalling the definition of torsion and torsionfree classes of abelian categories.
\begin{dfn}\label{dfn-torsion-serre}
Let $\cA$ be an abelian category and $\cC$ a subcategory of $\cA$.
\begin{enumerate}[{\rm (1)}]
\item
A pair $(\cT, \cF)$ of subcategories of $\cA$ is called a \emph{torsion pair} of $\cA$ if $\Hom_{\cA}(T, F)=0$ for any $T\in\cT$ and $F\in\cF$, and there exists a short exact sequence $0\to T \to X \to F \to 0$ with $T\in\cT$ and $F\in\cF$ for any $X\in\cA$.
\item
We say that $\cC$ is a \emph{torsion class} (respectively, \emph{torsionfree class}) of $\cA$ if $\cC$ is closed under factor objects (respectively, subobjects) and extensions.
\item
We say that $\cC$ is a \emph{Serre subcategory} of $\cA$ if it is closed under factor objects, subobjects and extensions in $\cA$.
\end{enumerate}
We denote by $\tors\cA$ (respectively, $\torf\cA$, $\serre\cA$) the set of all torsion classes (respectively, torsionfree classes, Serre subcategories) of $\cA$.
\end{dfn}
The sets $\tors\cA$, $\torf\cA$ and $\serre\cA$ are posets by inclusion, that is, $\cT\leq \cU$ if $\cT \subseteq \cU$.
For a Noetherian algebra $(R, \L)$, let $\tors\L:=\tors(\mod\L)$, $\torf\L:=\torf(\mod\L)$ and $\serre\L:=\serre(\mod\L)$ for simplicity.

For a subcategory $\cC$ of $\mod \L$, we denote by $\Filt_\L\cC$ the subcategory of $\mod\L$ consisting modules $X$ such that there exists a finite filtration $0\subseteq X_1 \subseteq\dots\subseteq X_{\ell}=X$ and $X_i/X_{i-1}\in\cC$ for each $i$.
We often write $\Filt \cC = \Filt_{\L}\cC$ if the algebra $\L$ is clear.

Then $\tors\cA$ is a complete lattice. In fact, for a family $\{\cT^i\}_{i\in I}$ of elements of $\tors\cA$, the meet $\bigwedge_{i\in I}\cT^i$ is given by the intersection $\bigcap_{i\in I}\cT^i$, and the join $\bigvee_{i\in I}\cT^i$ is given by the intersection of all $\cT\in\tors\L$ which contain all $\cT^i$. Alternatively we have $\bigvee_{i\in I}\cT^i=\Filt(\cT^i\mid i\in I)$
(e.g.\ \cite[Proposition 3.3]{DIJ}).
\begin{lem}\label{lem-exact-lift}
Let $F : \cA \to \cB$ be a functor of abelian categories.
For a subcategory $\cC\subseteq \cB$, let $F^{-1}\cC=\{X \in \cA \mid FX\in\cC\}$.
\begin{enumerate}[{\rm (a)}]
\item If $F$ is right exact and $\cC\in\tors\cB$, then $F^{-1}\cC\in\tors\cA$ holds.
\item If $F$ is left exact and $\cC\in\torf\cB$, then $F^{-1}\cC\in\torf\cA$ holds.
\item If $F$ is exact and $\cC\in\serre\cB$, then $F^{-1}\cC\in\serre\cA$ holds.
\end{enumerate}
\end{lem}
\begin{proof}
See \cite[Proposition 5.1]{DIRRT}.
\end{proof}
An object $X$ of an abelian category is called a \emph{Noetherian object} (respectively, \emph{Artinian object}) if each ascending (respectively, descending) chain of subobjects of $X$ is stationary.
A \emph{Noetherian} (respectively, \emph{Artinian})  \emph{abelian category} is an abelian category such that every object is a Noetherian (respectively, Artinian) object.
For example, $\mod\L$ is a Noetherian abelian category, and $\Fl\L$ is Noetherian as well as Artinian abelian category for a Noetherian algebra $(R, \L)$.
\begin{prop}\label{prop-tors-perp}
Let $\cA$ be a Noetherian (respectively, Artinian) abelian category and $\cC$ a subcategory of $\cA$.
Then $\cC$ is a torsion class (respectively, torsionfree class) of $\cA$ if and only if $(\cC$, $\cC^{\perp})$ (respectively, $({}^{\perp}\cC$, $\cC)$) is a torsion pair of $\cA$.
\end{prop}
\begin{proof}
The ``if" part is clear.
Conversely, assume that $\cC$ is a torsion class of a Noetherian abelian category $\cA$.
For any $X\in\cA$, let $\cS=\{\mbox{subobjects of $X$ contained in $\cC$}\}$, which is non-empty since $0\in\cS$.
Since $X$ is a Noetherian object, $\cS$ has a maximal element $T$ by inclusion.
We claim $X/T\in\cC^{\perp}$. Take a morphism $f:Y\to X/T$ with $Y\in\cC$, and take a subobject $Z$ of $X$ such that $\Im f=Z/T$. Since $\cC$ is a torsion class, we have $\Im f\in\cC$ and  $Z\in\cC$. By maximality of $T$, we have $Z=T$ and $f=0$, as desired. Thus $(\cC, \cC^{\perp})$ is a torsion pair.
By the similar argument, we can show the assertion in the case where $\cA$ is an Artinian abelian category.
\end{proof}
In the rest of this section, let $(R, \L)$ be a Noetherian algebra.
For a subcategory $\cC$ of $\mod\L$ and a multiplicative subset $S$ of $R$, we denote by $\cC_S$ a subcategory of $\mod\L_S$ defined as follows:
\[
\cC_S := \{ M_S \in\mod\L_S \mid M \in\cC \}.
\]
We give series of basic facts.
\begin{lem}\label{lem-torsion-local-lift}
If $\cC$ is a torsion class (respectively, torsionfree class, Serre subcategory) of $\mod\L_S$, then $\cC':=\{X\in\mod\L \mid X_S \in\cC\}$ is a torsion class (respectively, torsionfree class, Serre subcategory) of $\mod\L$ such that $\cC'_S=\cC$.
\end{lem}
\begin{proof}
The assertions directly follows from Lemma \ref{lem-exact-lift}.
\end{proof}
\begin{lem}\label{lem-torsion-local}
If $(\cT, \cF)$ is a torsion pair of $\mod\L$, then $(\cT_S, \cF_S)$ is a torsion pair of $\mod\L_S$.
\end{lem}
\begin{proof}
By (\ref{iso-localization}), $\Hom_{\L_S}(T_S, F_S)=0$ holds for any $T\in\cT$ and $F\in\cF$.
For any $\L$-module $X$, a short exact sequence $0 \to T \to X \to F \to 0$ with $T\in\cT$ and $F\in\cF$ induces a short exact sequence $0 \to T_S \to X_S \to F_S \to 0$.
By Lemma \ref{lem-local-global}(a), $(\cT_S, \cF_S)$ is a torsion pair of $\mod\L_S$.
\end{proof}

\begin{lem}\label{lem-subcat-local}
Let $\cC$ be a subcategory of $\mod\L$.
\begin{enumerate}[{\rm (a)}]
\item If $\cC$ is closed under extensions in $\mod\L$, then $\cC_S$ is closed under extensions in $\mod\L_S$.
\item
If $\cC$ is closed under factor modules (respectively, submodules) in $\mod\L$, then $\cC_S$ is closed under factor modules (respectively, submodules) in $\mod\L_S$.
\item
If $\cC$ is a torsion class (respectively, torsionfree class, Serre subcategory) of $\mod\L$, then $\cC_S$ is a torsion class (respectively, torsionfree class, Serre subcategory) of $\mod\L_S$.
\item $(\Filt\cC)_S=\Filt(\cC_S)$ holds.
\end{enumerate}
\end{lem}
\begin{proof}
(a)
We have $\cC_S \ast \cC_S = (\cC \ast \cC)_S = \cC_S$ by Lemma \ref{lem-derived-ext-local} and the assumption.
(b) follows from Lemma Lemma \ref{lem-local-global}(a).
(c) follows from (a) and (b).
(d) follows from Lemma \ref{lem-derived-ext-local}.
\end{proof}
By the canonical surjection $\L_\p\to\L(\p)$, we regard $\mod\L(\p)$ as a full subcategory of $\mod\L_\p$ closed under subfactor modules.
The following proposition gives a connection between torsion classes of $\Fl\L$ and torsion classes of a finite dimensional algebra.
\begin{prop}\label{prop-torsion-fl}
Let $(R, \L)$ be a Noetherian algebra and $\p\in\Spec R$.
Then the following two horizontal maps are isomorphisms of posets.
Two vertical maps are anti-isomorphisms of posets, and the diagram is commutative:
	\[
	\begin{tikzcd}
	\tors(\Fl\L_\p) \arrow[rrr, "(-)\cap\mod\L(\p)", "\simeq"'] \arrow[d, "(-)^{\perp}"] & & & \tors\L(\p) \arrow[d, "(-)^{\perp}"] \\
	\torf(\Fl\L_\p) \arrow[rrr, "(-)\cap\mod\L(\p)", "\simeq"'] & & & \torf\L(\p)
	\end{tikzcd}
	\]
\end{prop}
\begin{proof}
By replacing $(R, \L)$ to $(R_\p, \L_\p)$, we may assume that $R$ is local.
Let $\m$ be the maximal ideal of $R$.
By \cite[Theorem 5.4]{Kimura}, the upper horizontal map is an isomorphism of posets with an inverse map $\cC \mapsto \sT_{\Fl\L}(\cC)$, where $\sT_{\Fl\L}(\cC)$ the smallest torsion class of $\Fl\L$ containing $\cC$.
Clearly, both of vertical maps are anti-isomorphisms of posets.
Let $\cT\in\tors(\Fl\L)$ and $X\in\mod(\L/\m\L)$.
For $T\in\cT$, there exists an exact sequence $0 \to \Hom_\L(T/\m T, X) \to \Hom_\L(T, X) \to \Hom_\L(\m T, X)$.
Since $\m X=0$, we have that $\Hom_\L(T/\m T, X)=0$ if and only if $\Hom_\L(T, X)=0$.
Therefore the diagram is commutative.
\end{proof}

\subsection{Torsionfree classes}
Let $\bT_R(\L)$ and $\bF_R(\L)$ be the Cartesian products of $\tors(\Fl\L_\p)$ and $\torf(\Fl\L_\p)$ respectively, where $\p$ runs all prime ideals of $R$:
\[
\bT_R(\L):=\prod_{\mfp\in\Spec R}\tors(\Fl\L_\p), \qquad
\bF_R(\L):=\prod_{\mfp\in\Spec R}\torf(\Fl\L_\p).
\]
The set $\bT_R(\L)$ and $\bF_R(\L)$ are posets by inclusion, that is, $(\cX^{\mfp})_{\mfp}\leq (\cY^{\mfp})_{\mfp}$ if $\cX^{\mfp}\subseteq\cY^{\mfp}$ for any $\mfp$.
By Lemma \ref{lem-subcat-local}, the following maps are well-defined:
\begin{align*}
&\Phi_{\rm t}: \tors\L \longrightarrow \bT_R (\L), \qquad \cT \mapsto (\cT_\p \cap \Fl\L_\p)_{\p\in\Spec R}, \\
&\Phi_{\rm f}: \torf\L \longrightarrow \bF_R (\L), \qquad \cF \mapsto (\cF_\p \cap \Fl\L_\p)_{\p\in\Spec R}. \
\end{align*}

It is easy to see that these maps are morphisms of posets.

We first see that $\Phi_{\rm t}$ and $\Phi_{\rm f}$ commute with taking perpendicular categories.
By Proposition \ref{prop-tors-perp}, we have a bijection $(-)^{\perp}=(-)^{\perp_{\Fl\L_\p}}:\tors(\Fl\L_\p)\to\torf(\Fl\L_\p)$, whose inverse map is given by ${}^{\perp}(-)={}^{\perp_{\Fl\L_\p}}(-)$.
Thus we have a bijection between $\bT_R(\L)$ and $\bF_R(\L)$ indued from $(-)^{\perp}$, we also denote it by $(-)^{\perp}$:
\begin{align}\label{map-Phi}
(-)^{\perp} : \bT_R (\L) \longrightarrow \bF_R(\L),
\end{align}
whose inverse map is induced from ${}^{\perp}(-)$.
It is easy to see that this map is an anti-isomorphism of posets.
We have the following commutativity of maps.
\begin{lem}\label{lem-Phi-Phi'-comm}
We have the following commutative diagram
\[
\begin{tikzcd}
\tors\L \arrow[r, "\Phi_{\rm t}"] \arrow[d, "(-)^{\perp}"] & \bT_R(\L) \arrow[d, "(-)^{\perp}"] \\
\torf\L \arrow[r, "\Phi_{\rm f}"] & \bF_R(\L),
\end{tikzcd}
\]
where the right vertical map is an anti-isomorphism of posets and the left vertical map is an anti-embedding of posets.
\end{lem}
\begin{proof}
For $\cT\in\tors\L$ and $\p\in\Spec R$, by Proposition \ref{prop-tors-perp} and Lemma \ref{lem-torsion-local}, $(\cT_\p,\cT_\p^{\perp_{\L_\p}})$ is a torsion pair in $\mod\L_\p$ and $\cT_\p^{\perp_{\L_\p}} = (\cT^{\perp_{\L}})_\p$ holds.
Thus $(\cT_\p \cap \Fl\L_\p, \cT_\p^{\perp_{\L_\p}} \cap \Fl\L_\p)$ is a torsion pair in $\Fl\L_\p$.
So we have
\[
(\cT_\p \cap \Fl\L_\p)^{\perp_{\Fl\L_\p}}=\cT_\p^{\perp_{\L_\p}}\cap\Fl\L_\p=(\cT^{\perp_{\L}})_\p\cap\Fl\L_\p
\]
This implies that the diagram is commutative.
By Proposition \ref{prop-tors-perp}, ${}^{\perp}(-)\circ (-)^{\perp}$ is the identity map on $\tors\L$.
Thus the left vertical map is an anti-embedding of posets.
\end{proof}
In this subsection we study the set $\torf\L$ by using the map $\Phi_{\rm f} : \torf\L \to \bF_R(\L)$. Our main result is the following.
\begin{thm}\label{thm-torsion-free-iso}
Let $(R, \L)$ be a Noetherian algebra.
Then $\Phi_{\rm f}$ is an isomorphism of posets
\[
\Phi_{\rm f} : \torf\L \longrightarrow \bF_R(\L).
\]
\end{thm}
To prove this, we construct an inverse map $\Psi_{\rm f}$ of $\Phi_{\rm f}$.
\begin{dfn}
Let $\p$ be a prime ideal of $R$.
For $\cY\in\torf(\Fl\L_\p)$, let
\begin{align*}
\psi_{\rm f}(\cY):=\{X\in\mod\L \mid \ass_R X \subseteq \{\p\},\, X_\p \in \cY\}.
\end{align*}
It is easy to see that $\psi_{\rm f}(\cY)$ is a torsionfree class of $\mod\L$.
So we define a map $\Psi_{\rm f}$ as follows:
\begin{align*}
\Psi_{\rm f} : \bF_R(\L) \longrightarrow \torf\L, \qquad (\cY^\p)_{\p\in\Spec R} \mapsto \Filt\left(\psi_{\rm f}(\cY^\p) \ \middle|\ \p \in \Spec R\right).
\end{align*}
\end{dfn}
We need the following two lemmas.
\begin{lem}\label{lem-check-localization}
Let $\p$ be a prime ideal of $R$.
For $\cY\in\torf(\Fl\L_\p)$, the following statements hold.
\begin{enumerate}[{\rm (a)}]
	\item For a prime ideal $\q \neq \p$, $\psi_{\rm f}(\cY)_\q \cap \Fl\L_\q =0$ holds.
	\item $\psi_{\rm f}(\cY)_\p = \cY$ holds.
\end{enumerate}
\end{lem}
\begin{proof}
(a) For $\q\in\Spec R$ and $X\in\psi_{\rm f}(\cY)$, assume $0\neq X_\q\in\Fl\L_\q$.
Since $\ass_R X \subseteq \{\p\}$, we have $\supp_R X \subseteq V(\p)$ and hence $\p \subseteq \q$.
Since $X_\q\in\Fl\L_\q$, $\q$ is minimal in $\supp_R X$.
Thus we have $\p=\q$.

(b) Clearly, we have $\psi_{\rm f}(\cY)_\p \subseteq \cY$.
The converse inclusion directly follows from Lemma \ref{lem-local-global}.
\end{proof}
\begin{lem}\label{lem-local-injective}
For $X, Y\in\mod\L$ and a prime ideal $\p$ of $R$, if there is an injective morphism $X_\p \to Y_\p$ and $\ass_R X \subseteq \{\p\}$, then there is an injective morphism $X \to Y$.
\end{lem}
\begin{proof}
By the isomorphism (\ref{iso-localization}), any morphism from $X_\p$ to $Y_\p$ is of the form $s^{-1}f_\p$ for some $f : X \to Y$ and $s\in R\setminus \p$.
Assume that $f_\p$ is injective.
We have $(\Ker f)_\p=0$.
Since $\ass_R(\Ker f) \subseteq \ass_R X \subseteq \{\p\}$, we have $\ass_R(\Ker f)=\emptyset$.
Therefore $\Ker f=0$.
\end{proof}
We give one description of torsionfree class.
\begin{prop}\label{prop-torsion-free-ass}
For a torsionfree class $\cF$ of $\mod\L$, we have
\[
\cF = \Filt_\L\{X \in \cF \mid \sharp\ass_R X =1 \}.
\]
\end{prop}
\begin{proof}
We claim that for a non-zero $X\in\cF$, there is a non-zero $\L$-submodule $Y$ of $X$ such that $\sharp\ass_R Y =1$ and $X/Y \in \cF$.
Since $X$ is a Noetherian object, this claim induces the assertion.
Let $\p$ be a maximal element of $\ass_R X$.
There is an exact sequence $R^{\oplus \ell} \to R \to R/\p \to 0$ for some integer $\ell$.
By applying the functor $\Hom_R(-, X)$, an exact sequence
\[
0 \to Y \to X \to X^{\oplus \ell}
\]
in $\mod\L$ for $Y=\Hom_R(R/\p, X)$.
We have $\supp_R Y \subseteq V(\p)$ and $X/Y\in\cF$.
This implies that $\ass_R Y \subseteq \ass_R X \cap V(\p)=\{\p\}$, where the last equality follows from the maximality of $\p$.
Thus the claim holds.
\end{proof}
We prove Theorem \ref{thm-torsion-free-iso}.
\begin{proof}[Proof of Theorem \ref{thm-torsion-free-iso}]
We first show that $\Phi_{\rm f} \circ \Psi_{\rm f}$ is an identity map on $\bF_R(\L)$.
Let $\cY=(\cY^\p)_{\p\in\Spec R}$ be an element of $\bF_R(\L)$.
For any prime ideal $\q$ of $R$, we have
\begin{align*}
&\Psi_{\rm f}(\cY)_\q \cap \Fl\L_\q = \left(\Filt_\L(\psi_{\rm f}(\cY^\p) \mid\p \in \Spec R)\right)_\q \cap \Fl\L_\q\\
\stackrel{{\rm \ref{lem-subcat-local}(d)}}{=}&\Filt_{\L_\q}(\psi_{\rm f}(\cY^\p)_\q \mid \p \in \Spec R) \cap \Fl\L_\q
=\Filt_{\L_\q}(\psi_{\rm f}(\cY^\p)_\q\cap \Fl\L_\q \mid \p \in \Spec R) \stackrel{{\rm\ref{lem-check-localization}}}{=}\Filt_{\L_\q}(\cY^\q) \stackrel{}{=} \cY^\q,
\end{align*}
where the third equality holds since $\Fl\L_\q \in \serre\L_\q$.
Therefore $\Phi_{\rm f} \circ \Psi_{\rm f} = {\rm id}$ holds.

We show that $\Psi_{\rm f}\circ \Phi_{\rm f}={\rm id}$ holds.
For $\cF\in\torf\L$, let $\Phi_{\rm f}(\cF):=(\cY^\p)_{\p\in\Spec R}$ with $\cY^\p:=\cF_\p\cap\Fl\L_\p$ and $\cF':=\Psi_{\rm f}\circ\Phi_{\rm f}(\cF)=\Filt\left(\psi_{\rm f}(\cY^\p) \ \middle|\ \p \in \Spec R\right)$.
For each $\p\in\Spec R$ and each $X \in \psi_{\rm f}(\cY^\p)$, there is $Y\in\cF$ such that $X_\p \simeq Y_\p$.
Then by Lemma \ref{lem-local-injective}, we may assume that $X$ is a $\L$-submodule of $Y$.
Since $\cF\in\torf\L$, we have $X\in\cF$.
Therefore $\psi_{\rm f}(\cY^\p) \subseteq \cF$ and hence $\cF' \subseteq \cF$ holds.
It remains to show $\cF \subseteq \cF'$.
Clearly we have $\{X\in\cF \mid \ass_R X \subseteq \{\p\}\} \subseteq \psi_{\rm f}(\cY^\p)$.
This implies the following middle inclusion:
\[
\cF \stackrel{\ref{prop-torsion-free-ass}}{=} \Filt\{X \in \cF \mid \sharp\ass_R X =1 \} \subseteq \Filt(\psi_{\rm f}(\cY^\p) \mid \p\in\Spec R) = \cF'.
\]
We have completed the proof.
\end{proof}
In the case where $\L=R$, the following corollary recovers a result by Takahashi \cite{Takahashi}, see also \cite{Krause-thick}.
\begin{cor}\label{cor-local-free-subset}
Let $(R, \L)$ be a Noetherian algebra such that $\L_\p$ is Morita equivalent to a local ring for each $\p\in\Spec R$.
Then there is a canonical isomorphism $\torf\L\simeq\sP(\Spec R)$ of posets.
\end{cor}
\begin{proof}
By the assumption, $\torf(\Fl\L_\p)=\{0, \Fl\L_\p\}$ holds for each $\p$.
Thus we have an isomorphism $\bF_R(\L)\simeq\sP(\Spec R)$ of posets given by ${\rm S}'((\cY^\p)_{\p\in\Spec R}):=\{\p \in\Spec R \mid \cY^\p \neq 0\}$.
Now the assertion follows from Theorem \ref{thm-torsion-free-iso}.
\end{proof}
There are many Noetherian algebras satisfying the condition in Corollary 3.14.
We give an example.
\begin{exa}\label{exa-L-local}
Assume that $(R, \m)$ is local and
\[
\L=\left\{ \left( \begin{array}{cc}
a & b \\
c & d
\end{array}
\right) \in \left( \begin{array}{cc}
R & R \\
\mfm & R
\end{array}
\right)\ \middle|\ a - d \in \m \right\}.
\]
Then $\L_\m=\L$ is a local ring, and $\L_\p = \mathrm{Mat}_2(R_\p)$ is Morita equivalent to the local ring $R_\p$ for any $\p\in\Spec R\setminus\{\m\}$.
\end{exa}
In Subsection \ref{subsection-example}, we explain Theorem \ref{thm-torsion-free-iso} in a concrete example.

By Theorem \ref{thm-torsion-free-iso} and Lemma \ref{lem-Phi-Phi'-comm}, the map $\Phi_{\rm t} : \tors\L \to \bT_R(\L)$ is an embedding of posets.
In the next subsection, we study this map.
\subsection{Torsion classes}
In this subsection we study the set $\tors\L$ by using the map $\Phi_{\rm t} : \tors\L \to \bT_R(\L)$.
We construct a map in the opposite direction to $\Phi_{\rm t}$.
\begin{propdef}\label{dfn-bar}
Let $(R,\L)$ be a Noetherian algebra. For $\p\in\Spec R$ and $\cX\in\tors(\Fl\L_\p)$, let
\begin{align*}
\psi_\p(\cX)=\{X\in\mod\L_\p \mid X/\p X \in \cX\}\ \mbox{ and }\ 
\psi(\cX)=\{X\in\mod\L \mid X_\p/\p X_\p \in \cX\}.
\end{align*}
They give morphisms $\psi_\p:\tors(\Fl\L_\p)\to\tors\L_\p$ and $\psi:\tors(\Fl\L_\p)\to\tors\L$ of posets.
They satisfy $\psi(\cX)_\p = \psi_\p(\cX)$.
Moreover, for any $\cT\in\tors\L_\p$, we have
\begin{align}\label{psi XT}
\cX\subseteq\psi_\p(\cX)\ \mbox{ and }\ \cT\subseteq \psi_\p(\cT\cap\Fl\L_\p).
\end{align}
\end{propdef}
\begin{proof}
The equality $\psi(\cX)_\p=\psi_\p(\cX)$ is immediate from Lemma \ref{lem-local-global}(a).

Since the functor $(R_\p/\p R_\p)\otimes_{R_\p}(-) : \mod\L_\p \to \mod \Fl\L_\p$ is right exact and by Lemma \ref{lem-exact-lift}, we have $\psi_\p(\cX)\in\tors\L_\p$.
Since $\psi(\cX)=\{X\in\mod\L \mid X_\p \in \psi_\p(\cX)\}$ holds, Lemma \ref{lem-torsion-local-lift} implies that $\psi(\cX)\in\tors\L$.
The latter statements are clear.
\end{proof}
For an element $\cX=(\cX^{\mfp})_{\mfp}$ of  $\bT_R(\L)$ let
\[
\Psi_{\rm t}(\cX) := \bigcap_{\p \in \Spec R}\psi(\cX^\p)= \{ X \in \mod\L \mid \forall\p \in \Spec R,\, X_\p/\p X_\p\in  \cX^{\mfp}\}.
\]
Since $\psi(\cX^\p)\in\tors\L$ for each $\p$, we have $\Psi_{\rm t}(\cX)\in\tors\L$. Thus we have an order preserving map.
\begin{align}\label{map-Psi}
\Psi_{\rm t}: \bT_R(\L) \longrightarrow \tors\L, \quad \cX \mapsto \Psi_{\rm t}(\cX).
\end{align}
We have the following proposition, whose proof is independent from Theorem \ref{thm-torsion-free-iso}.
\begin{thm}\label{prop-injective-localization}
Let $(R, \L)$ be a Noetherian algebra. 
\begin{enumerate}[{\rm (a)}]
\item The composite $\Psi_{\rm t}\circ\Phi_{\rm t}$ is an identity map on $\tors\L$.
\item The map $\Phi_{\rm t} : \tors\L \to \bT_R(\L)$ is an embedding of posets.
\end{enumerate}
\end{thm}
\begin{proof}
(a) Fix $\cT\in\tors\L$. Then
$\Psi_{\rm t}\circ\Phi_{\rm t}(\cT)=\bigcap_{\p\in\Spec R}\psi(\cT_\p\cap\Fl\L_\p)\supseteq\cT$ holds by \eqref{psi XT}.
To prove $\Psi_{\rm t}\circ\Phi_{\rm t}(\cT)\subseteq\cT$, fix $X\in\Psi_{\rm t}\circ\Phi_{\rm t}(\cT)$.
Take a short exact sequence $0 \to T \to X \xto{f} F \to 0$ with $T\in\cT$ and $F\in\cT^\perp$.
It suffices to prove $F=0$.
Otherwise, take a minimal element $\p$ in $\supp_R F$.
Then $F_\p$ belongs to $\Fl\L_\p$, and there exists $\ell\ge1$ such that $\p^\ell F_\p=0$. The surjection $f:X\to T$ induces a surjection $X_\p/\p^\ell X_\p\to F_\p$.
Since $X\in\psi(\cT_\p\cap\Fl\L_\p)$, we have $X_\p/\p X_\p\in\cT_\p$ and hence $X_\p/\p^{\ell}X_\p\in \Filt(\gen(X_\p/\p X_\p)) \subseteq \cT_\p$. Since $F_\p\in\cT_\p^\perp$, we have $F_\p=0$, a contradiction.

(b) For $\cT,\cT'\in\tors\L$, assume $\Phi_{\rm t}(\cT)\le\Phi_{\rm t}(\cT')$. Since $\Psi_{\rm t}$ is a morphism of posets, we have $\cT=\Psi_{\rm t}\circ\Phi_{\rm t}(\cT)\le\Psi_{\rm t}\circ\Phi_{\rm t}(\cT')=\cT'$ by (a), as desired.
\end{proof}
Recall that $\tors\L$ and $\bT_R(\L)$ are complete lattices (see \cite[Proposition 2.3]{IRTT} for a finite dimensional algebra case), where meets and joins in $\bT_R(\L)$ are given by term-wise.
\begin{lem}\label{thm-Phi-image}
Let $(R, \L)$ be a Noetherian algebra.
The map $\Phi_{\rm t} : \tors\L \to \bT_R(\L)$ is a morphism of complete join-semilattices.
\end{lem}
\begin{proof}
The assertion follows from the equalities
\begin{align*}
&(  \bigvee_{i\in I}\cT^i)_\p \cap \Fl\L_\p=\Filt_\L(\cT^i\mid i\in I)_\p\cap\Fl\L_\p=\Filt_{\L_\p}(\cT^i_\p\mid i\in I)\cap\Fl\L_\p\\
=&\Filt_{\L_\p}(\cT^i_\p\cap\Fl\L_\p\mid i\in I)=\Filt_{\L_\p}(\cT^i_\p\cap\Fl\L_\p\mid i\in I)=\bigvee_i(\cT^i_\p \cap \Fl\L_\p).\qedhere
\end{align*}
\end{proof}
Contrary to $\Phi_{\rm f}$, the map $\Phi_{\rm t}$ is rarely bijective.
The rest of this paper is devoted to studying characterizations of the elements of $\Im \Phi_{\rm t}$.
We introduce the notion of compatible elements of $\bT_R(\L)$ which is a necessary condition for elements of $\bT_R(\L)$ to belong to $\Im\Phi_{\rm t}$.
\begin{dfn}\label{dfn-pairwise-compatibel}
For two prime ideals $\mfp \supseteq \mfq$ of $R$, we define a map ${\rm r_{\p,\q}}: \tors(\Fl\L_\p) \to \tors(\Fl\L_\q)$ as the composition
\begin{align}\label{map-pairwise-comp}
{\rm r_{\p,\q}}: \tors(\Fl\L_\p) \xto{\psi_\p}\tors\L_\p\xto{(-)_\q\cap\Fl\L_\q}\tors(\Fl\L_\q), \qquad \cT \mapsto \psi_\p(\cT)_\q\cap\Fl\L_\q.
\end{align}
We call an element $(\cX^{\mfp})_{\mfp}$ of $\bT_R(\L)$ \emph{compatible} if ${\rm r_{\p,\q}}(\cX^{\p})\supseteq\cX^{\q}$ holds for any pair $\p \supseteq \q$ of prime ideals of $R$.
We denote by $\bT_R^{\rm c}(\L)$ the subposet of $\bT_R(\L)$ consisting of all compatible elements in $\bT_R(\L)$.
We say that $(R,\L)$ is \emph{compatible} if $\Im\Phi_{\rm t}=\bT_R^{\rm c}(\L)$.
\end{dfn}
We give a few remarks on the map ${\rm r_{\p,\q}}$.

\begin{rem}\label{prop-comp-rpq}
Let $\p \supseteq \q \supseteq \mathfrak{r}$ be prime ideals of $R$.
\begin{enumerate}[\rm(a)]
\item We have ${\rm r_{\p,\p}}=\id_{\tors(\Fl\L_\p)}$ and ${\rm r_{\p,\q}}(\Fl\L_\p)=\Fl\L_\q$.
\item For each $\cT\in\tors(\Fl\L_\p)$, we have
${\rm r}_{\q, \mathfrak{r}} \circ {\rm r}_{\p, \q}(\cT) \supseteq {\rm r}_{\p, \mathfrak{r}}(\cT)$.
\item The map ${\rm r_{\p,\q}}$ is natural from the point of view of silting theory, see Corollary \ref{cor-psilt-rpq} below.
\end{enumerate}
\end{rem}
\begin{proof}
(a)
Let $\cX\in\tors(\Fl\L_\p)$. Then ${\rm r}_{\p,\p}(\cX)\subseteq\cX$ holds by definition of ${\rm r}_{\p,\p}$.  The other inclusion follows from $\cX\subseteq \psi_\p(\cX)$ and $\cX_\p\cap\Fl\L_\p=\cX$.
The second equality follows from $\psi_\p(\Fl\L_\p)=\mod\L_\p$.

(b)
By applying $\psi_\q$ for ${\rm r}_{\p,\q}(\cT)$, we have  $\psi_\q({\rm r}_{\p,\q}(\cT)) = \psi_\q(\psi_\p(\cT)_\q\cap\Fl\L_\q) \stackrel{\eqref{psi XT}}{\supseteq} \psi_\p(\cT)_\q$. Applying $(-)_{\mathfrak{r}}\cap\Fl\L_{\mathfrak{r}}$, we have
${\rm r}_{\q, \mathfrak{r}} \circ {\rm r}_{\p, \q}(\cT)= \psi_\q({\rm r}_{\p,\q}(\cT))_{\mathfrak{r}}\cap\Fl\L_{\mathfrak{r}} \supseteq (\psi_\p(\cT)_\q)_{\mathfrak{r}}\cap\Fl\L_{\mathfrak{r}} = {\rm r_{\p,\mathfrak{r}}}(\cT)$.
\end{proof}
The importance of the notion of compatible elements is explained by the next observation. Notice that our condition ${\rm r_{\p,\q}}(\cX^{\p})\supseteq\cX^{\q}$ can not be replaced by the stronger one ${\rm r_{\p,\q}}(\cX^{\p})=\cX^{\q}$, which is not satisfied by $\Phi_{\rm t}(\Fl\L)$ if $\dim R\ge1$.
\begin{prop}\label{prop-wcomp-comp-easy}
For a Noetherian algebra $(R, \L)$, we have $\Im \Phi_{\rm t}\subseteq\bT_R^{\rm c}(\L)$. Thus we have an embedding of posets
\[
\Phi_{\rm t} : \tors\L \hookrightarrow \bT_R^{\rm c}(\L).
\]
\end{prop}
\begin{proof}
For $\p\supseteq \q$ and $\cT\in\tors\L$, we have $\psi_{\p}(\cT_\p \cap \Fl\L_\p)\supseteq\cT_\p$ by \eqref{psi XT}. Thus
\[
{\rm r}_{\p,\q}(\cT_\p \cap \Fl\L_\p)=\psi_\p(\cT_\p \cap \Fl\L_\p)_\q \cap \Fl\L_\q \supseteq \cT_\q \cap\Fl\L_\q,
\]
that is $\Phi_{\rm t}(\cT)$ is compatible. The latter assertion is immediate from Theorem \ref{prop-injective-localization}(b).
\end{proof}
As an immediate consequence of Proposition \ref{prop-wcomp-comp-easy}, we give a complete description of $\tors\L$ for some special cases. The first one is the following.
\begin{prop}\label{prop-local-one}
Let $(R,\L)$ be a Noetherian algebra. If $(R,\m)$ is a local ring with $\dim R=1$, then $(R, \L)$ is compatible.
\end{prop}
\begin{proof}
Let $\cX=(\cX^{\mfp})_{\mfp}\in\bT_R^{\rm c}(\L)$. For $\Psi_{\rm t}(\cX)=\bigcap_{\p\in\Spec R}\psi(\cX^\p)$, 
it suffices to show $\cX^\p\subseteq\Psi_{\rm t}(\cX)_\p$ for each $\p\in\Spec R$. This is clear if $\p=\m$. In fact, each $X\in\cX^\m$ satisfies $\supp_RX\subseteq\{\m\}$ and $X_\m/\m X_\m\in\cX^{\m}$, and hence $X\in\Psi_{\rm t}(\cX)$ holds.
Now assume $\p\neq\m$. Since ${\rm r_{\m,\,\p}}(\cX^{\m})\supseteq\cX^{\p}$ holds, for each $X\in\cX^\p$, there exists $M\in\psi(\cX^\m)$ such that $M_\p \simeq X$ and hence $M\in\psi(\cX^\p)$. Replacing $M$ by the image of $M\to X$, we can assume that $M$ is a $\L$-submodule of $X$. Since $X$ has finite length as $\L_\p$-module, Lemma \ref{lem-local-global}(b) implies $\supp_RM\subseteq V(\p)=\{\p,\m\}$.
Thus $M\in \Psi_{\rm t}(\cX)$ holds, and hence $X=M_\p\in\Psi_{\rm t}(\cX)_\p$.
\end{proof}
\begin{rem}
In the case where $(R, \m)$ is a local domain with the Krull dimension one, in \cite[Theorem 5.9]{Kimura}, it was shown that there is a bijection
\[
\tors\L \simeq \bigsqcup_{\cT\in\tors(\Fl\L)}[0, {\rm r}_{\m, 0}(\cT)]_{\tors\L_0},
\]
where $[0, {\rm r}_{\m, 0}(\cT)]_{\tors\L_0}$ is an interval in the poset $\tors\L_0$.
Clearly, this disjoint union bijectively corresponds to $\bT_R^{\rm c}(\L)$.
Thus Proposition \ref{prop-local-one} recovers this bijection.
\end{rem}

We explain the second case.
A subset $\cS$ of $\Spec R$ is said to be \emph{specialization-closed} if for any two prime ideals $\p\supseteq \mfq$ of $R$, $\q\in\cS$ always implies $\mfp\in\cS$. Equivalently, $\cS$ is a union of Zariski closed subsets.
\begin{cor}\label{cor-local-sp-closed}
Let $(R, \L)$ be a Noetherian algebra such that $\L_\p$ is Morita equivalent to a (non-zero) local ring for each $\p\in\Spec R$.
\begin{enumerate}[\rm(a)]
\item There is a canonical isomorphism $\tors\L\simeq\sP_{\rm spcl}(\Spec R)$ of posets. Moreover $(R, \L)$ is compatible.
\item $\tors\L=\serre\L$ holds.
\end{enumerate}
\end{cor}
Notice that the case $\L=R$ recovers famous results of Gabriel \cite{Gabriel} and Stanley-Wang \cite{Stanley-Wang}.
\begin{proof}
(a) Since $\L_\p$ is Morita equivalent to a local ring, $\tors(\Fl\L_\p)=\{0, \Fl\L_\p\}$ holds for any $\p\in\Spec R$.
Since ${\rm r}_{\p,\q}(0)=0$ and ${\rm r}_{\p,\q}(\Fl\L_\p)=\Fl\L_\q$, there exists a bijection between $\bT_R^{\rm c}(\L)$ and specialization closed subsets of $\Spec R$ given by
\[
{\rm s}((\cX^\p)_\p)=\{\p\in\Spec R \mid \cX^\p \neq 0\}.
\]
Thus ${\rm s}\circ \Phi_{\rm t}$ is an injection from $\tors\L$ to the set of specialization closed subsets of $\Spec R$.
It remains to show that ${\rm s}\circ \Phi_{\rm t}$ is surjective.
For a specialization closed subset $\cS$ of $\Spec R$, let
\[\cT:=\{X\in\mod\L\mid \supp_RX\subseteq\cS\}.\]
Clearly $\cT\in\tors\L$ holds. Let $(\cX^\p)_\p:=\Phi_{\rm t}(\cT)$. By Lemma \ref{lem-local-global}, we have $\cX^\p=\Fl\L_\p$ if $\p\in\cS$, and $\cX^\p=0$ otherwise. Thus
${\rm s}\circ\Phi_{\rm t}(\cT)=\cS$ holds

(b) It suffices to show $\tors\L\subseteq\serre\L$. Since the subcategory $\cT$ defined above belongs to $\serre\L$, the assertion follows.
\end{proof}
These results give partial answers to Question \ref{intro-ques-comp}.
In Subsections \ref{subsection-dim-one} and \ref{subsection-comp-silt-finite}, we give other examples of compatible Noetherian algebras.
\subsection{Serre subcategories}
In this subsection, we give a classification of Serre subcategories of $\mod\L$ for an arbitrary Noetherian algebra $(R, \L)$, which is a vast generalization of a famous result by Gabriel \cite{Gabriel}.

Let $\bS_R(\L)$ be the Cartesian product of $\serre(\Fl\L_\p)$, where $\p$ runs over all prime ideals of $R$:
\[
\bS_R(\L):=\prod_{\mfp\in\Spec R}\serre(\Fl\L_\p).
\]
By Lemma \ref{lem-subcat-local}, we have a well-defined map
\[\Phi_{\rm s}:=\Phi_{\rm t}|_{\serre\L}:\serre\L\to\bS_R(\L),\]
which is an embedding of posets by Theorem \ref{prop-injective-localization}.
For a ring $\Gamma$, we denote by $\simple\Gamma$ the set of all isomorphism classes of simple $\Gamma$-modules.
We just remark that $\simple\L_\p=\simple\L(\p)$ holds, since any simple $\L_\p$-module $S$ satisfies $\p S=0$ by Nakayama's lemma.
It is easy to see that there is a bijection between the set $\sP(\simple\L_\p)$ of all subsets of $\simple\L_\p$ and $\serre(\Fl\L_\p)$:
\begin{align}\label{bij-simple-serre}
\serre(\Fl\L_\p) \longrightarrow \sP(\simple\L_\p), \quad \cC \mapsto \cC\cap\simple\L_\p,
\end{align}
where the inverse map is given by $\cS \mapsto \Filt\cS$ for a subset $\cS$ of $\simple\L_\p$.
These bijections are isomorphisms of posets, where $\sP(\simple\L_\p)$ is a poset by inclusion.
Let $\Simple_R\L$ be the disjoint union of $\simple\L_\p$:
\[
\Simple_R\L:=\bigsqcup_{\mfp\in\Spec R}\simple\L_\p.
\]
By using \eqref{bij-simple-serre}, we have an isomorphism of posets $\iota : \bS_R(\L) \to\sP(\Simple_R\L)$.
By composing $\iota$ and $\Phi_{\rm s}$, we have the following morphism of posets, which is injective since $\Phi_{\rm s}$ is injective:
\[
\iota\circ\Phi_{\rm s} : \serre\L \longrightarrow \bS_R(\L) \longrightarrow \sP(\Simple_R\L), \qquad \cC \mapsto \bigsqcup_{\p\in\Spec R}( \cC_{\mfp}\cap\simple\L_\p).
\]
Now we introduce a partial order on $\Simple_R\L$.
For $\p \supseteq \q$, we regard a $\L_\q$-module as a $\L_\p$-module via a natural morphism $\L_\p \to \L_\q$.
\begin{dfn}\label{dfn-simple-poset}
We define a partial order on $\Simple_R\L$ as follows.
\begin{enumerate}[{\rm (1)}]
\item
For given $S, T\in\Simple_R\L$, we write $S \leq T$ if there exist prime ideals $\p \supseteq \q$ of $R$ such that $S\in\simple\L_\p$, $T\in\simple\L_\q$ and $S$ is a subfactor $\L_\p$-module of $T$.
\item
A subset $\cS$ of $\Simple_R\L$ is called a \emph{down-set} if $S\leq T$ and $T\in\cS$ always implies $S\in\cS$.
\end{enumerate}
\end{dfn}
For instance, if $\L=R$, then $\Simple_R\L=\{\kappa(\p)\mid\p\in\Spec R\}$, and $\kappa(\p)\le\kappa(\q)$ holds if and only if $\p\supseteq \q$ holds. Thus down-sets of $\Simple_R\L$ correspond bijectively with specialization closed subsets of $\Spec R$.

By using down-sets, we classify Serre subcategories.
\begin{thm}\label{thm-serre-poset-bijection}
For a Noetherian algebra $(R, \L)$, the injective map
\[
\iota\circ\Phi_{\rm s} : \serre\L \longrightarrow \sP(\Simple_R\L), \quad \cC \mapsto \bigsqcup_{\p\in\Spec R}( \cC_{\mfp}\cap\simple\L_\p)
\]
induces an isomorphism $\serre\L\simeq\sP_{\rm down}(\Simple_R\L)$ of posets.
\end{thm}
See Subsection \ref{subsection-example} for a concrete example of the theorem.

To prove Theorem \ref{thm-serre-poset-bijection}, we need the following observation.
\begin{lem}\label{lem-comp-local}
Let $(R, \Gamma)$ be a Noetherian algebra, $X\in\mod\Gamma$, $S\in\simple\Gamma$, and $\mfq\in\Spec R$.
If $S$ is a subfactor of $X_{\q}$ as a $\Gamma$-module, then $S$ is a subfactor of $X$.
The converse holds if a natural morphism $X\to X_{\mfq}$ is injective.
\end{lem}
\begin{proof}
Assume that $S$ is a subfactor of $X_{\mfq}$.
By Lemma \ref{lem-local-global}(a), there are $\Gamma$-submodules $Z\subseteq Y$ of $X$ such that $(Y/Z)_\q$ is isomorphic to $S$ as $\Gamma$-modules.
So there is a non-zero $\Gamma$-morphism from $Y$ to $S$, which is surjective since $S$ is simple.
Therefore $S$ is a subfactor $\Gamma$-module of $X$.
The last assertion is clear.
\end{proof}
We prove the theorem.
\begin{proof}[Proof of Theorem \ref{thm-serre-poset-bijection}]
Let $\cC\in\serre\L$.
We first show that $\iota\circ\Phi_{\rm s}(\cC)$ is a down-set of $\Simple_R\L$.
Let $S\leq T$ with $S \in \simple\L_\p$ and $T \in \simple\L_\q$ for prime ideals $\p\supseteq \q$ of $R$.
Assume that $T\in\iota\circ\Phi_{\rm s}(\cC)$ and take $X\in\cC$ such that $X_\q\simeq T$.
Since $S$ is a subfactor of $T\simeq X_\q$, Lemma \ref{lem-comp-local} implies that $S$ is a subfactor of $X_\p\in\cC_\p$. Since $\cC_\p$ is a Serre subcategory, we obtain $S\in\cC_\p$.

Let $\cS$ be a down-set of $\Simple_R\L$.
Let $\cC$ be the full subcategory of $\mod\L$ consisting of modules $X$ such that for each $\p\in\Spec R$, all composition factor of the $\L_\p$-module $X_{\mfp}$ belong to $\cS$.
Clearly $\cC$ is a Serre subcategory of $\mod\L$.
We show $\iota\circ\Phi_{\rm s}(\cC)=\cS$ holds.
The inclusion ``$\subseteq$'' is clear.
To prove ``$\supseteq$'', we fix $\q\in\Spec R$ and $S\in \simple\L_\q\cap\cS$.
By Lemma \ref{lem-local-global}(b), there exists a $\L$-submodule $X\in\mod\L$ of $S$ such that $X_\q\simeq S$ and $\supp_R X\subseteq V(\mfq)$.
We claim $X\in\cC$. In fact, for $\p \in\Spec R\setminus V(\q)$, we have $X_\p=0$.
For $\p \in V(\q)$, since $X_\p$ is a $\L_\p$-submodule of $S$, each simple subfactor $\L_\p$-module of $X_\p$ is a subfactor of $S$ and hence belongs to the down set $\cS$. Thus $X\in\cC$ holds.
Therefore $S\in\iota\circ\Phi_{\rm s}(\cC)$ holds as desired.
\end{proof}
\begin{rem}\label{rem-serre-kanda}
Serre subcategories of $\mod\L$ were classified by Kanda \cite{Kanda-ext} by using his theory of atom spectrum of a noetherian abelian category \cite{Kanda-class}.
In fact, one can check that the atom spectrum ${\rm ASpec}\,\L$ of $\mod\L$ is homeomorphic to $\Simple_R\L$ with poset topology \cite[Theorem 7.6]{Kanda-ext}. Thus we have a bijection
\[\{\mbox{Down-sets of $\Simple_R\L$}\}\simeq\{\mbox{Open subsets of ${\rm ASpec}\L$}\},\]
which correspond bijectively with $\serre\L$.
\end{rem}
Now we give an analogue of Theorem \ref{prop-injective-localization}(a) for Serre subcategories.
\begin{prop}\label{lem-inverse-direction-serre}
Let $(R,\L)$ be a Noetherian algebra.
\begin{enumerate}[\rm(a)]
\item Let $\p\in\Spec R$. The morphisms $\psi_\p$ and $\psi$ in Definition \ref{dfn-bar} give morphisms $\psi_\p:\serre(\Fl\L_\p)\to\serre\L_\p$ and $\psi:\serre(\Fl\L_\p)\to\serre\L$ of posets.
\item We have a morphism of posets
\[\Psi_{\rm s}:=\Psi_{\rm t}|_{\bS_R(\L)}:\bS_R(\L)\to\serre\L\]
such that the composite $\Psi_{\rm s}\circ\Phi_{\rm s}$ is the identity map on $\serre\L$.
\end{enumerate}
\end{prop}
\begin{proof}
(a) For $\cX\in\serre(\Fl\L_\p)$, we show that $\psi_\p(\cX)\in\serre\L_\p$.
It is enough to show that $\psi_\p(\cX)$ is closed under submodules.
Let $X\in\psi_\p(\cX)$ and $Y$ be a submodule of $X$.
We show $Y/\p Y\in\cX$.
By the Artin-Rees lemma, there is a positive integer $n$ such that $\p^{n+1}X\cap Y =\p(\p^nX \cap Y)$ holds. Then
\[
\frac{X}{\p^{n+1}X}\supseteq\frac{Y+\p^{n+1}X}{\p^{n+1}X} \simeq \frac{Y}{\p^{n+1}X\cap Y} = \frac{Y}{\p(\p^nX\cap Y)}.
\]
Since $Y/\p Y$ is a factor module of the right term and $X/\p^{n+1}X\in \Filt(\gen(X/\p X)) \subseteq \cX$, we have $Y/\p Y\in\cX$.
Since $\psi(\cX)=\{X\in\mod\L \mid X_\p \in \psi_\p(\cX)\}$ holds, Lemma \ref{lem-torsion-local-lift} implies $\psi(\cX)\in\serre\L$.

(b) Immediate from (a) and Theorem \ref{prop-injective-localization}(a).
\end{proof}
We give an easy observation characterizing an algebra $\L$ with the property $\tors\L=\serre\L$, which appeared in Corollary \ref{cor-local-sp-closed}.
\begin{prop}\label{prop-tors-serre}
For a Noetherian algebra $(R,\L)$, the following statements are equivalent.
\begin{enumerate}[{\rm (i)}]
\item $\tors\L=\serre\L$ holds.
\item For each $\p\in\Spec R$, $\L_\p$ is Morita equivalent to a finite direct product of local rings.
\end{enumerate}
\end{prop}
We need the following easy observation.
\begin{lem}\label{criterion for tors=serre}
Let $\cA$ be an abelian category such that each object has finite length and $\cA$ is non-zero and indecomposable as an abelian category. Then $\tors\cA=\serre\cA$ holds if and only if $\cA$ has a unique simple object up to isomorphism.
\end{lem}
\begin{proof}
The ``if'' part is clear since $\tors\cA=\{\cA,0\}$.
We prove the ``only if'' part. It suffices to show that $\Ext^1_{\cA}(S,S')=0$ holds for each pair of simple objects $S\not\simeq S'$ in $\cA$. Let $\cT$ be the subcategory of $\cA$ consisting of objects $X$ such that ${\rm top} X\in\add S$. Then $\cT\in\tors\cA=\serre\cA$. Thus each composition factor of each object in $\cT$ is isomorphic to $S$. This implies $\Ext^1_{\cA}(S,S')=0$ as desired.
\end{proof}
We are ready to prove Proposition \ref{prop-tors-serre}.
\begin{proof}[Proof of Proposition \ref{prop-tors-serre}]
(i)$\Rightarrow$(ii)
By Lemma \ref{criterion for tors=serre}, it suffices to prove that for each $\p\in\Spec R$, each
$\cX\in\tors(\Fl\L_\p)$ is a Serre subcategory. By Lemma \ref{lem-torsion-local-lift} and our assumption (i),
$\psi(\cX)\in\tors\L=\serre\L$ holds.
Thus $\cX=\psi(\cX)_\p\cap\Fl\L_\p\in\serre(\Fl\L_\p)$ by Lemma \ref{lem-subcat-local}(b).

(ii)$\Rightarrow$(i) By Lemma \ref{criterion for tors=serre} and our assumption (ii), $\tors(\Fl\L_\p)=\serre(\Fl\L_\p)$ holds.
Thus $\bT_R(\L)=\bS_R(\L)$, and we have $\tors\L\stackrel{}{=}\Psi_{\rm t}(\bT_R(\L))=\Psi_{\rm t}(\bS_R(\L))=\serre\L$ by Theorem \ref{prop-injective-localization}(a) and Proposition \ref{lem-inverse-direction-serre}(b).
\end{proof}
\subsection{Compatible elements}\label{subsection-dim-one}
In this subsection, we show that $(R, \L)$ is compatible if $R$ is a semi-local ring with the Krull dimension one.

In the rest of this subsection, for a subset $V$ of $\Spec R$ we denote by $V^{\rm c}$ the complement of $V$ in $\Spec R$.

The main result of this subsection is the following theorem.
\begin{thm}\label{thm-cofinite-compatible}
Let $(R, \L)$ be a Noetherian algebra and $\cX=(\cX^\p)_\p\in\bT_R^{\rm c}(\L)$.
Assume that there exist $\cT\in\tors\L$ and a finite Zariski closed subset $V$ of $\Spec R$ such that
\begin{itemize}
\item $\cX^\p \subseteq \cT_\p \cap\Fl\L_\p$ holds for any $\p\in\Spec R$, and the equality holds if $\p\in V^{\rm c}$.
\end{itemize}
Then there exists $\cU\in\tors\L$ such that $\cX=\Phi_{\rm t}(\cU)$.
\end{thm}
In the rest of this subsection, we prove Theorem \ref{thm-cofinite-compatible}.
We prepare the following easy observation.
\begin{lem}\label{X^perp}
Assume that $\p,\q\in\Spec R$ satisfies $\p\supseteq\q$. Then for each $(\cX^\q)_\q\in\bT_R^{\rm c}(\L)$, we have $\left(\psi(\cX^\p)^{\perp_{\L}}\right)_\q \subseteq \left( \cX^\q\right)^{\perp_{\L_\q}}$.
\end{lem}
\begin{proof}
We have $\psi(\cX^\p)_\q = (\psi(\cX^\p)_\p)_\q=\psi_\p(\cX^\p)_\q \supseteq{\rm r}_{\p,\q}(\cX^\p)\supseteq \cX^\q$.
Thus $\left(\psi(\cX^\p)^{\perp_{\L}}\right)_\q \stackrel{\ref{lem-torsion-local}}{=} \left( \psi(\cX^\p)_\q \right)^{\perp_{\L_\q}}\subseteq \left( \cX^\q\right)^{\perp_{\L_\q}}$.
\end{proof}
The following observation is crucial.

\begin{lem}\label{lem-spec-lower-set}
Let $(\cX^\q)_\q\in\bT_R^{\rm c}(\L)$, $V$ a specialization closed subset of $\Spec R$, and $\p$ a minimal element in $V$.
Then each $X\in \bigcap_{\q\in V^{\rm c}}\psi(\cX^\q)$ admits a $\L$-submodule $Y$ such that $Y\in \bigcap_{\q\in V^{\rm c} \cup \{\p\}} \psi(\cX^\q)$ and $Y_\q = X_\q$ for any $\q\in V^{\rm c}$.
\end{lem}
\begin{proof}
Take a short exact sequence $0 \to Y \to X \xto{f} Z \to 0$ in $\mod\L$ such that $Y\in\psi(\cX^\p)$ and $Z\in\psi(\cX^\p)^{\perp}$.
It suffices to show that $\supp_RZ\subseteq V(\p)$ since $X_\q=Y_\q$ holds for each $\q\in(\supp_RZ)^{\rm c}\supseteq V(\p)^{\rm c}\supseteq V^{\rm c}$.
We divide the proof into two steps.

(i)
We show that each $\q\in\ass_RZ$ satisfies $\q\subseteq \p$.
Take $x\in Z$ such that $Rx\simeq R/\q$ as $R$-modules. By Lemma \ref{rxAx}, we have $\supp_R(\L x)=\supp_R(R/\q)=V(\q)$. Thus, if $\q \not\subseteq\p$, then $(\L x)_\p=0$ and hence $\L x\in\psi(\cX^\p)$ hold. Then $\Hom_{\L}(\L x, Z)=0$ holds, a contradiction to $\L x \subseteq Z$.
Thus $\q \subseteq \p$ holds.

(ii)
We show $\supp_R Z \subseteq V(\p)$.
It suffices to show that each minimal element $\q$ in $\supp_R Z$ satisfies $\q=\p$.
By (i) $\q\subseteq \p$ holds. By Lemma \ref{X^perp}, we have $Z_\q \in \left(\psi(\cX^\p)^{\perp_{\L}}\right)_\q\subseteq\left( \cX^\q\right)^{\perp_{\L_\q}}$.
Assume $\q\neq\p$. Then $\q\in V^{\rm c}$ holds since $\p$ is a minimal element of $V$.
Thus we have $X_\q \in \psi_\q(\cX^\q)$ and hence $Z_\q\in\psi_\q(\cX^\q)\cap\Fl\L_\q=\cX^\q$.
Since $Z_\q \in \left( \cX^\q\right)^{\perp_{\L_\q}}$, we have $Z_\q=0$, a contradiction.
Thus $\p=\q$ holds as desired.
\end{proof}
Using Lemma \ref{lem-spec-lower-set} repeatedly, we obtain the following result.

\begin{prop}\label{prop-cofinite-down-set}
Let $\cX=(\cX^\p)_\p\in\bT_R^{\rm c}(\L)$, and $V$ a finite Zariski closed subset of $\Spec R$.
Then each $X\in \bigcap_{\q\in V^{\rm c}}\psi(\cX^\q)$ admits a $\L$-submodule $Y$ such that $Y\in \Psi_{\rm t}(\cX)$ and $Y_\p = X_\p$ for any $\p\in V^{\rm c}$.
\end{prop}
\begin{proof}
We use induction on $\#V$. Take a minimal element $\p$ of $V$, and let $W:=V\setminus\{\p\}$.
By Lemma \ref{lem-spec-lower-set}, $X$ admits a $\L$-submodule $Y$ such that $Y\in \bigcap_{\q\in W^{\rm c}}\psi(\cX^\q)$ and $Y_\q = X_\q$ for any $\q\in V^{\rm c}$.
Since $W$ is a finite Zariski closed subset with $\#W=\#V-1$, our induction hypothesis implies that $Y$ admits a $\L$-submodule $Z$ such that $Z\in \Psi_{\rm t}(\cX)$ and $Z_\q = Y_\q$ for any $\q\in W^{\rm c}$. Thus $Z$ gives the desired submodule of $X$.
\end{proof}
Then we are ready to show Theorem \ref{thm-cofinite-compatible}.
\begin{proof}[Proof of Theorem \ref{thm-cofinite-compatible}]
Assume that $\cT\in\tors\L$ and $\cX=(\cX^\p)_\p\in\bT_R(\L)$ satisfy the assumption.
We show that $\Phi_{\rm t}(\Psi_{\rm t}(\cX))=\cX$ holds, that is, $\Psi_{\rm t}(\cX)_\p\cap\Fl\L_\p = \cX^\p$ holds for any $\p\in\Spec R$.
By the definition of $\Psi_{\rm t}$, $\Psi_{\rm t}(\cX)_\p\cap\Fl\L_\p \subseteq \cX^\p$ always holds.
We show that $\cX^\p\subseteq \Psi_{\rm t}(\cX)_\p$ for any $\p\in\Spec R$.

Fix one $\p\in\Spec R$ and $X\in\cX^\p$.
We show that $X\in\Psi_{\rm t}(\cX)_\p$.
Since $\cX^\p \subseteq \cT_\p \cap \Fl\L_\p$ holds and by Lemma \ref{lem-local-global}, there exists $T\in\cT$ such that $X = T_\p$ and $\supp_R T \subseteq V(\p)$.
Let $V'=V\cap (V(\p)\setminus \{\p\})$, then $V'^{\rm c}=V^{\rm c} \cup V(\p)^{\rm c} \cup \{\p\}$ holds.
We claim that $T\in\bigcap_{\q\in V'^{\rm c}}\psi(\cX^\q)$.
Then by applying Proposition \ref{prop-cofinite-down-set} to $(X, V):=(T, V')$, we obtain $Y\in\Psi_{\rm t}(\cX)$ such that $Y_\p=T_\p=X$.
If $\q=\p$, then $T_\q\in\cX^\q$ by our choice of $T$.
If $\q \in V(\p)^{\rm c}$, then $\supp_R T \subseteq V(\p)$ implies $T_\q =0\in\psi_\q(\cX^\q)$.
If $\q \in V^{\rm c}$, we have $T_\q\in\psi_\q(\cT_\q\cap\Fl\L_\q)=\psi_\q(\cX^\q)$ by our assumption.
Therefore $T\in\bigcap_{\q\in V'^{\rm c}}\psi(\cX^\q)$ as desired.
\end{proof}
As an application of the theorem, we obtain the following result.
\begin{cor}\label{cor-semi-local}
Let $(R, \L)$ be a Noetherian algebra.
If $R$ is semi-local and $\dim R=1$, then $(R, \L)$ is compatible.
Thus we have an isomorphism $\Phi_{\rm t} : \tors\L \simeq \bT_R^{\rm c}(\L)$ of posets.
\end{cor}
\begin{proof}
By Proposition \ref{Spec R finite}, $\Spec R$ is a finite set. Thus we can apply Theorem \ref{thm-cofinite-compatible} for $\cT=\mod\L$.
\end{proof}
\subsection{Examples}\label{subsection-example}
We explaing our results by giving explicit descriptions of torsion classes, torsionfree classes and Serre subcategories of the following example: let $k$ be a field and $R=k[[x]]$ with a unique maximal ideal $\mfm=(x)$. We consider a Noetherian algebra
\[
\L=\left\{ \left(f, \left( \begin{array}{cc}
a & b \\
c & d
\end{array}
\right) \right) \in R \times \left( \begin{array}{cc}
R & R \\
\mfm & R
\end{array}
\right)\ \middle|\ f - a \in \mfm \right\}.
\]
Let $K=R_{0}=k((x))$ be the fractional field of $R$.
Then
\[
\overline{\L}:=\L/\m\L  \quad \mbox{and} \quad \L_0=K\otimes_R\L=K\times M_2(K)
\]
are finite dimensional algebras over $k$ and $K$ respectively, where $M_2(K)$ is the $K$-algebra of $(2\times 2)$-matrices over $K$.

First we calculate the Hasse quiver of $\serre\L$.
There are two simple $\L_0$-modules
\begin{equation}\label{T_i}
T_1=K \times 0, \qquad T_2= 0\times \begin{pmatrix} K\\ K \end{pmatrix}.
\end{equation}
Also there are two simple $\overline{\L}$-modules $S_1$ and $S_2$ associated with the idempotents $e_1=(1_R, E_1)$ and $e_2=(0, E_2)$ of $\L$, respectively, where $E_i$ is the matrix unit whose $(i, i)$ entry is $1_R$.
Then the Hasse quiver of the poset $(\Simple_R\L, \leq)$ is the following.
\[
\begin{tikzpicture}[baseline=-20]
\node(T1)at(0,0){$T_1$};
\node(T2)at(1.5,0){$T_2$};
\node(S1)at(0,-1){$S_1$};
\node(S2)at(1.5,-1){$S_2$};
\draw[thick, ->] (T1)--(S1);
\draw[thick, ->] (T2)--(S1);
\draw[thick, ->] (T2)--(S2);
\end{tikzpicture}
\]
One can easily check that the poset of down-sets of $(\Simple_R\L, \leq)$ has the following Hasse quiver,
which is isomorphic to the Hasse quiver of $\serre\L$  by Theorems \ref{thm-serre-poset-bijection}.
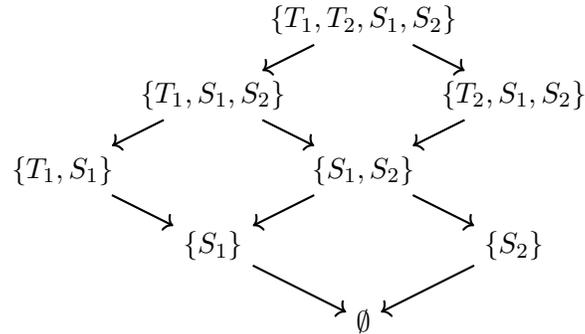
\[
\begin{tikzpicture}[baseline=-30]
\node(mod)at(0,1){$\{T_1, T_2, S_1, S_2\}$};
\node(112)at(-2,0){$\{T_1, S_1, S_2\}$};
\node(212)at(2,0){$\{T_2, S_1, S_2\}$};
\node(11)at(-4,-1){$\{T_1, S_1\}$};
\node(12)at(0,-1){$\{S_1, S_2\}$};
\node(1)at(-2,-2){$\{S_1\}$};
\node(2)at(2,-2){$\{S_2\}$};
\node(e)at(0,-3){$\emptyset$};
\draw[thick, ->] (mod)--(112);
\draw[thick, ->] (mod)--(212);
\draw[thick, ->] (112)--(11);
\draw[thick, ->] (112)--(12);
\draw[thick, ->] (212)--(12);
\draw[thick, ->] (11)--(1);
\draw[thick, ->] (12)--(1);
\draw[thick, ->] (12)--(2);
\draw[thick, ->] (1)--(e);
\draw[thick, ->] (2)--(e);
\end{tikzpicture}.
\]
Explicit descriptions of Serre subcategories inside $\mod\L$ are given in $\tors \L$ below.

Next we calculate $\tors\L$.
By direct calculations, one can see that $\siltm\L$ and $\siltm\L_0$ are finite sets (see \cite{Adachi-Iyama-Reiten} and \cite[Theorem 2.22]{Kimura} for mutation theory of silting modules).
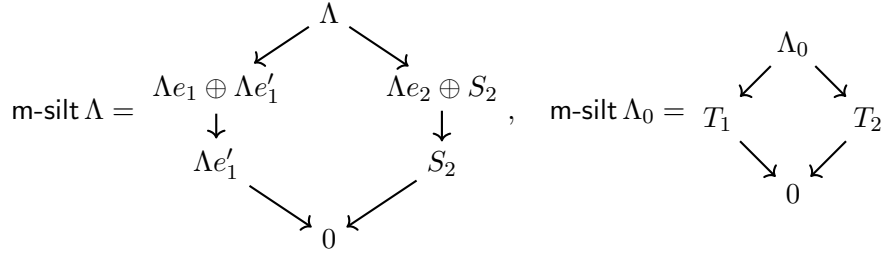
\begin{equation*}
\siltm\L = \begin{tikzpicture}[baseline=-40]
\node(L)at(0,0){$\L$};
\node(M2P2)at(-1.5,-1){$\L e_1\oplus \L e'_1$};
\node(P1M1)at(1.5,-1){$\L e_2\oplus S_2$};
\node(M2)at(-1.5,-2){$\L e'_1$};
\node(M1)at(1.5,-2){$S_{2}$};
\node(0)at(0,-3){$0$};
\draw[thick, ->] (L)--(M2P2);
\draw[thick, ->] (L)--(P1M1);
\draw[thick, ->] (M2P2)--(M2);
\draw[thick, ->] (P1M1)--(M1);
\draw[thick, ->] (M2)--(0);
\draw[thick, ->] (M1)--(0);
\end{tikzpicture},
\quad
\siltm\L_0=\begin{tikzpicture}[baseline=0]
\node(mod)at(0,1){$\L_0$};
\node(T1)at(-1,0){$T_1$};
\node(T2)at(1,0){$T_2$};
\node(0)at(0,-1){$0$};
\draw[thick, ->] (mod)--(T1);
\draw[thick, ->] (mod)--(T2);
\draw[thick, ->] (T1)--(0);
\draw[thick, ->] (T2)--(0);
\end{tikzpicture}
\end{equation*}
where $e'_1=(1, \begin{psmallmatrix} 0 & 0 \\ 0 & 0  \end{psmallmatrix})\in\L_0$ is a central idempotent so that $\L e'_1\simeq R$ as $R$-algebras.
By \cite[Theorem 3.8]{DIJ} and Theorem \ref{thm-silting-finite} below, each torsion class of $\tors(\Fl\L)$ and $\tors\L_0$ is of the form $\Fac M$ for some silting module $M$.
Note that $\Fac(\L e_1\oplus\L e'_1)=\Fac\L e_1$ and $\Fac(\L e_2\oplus S_2) =\Fac\L e_2$ hold.
Therefore the Hasse quivers of $\tors(\Fl\L)$ and $\tors\L_0$ are given as follows, and the map ${\rm r}_{\m\,0} : \tors(\Fl\L) \to \tors\L_0$ is shown by the dashed arrows by applying Corollary \ref{cor-psilt-rpq} below.
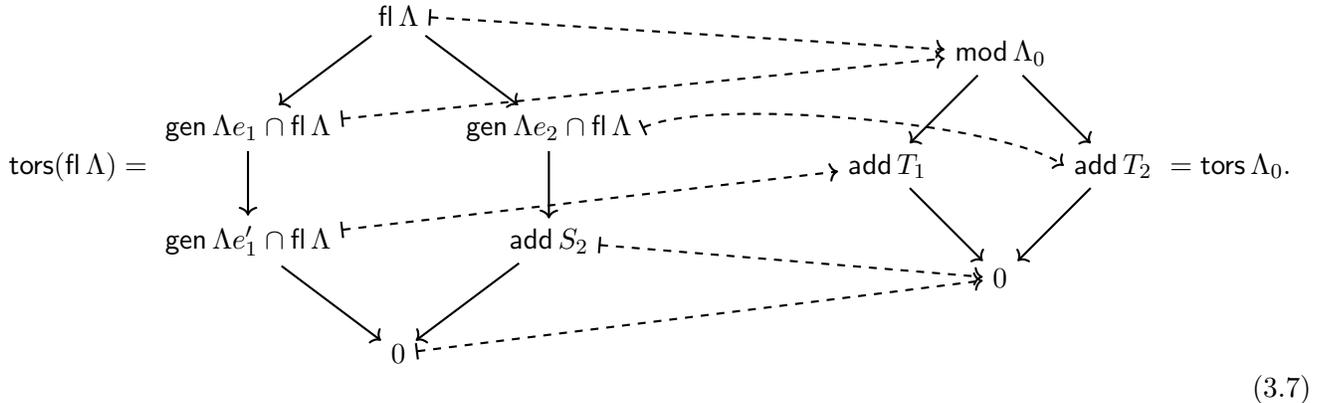
\begin{equation}\label{example-tors-fl}
\tors(\Fl\L) = 
\begin{tikzpicture}[baseline=-60]
\node(L)at(0,0){$\Fl\L$};
\node(M2P2)at(-2,-1.5){$\Fac\L e_1\cap\Fl\L$};
\node(P1M1)at(2,-1.5){$\Fac\L e_2\cap\Fl\L$};
\node(M2)at(-2,-3){$\Fac\L e'_1\cap\Fl\L$};
\node(M1)at(2,-3){$\add S_{2}$};
\node(0)at(0,-4.5){$0$};
\draw[thick, ->] (L)--(M2P2); 
\draw[thick, ->] (L)--(P1M1);
\draw[thick, ->] (M2P2)--(M2);
\draw[thick, ->] (P1M1)--(M1);
\draw[thick, ->] (M2)--(0);
\draw[thick, ->] (M1)--(0); 

\node(mod)at(8,-0.5){$\mod\L_0$};
\node[fill=white!30](T1)at(6.5,-2){$\add T_1$};
\node(T2)at(9.5,-2){$\add T_2$};
\node(00)at(8,-3.5){$0$};
\draw[thick, ->] (mod)--(T1);
\draw[thick, ->] (mod)--(T2);
\draw[thick, ->] (T1)--(00);
\draw[thick, ->] (T2)--(00);

\draw[thick, dashed, |->] (L)--(mod);
\draw[thick, dashed, |->] (M2P2)--(mod);
\draw[thick, dashed, |->] (M2)--(T1);
\draw[thick, dashed, bend left, distance=1cm, |->] (P1M1.east) to (T2.west);
\draw[thick, dashed, |->] (M1)--(00);
\draw[thick, dashed, |->] (0)--(00);

\begin{pgfonlayer}{bg}    
\end{pgfonlayer}
\end{tikzpicture}
=\tors\L_0.
\end{equation}
By Corollary \ref{cor-semi-local}, we have an isomorphism $\Phi_{\rm t}:\tors\L\simeq\bT_R^{\rm c}(\L)$ of posets.
One can easily check that the Hasse quiver of these posets are the following, where a vertex $(S,T)$ in the left quiver gives an element $(\gen S \cap \Fl\L, \gen T)\in\bT_R(\L)$.
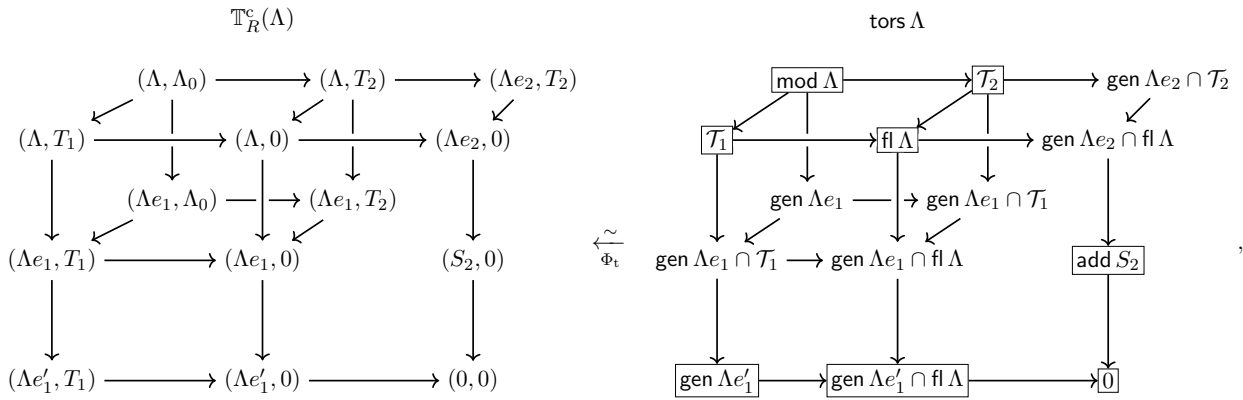
\begin{equation*}
\scalebox{0.8}{
$
\begin{tikzpicture}[baseline=-80]
\node(A)at(1.5,1){$\bT_R^{\rm c}(\L)$};
\node(LL0)at(0,0){$(\L, \L_0)$};
\node(LT1)at(-2,-1){$(\L, T_1)$};
\node(LT2)at(3,0){$(\L, T_2)$};
\node(L0)at(1.5,-1){$(\L, 0)$};
\node(L1L0)at(0,-2){$(\L e_1, \L_0)$};
\node(L1T1)at(-2,-3){$(\L e_1, T_1)$};
\node(L1T2)at(3,-2){$(\L e_1, T_2)$};
\node(L10)at(1.5,-3){$(\L e_1, 0)$};
\node(S1T1)at(-2,-5){$(\L e'_1, T_1)$};
\node(S10)at(1.5,-5){$(\L e'_1, 0)$};
\node(L2T2)at(6,0){$(\L e_2, T_2)$};
\node(L20)at(5,-1){$(\L e_2, 0)$};
\node(S20)at(5,-3){$(S_2, 0)$};
\node(00)at(5,-5){$(0, 0)$};
\draw[thick, ->] (LL0)--(LT1);
\draw[thick, ->] (LL0)--(LT2);
\draw[thick, ->] (LL0)--(L1L0);
\draw[white, line width=6pt] (LT1)--(L0);
\draw[thick, ->] (LT1)--(L0);
\draw[thick, ->] (LT1)--(L1T1);
\draw[thick, ->] (LT2)--(L0);
\draw[thick, ->] (LT2)--(L1T2);
\draw[thick, ->] (L1L0)--(L1T1);
\draw[thick, ->] (L1L0)--(L1T2);
\draw[white, line width=6pt] (L0)--(L10);
\draw[thick, ->] (L0)--(L10);
\draw[thick, ->] (L1T1)--(S1T1);
\draw[thick, ->] (L1T1)--(L10);
\draw[thick, ->] (L1T2)--(L10);
\draw[thick, ->] (L10)--(S10);
\draw[thick, ->] (S1T1)--(S10);
\draw[thick, ->] (LT2)--(L2T2);
\draw[white, line width=6pt] (L0)--(L20);
\draw[thick, ->] (L0)--(L20);
\draw[thick, ->] (L2T2)--(L20);
\draw[thick, ->] (L20)--(S20);
\draw[thick, ->] (S10)--(00);
\draw[thick, ->] (S20)--(00);
\end{tikzpicture}
\xleftarrow[\Phi_{\rm t}]{\sim}\ \ 
\begin{tikzpicture}[baseline=-80]
\node(A)at(1.5,1){$\tors\L$};
\node[draw,inner sep=2pt](LL0)at(0,0){$\mod\L$};
\node[draw,inner sep=2pt](LT1)at(-1.5,-1){$\cT_1$};
\node[draw,inner sep=2pt](LT2)at(3,0){$\cT_2$};
\node[draw,inner sep=2pt](L0)at(1.5,-1){$\Fl\L$};
\node(L1L0)at(0,-2){$\Fac \L e_1$};
\node(L1T1)at(-1.5,-3){$\Fac\L e_1 \cap\cT_1$};
\node(L1T2)at(3,-2){$\Fac\L e_1\cap\cT_1$};
\node(L10)at(1.5,-3){$\Fac\L e_1 \cap \Fl\L$};
\node[draw,inner sep=2pt](S1T1)at(-1.5,-5){$\Fac \L e'_1 $};
\node[draw,inner sep=2pt](S10)at(1.5,-5){$\Fac \L e'_1 \cap \Fl\L$};
\node(L2T2)at(6,0){$\Fac\L e_2 \cap \cT_2$};
\node(L20)at(5,-1){$\Fac\L e_2 \cap \Fl\L$};
\node[draw,inner sep=2pt](S20)at(5,-3){$\add S_2$};
\node[draw,inner sep=2pt](00)at(5,-5){$0$};
\draw[thick, ->] (LL0)--(LT1);
\draw[thick, ->] (LL0)--(LT2);
\draw[thick, ->] (LL0)--(L1L0);
\draw[white, line width=6pt] (LT1)--(L0);
\draw[thick, ->] (LT1)--(L0);
\draw[thick, ->] (LT1)--(L1T1);
\draw[thick, ->] (LT2)--(L0);
\draw[thick, ->] (LT2)--(L1T2);
\draw[thick, ->] (L1L0)--(L1T1);
\draw[thick, ->] (L1L0)--(L1T2);
\draw[white, line width=6pt] (L0)--(L10);
\draw[thick, ->] (L0)--(L10);
\draw[thick, ->] (L1T1)--(S1T1);
\draw[thick, ->] (L1T1)--(L10);
\draw[thick, ->] (L1T2)--(L10);
\draw[thick, ->] (L10)--(S10);
\draw[thick, ->] (S1T1)--(S10);
\draw[thick, ->] (LT2)--(L2T2);
\draw[white, line width=6pt] (L0)--(L20);
\draw[thick, ->] (L0)--(L20);
\draw[thick, ->] (L2T2)--(L20);
\draw[thick, ->] (L20)--(S20);
\draw[thick, ->] (S10)--(00);
\draw[thick, ->] (S20)--(00);
\end{tikzpicture},
$
}
\end{equation*}
Marked torsion classes in the right quiver are Serre subcategories.

Finally, we calculate $\torf\L$.
By Proposition \ref{prop-torsion-fl}, $\torf(\Fl\L)$ and $\torf\L_0$ have the following Hasse quivers.
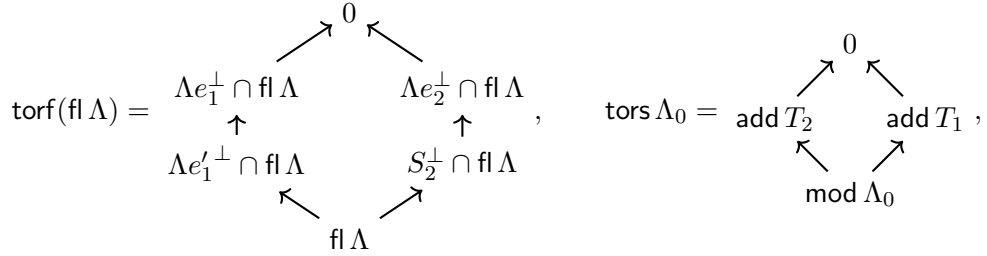
\begin{equation*}
\torf(\Fl\L) = \begin{tikzpicture}[baseline=-40]
\node(L)at(0,0){$0$};
\node(M2P2)at(-1.5,-1){$\L e_1^{\perp}\cap \Fl\L$};
\node(P1M1)at(1.5,-1){$\L e_2^{\perp}\cap\Fl\L$};
\node(M2)at(-1.5,-2){${\L e'_1}^{\perp} \cap \Fl\L$};
\node(M1)at(1.5,-2){$S_2^{\perp} \cap \Fl\L$};
\node(0)at(0,-3){$\Fl\L$};
\draw[thick, ->] (M2P2)--(L);
\draw[thick, ->] (P1M1)--(L);
\draw[thick, ->] (M2)--(M2P2);
\draw[thick, ->] (M1)--(P1M1);
\draw[thick, ->] (0)--(M2);
\draw[thick, ->] (0)--(M1);
\end{tikzpicture},
\qquad
\tors\L_0=\begin{tikzpicture}[baseline=0]
\node(mod)at(0,1){$0$};
\node(T1)at(-1,0){$\add T_2$};
\node(T2)at(1,0){$\add T_1$};
\node(0)at(0,-1){$\mod\L_0$};
\draw[thick, ->] (T1)--(mod);
\draw[thick, ->] (T2)--(mod);
\draw[thick, ->] (0)--(T1);
\draw[thick, ->] (0)--(T2);
\end{tikzpicture},
\end{equation*}
where for $\cT\in\tors(\Fl\L)$ or $\tors\L_0$, $\cT^{\perp}$ sits in the same position as in each quiver of \eqref{example-tors-fl}.
By Theorem \ref{thm-torsion-free-iso}, $\torf\L \simeq \torf(\Fl\L)\times\torf\L_0$ holds as posets.
The Hasse quiver of $\torf\L$ is described as follows, where a vertex $(S,T)$ gives an element $(S^\perp\cap\Fl\L, T^\perp)\in\torf(\Fl\L)\times\torf\L_0$, and the marked vertices belong to the image of $(-)^{\perp} : \bT_R(\L) \to \bF_R(\L)$.
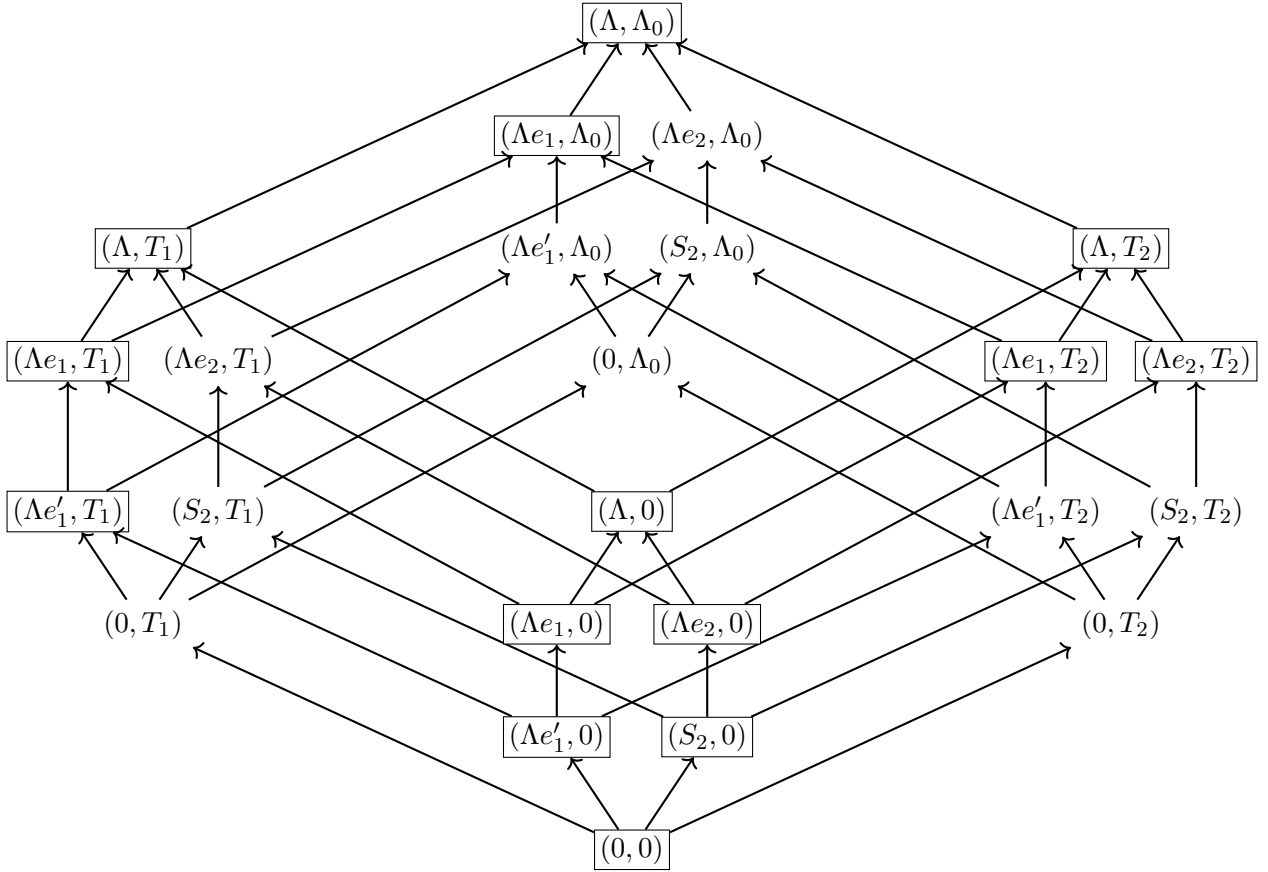
\begin{equation*}
\begin{tikzpicture}[baseline=-40]
\node[draw,inner sep=2pt](00)at(0,0){$(\L, \L_0)$};
\node(e20)at(1,-1.5){$(\L e_2, \L_0)$};
\node[draw,inner sep=2pt](e10)at(-1,-1.5){$(\L e_1, \L_0)$};
\node(s20)at(1,-3){$(S_2, \L_0)$};
\node(e1'0)at(-1,-3){$(\L e_1', \L_0)$};
\node(L0)at(0,-4.5){$(0, \L_0)$};

\draw[thick,->] (e20)--(00);
\draw[thick,->] (e10)--(00);
\draw[thick,->] (s20)--(e20);
\draw[thick,->] (e1'0)--(e10);
\draw[thick,->] (L0)--(s20);
\draw[thick,->] (L0)--(e1'0);

\node[draw,inner sep=2pt](0t2)at(-6.5,-3){$(\L, T_1)$};
\node(e2t2)at(-5.5,-4.5){$(\L e_2, T_1)$};
\node[draw,inner sep=2pt](e1t2)at(-7.5,-4.5){$(\L e_1, T_1)$};
\node(s2t2)at(-5.5,-6.5){$(S_2, T_1)$};
\node[draw,inner sep=2pt](e1't2)at(-7.5,-6.5){$(\L e'_1, T_1)$};
\node(Lt2)at(-6.5,-8){$(0, T_1)$};

\draw[thick,->] (e2t2)--(0t2);
\draw[thick,->] (e1t2)--(0t2);
\draw[thick,->] (s2t2)--(e2t2);
\draw[thick,->] (e1't2)--(e1t2);
\draw[thick,->] (Lt2)--(s2t2);
\draw[thick,->] (Lt2)--(e1't2);

\node[draw,inner sep=2pt](0t1)at(6.5,-3){$(\L, T_2)$};
\node[draw,inner sep=2pt](e2t1)at(7.5,-4.5){$(\L e_2, T_2)$};
\node[draw,inner sep=2pt](e1t1)at(5.5,-4.5){$(\L e_1, T_2)$};
\node(s2t1)at(7.5,-6.5){$(S_2, T_2)$};
\node(e1't1)at(5.5,-6.5){$(\L e'_1, T_2)$};
\node(Lt1)at(6.5,-8){$(0, T_2)$};

\draw[thick,->] (e2t1)--(0t1);
\draw[thick,->] (e1t1)--(0t1);
\draw[thick,->] (s2t1)--(e2t1);
\draw[thick,->] (e1't1)--(e1t1);
\draw[thick,->] (Lt1)--(s2t1);
\draw[thick,->] (Lt1)--(e1't1);

\node[draw,inner sep=2pt](0L)at(0,-6.5){$(\L, 0)$};
\node[draw,inner sep=2pt, fill=white!30](e2L)at(1,-8){$(\L e_2, 0)$};
\node[draw,inner sep=2pt, fill=white!30](e1L)at(-1,-8){$(\L e_1, 0)$};
\node[draw,inner sep=2pt](s2L)at(1,-9.5){$(S_2, 0)$};
\node[draw,inner sep=2pt](e1'L)at(-1,-9.5){$(\L e'_1, 0)$};
\node[draw,inner sep=2pt](LL)at(0,-11){$(0, 0)$};

\draw[thick,->] (e2L)--(0L);
\draw[thick,->] (e1L)--(0L);
\draw[thick,->] (s2L)--(e2L);
\draw[thick,->] (e1'L)--(e1L);
\draw[thick,->] (LL)--(s2L);
\draw[thick,->] (LL)--(e1'L);

\draw[thick,->] (0L)--(0t2);
\draw[thick,->] (e2L)--(e2t2);
\draw[thick,->] (e1L)--(e1t2);
\draw[thick,->] (e1'L)--(e1't2);
\draw[thick,->] (LL)--(Lt2);

\draw[thick,->] (0L)--(0t1);
\draw[thick,->] (e2L)--(e2t1);
\draw[thick,->] (e1L)--(e1t1);
\draw[thick,->] (s2L)--(s2t1);
\draw[thick,->] (LL)--(Lt1);

\draw[thick,->] (0t2)--(00);
\draw[thick,->] (e2t2)--(e20);
\draw[thick,->] (e1t2)--(e10);
\draw[thick,->] (s2t2)--(s20);
\draw[thick,->] (e1't2)--(e1'0);
\draw[thick,->] (Lt2)--(L0);

\draw[thick,->] (0t1)--(00);
\draw[thick,->] (e2t1)--(e20);
\draw[thick,->] (e1t1)--(e10);
\draw[thick,->] (s2t1)--(s20);
\draw[thick,->] (e1't1)--(e1'0);
\draw[thick,->] (Lt1)--(L0);

\begin{pgfonlayer}{bg}    
        \draw[thick, ->] (e1'L)--(e1't1);
        \draw[thick, ->] (s2L)--(s2t2);
\end{pgfonlayer}
\end{tikzpicture}
\end{equation*}
In Remark \ref{prop-comp-rpq}, we proved ${\rm r}_{\q, \mathfrak{r}} \circ {\rm r}_{\p, \q}(\cT) \supseteq {\rm r}_{\p, \mathfrak{r}}(\cT)$ for $\cT\in\tors(\Fl\L_\p)$.
Here we give one example such that the inclusion is strict.
\begin{exa}
Let $K$ be a field and $R=K[[x, y]]$ a ring of formal power series in two variables with the maximal ideal $\m=(x, y)$.
Let $\L$ be a Noetherian $R$-algebra defined as follows, and let $M_i$ be a $\L$-module defined as follows.
\[
\L=\left( \begin{array}{cc}
R & R \\
\mfm & R
\end{array}
\right),
\qquad
M_i=\left( \begin{array}{cc}
\m^i \\
\m^{i+1}
\end{array}
\right).
\]
Let $\cT=\bigcap_{i\geq 0}\Fac M_i$ and $(x)$ the ideal of $R$ generated by $x$.
Then the following statements hold.
\begin{enumerate}[{\rm (a)}]
\item For each integer $i\geq 0$, $M_i$ is a presilting $\L$-module.
\item We have ${\rm r}_{\m, 0}(\cT\cap \Fl\L)=0$ and ${\rm r}_{(x), 0}\circ {\rm r}_{\m, (x)}(\cT\cap \Fl\L) =\mod\L_0.$
\end{enumerate}
\end{exa}
\begin{proof}
(a)
For each integer $i\geq 0$, the following is a minimal projective presentation of $M_i$:
\[
(\L e_2)^{\oplus i} \xto{f_i} (\L e_1)^{\oplus i+1} \xto{{}^{\rm t}(x^i\, x^{i-1}y \, \cdots \, y^i)} M_i \to 0,
\qquad
f_i = \begin{pmatrix}
  y & -x       &   &     & \text{\huge{0}} \\
  \vdots & y  & -x & & \vdots                         \\
   &  &  \\
  \vdots &        & \ddots & -x      & \vdots            \\
  \text{\huge{0}} &  &  \dots & y  & -x
\end{pmatrix}\in \mathrm{Mat}_{i, i+1}(R).
\]
Let $\ov{\L}$ and $\Gamma$ be the algebras defined by
\[
\ov{\L}:=\L/\m\L=\left( \begin{array}{cc}
k & k \\
\m/\m^2 & k
\end{array}
\right)
\supseteq
\Gamma:=
\left( \begin{array}{cc}
k & 0 \\
\mfm /\mfm^2 & k
\end{array}
\right) \simeq k(1 \rightrightarrows 2).
\]
Then $(\Gamma e_2)^{\oplus i} \xto{f_i} (\Gamma e_1)^{\oplus i+1}$ is a minimal projective resolution of an indecomposable preprojective $\Gamma$-module, which is a presilting complex of $\Gamma$.
By \cite[Proposition 3.3(3)]{Aoki}, $(\ov{\L} e_2)^{\oplus i} \xto{f_i} (\ov{\L}e_1)^{\oplus i+1}$ is a presilting complex of $\ov{\L}$.
By \cite[Proposition 4.2(a)]{Kimura}, $(\L e_2)^{\oplus i} \xto{f_i} (\L e_1)^{\oplus i+1}$ is presilting.

\if()
We show that $\Hom_{\L}(\pi_i, M_i)$ is surjective.
The map $\Hom_{\L}(\pi_i, M_i)$ is surjective if and only if for any elements $f_1,\dots, f_i \in \m^{i+1}$, there exist $g_1, \dots, g_{i+1}\in\m^i$ such that
\begin{align}\label{eq-f-ygxg}
f_j = yg_j -xg_{j+1}
\end{align}
holds for each $j\in\{1,2,\dots, i\}$.
Since the equations (\ref{eq-f-ygxg}) are liner, it is sufficient to consider the following case: for one $j\in\{1, 2, \dots, i\}$, $f_j = x^{\ell}y^{m}$ for $i+1\leq \ell + m$ and $f_{k}=0$ if $k \neq j$.
In this case, the following $g_k$'s satisfy (\ref{eq-f-ygxg}).

If $\ell \geq i-j+2$, then let $g_{k}=0$ for $k\in\{1,2, \dots, j\}$ and $g_{j+k}=-x^{\ell-k}y^{m+k-1}$ for $k\in\{1, 2, \dots, i+1-j\}$.

If $\ell<i-j+2$, then let $g_k=0$ for $k\in\{j+1, j+2, \dots, i+1\}$ and $g_{j-k}=x^{\ell+k}y^{m-k-1}$ for $k\in\{0, 1, \dots, j-1\}$.
\fi

(b)
We have $\psi(\cT \cap \Fl\L) = \cT$.
In fact $\psi(\cT \cap \Fl\L) \subseteq \psi(\Fac M_i \cap \Fl\L)=\Fac M_i$ holds for each $i\geq 0$ by Theorem \ref{thm-tors-closure-stable}(a).
For a prime ideal $\p\neq \m$, an algebra $\L_\p$ is Morita equivalent to $R_\p$ and so $\tors(\Fl\L_\p)=\{0, \Fl\L_\p\}$.
Consider a $\L$-module \[N=\left( \begin{array}{cc}
R/(x) \\
\m/(x)
\end{array}
\right).\]
Since multiplying $y^i$ gives an isomorphism
\[
N=\left( \begin{array}{cc}
R/(x) \\
\m/(x)
\end{array}
\right) \xrightarrow[\sim]{\cdot y^i}
\left( \begin{array}{cc}
\m^i/x\m^{i-1} \\
\m^{i+1}/x\m^i
\end{array}
\right)
\in \gen M_i
\]
for each $i\geq 0$, we have $N\in\cT$.
Since $N_{(x)}\neq 0$ and $\tors(\Fl\L_{(x)})=\{0, \Fl\L_{(x)}\}$, we have ${\rm r}_{\m, (x)}(\cT\cap\Fl\L)=\Fl\L_{(x)}$.
Therefore ${\rm r}_{(x), 0}\circ {\rm r}_{\m, (x)}(\cT\cap\Fl\L) = {\rm r}_{(x), 0}(\Fl\L_{(x)}) = \mod\L_0$ holds.

By Lemma \ref{lem-cap-m-one} below, $X_0=0$ holds for any $X\in\cT$.
Therefore we have ${\rm r}_{\m, 0}(\cT\cap\Fl\L) = \cT_0= 0.$
\end{proof}
\begin{lem}\label{lem-cap-m-one}
Let $R=k[[x, y]]$ and $\m=(x, y)$.
Then the Krull dimension of any $R$-module in $\bigcap_{i\geq 0}\Fac(\m^i)$ is at most one.
\end{lem}
\begin{proof}
For an $R$-module $X$, let $X^\ast:=\Hom_R(X, R)$.
By \cite{AB}, there exists an exact sequence
\[
0 \to \Ext_R^1(\Tr X, R) \to X \xto{\rm ev} X^{\ast\ast} \to \Ext_R^2(\Tr X, R) \to 0,
\]
where $\Tr X$ denotes the transpose of $X$.
Since $\Ext_R^1(\Tr X, R)_0=0$, the Krull dimension of $\Ext_R^1(\Tr X, R)$ is at most one.
In the rest, we assume $X\in\bigcap_{i\geq 0}\Fac(\m^i)$.
We prove $X':=\Im({\rm ev})$ is zero, which implies the assertion.

Since the global dimension of $R$ is two and $X^{\ast\ast}$ is a second syzygy, we have $X^{\ast\ast}=R^{\oplus m}$ for some $m\geq 0$.
Let $\iota : X' \to R^{\oplus m}$ be the inclusion.
Fix an integer $i\geq 0$, and let $\pi_i : (\m^i)^{\oplus \ell} \to X'$ be a surjection.
Since the depth of $R$ is two, $\Ext_R^1(R/\m^i, R)=0$.
Thus $\iota\pi_i$ extends to a map $f_i : R^{\oplus \ell}\to R^{\oplus m}$.
\[
\begin{tikzcd}
0 \arrow[r] & (\m^i)^{\oplus \ell} \arrow[r] \arrow[d, "\pi_i"] & R^{\oplus \ell} \arrow[r] \arrow[d, "f_i"] & (R/\m^i)^{\oplus \ell} \arrow[r] & 0 \\
0 \arrow[r]  & X' \arrow[r, "\iota"] & R^{\oplus m}
\end{tikzcd}
\]
In particular, we have \[X' = \Im(\iota\pi_i) = \m^i \Im f_i \subseteq (\m^i)^{\oplus m}\]
for each $i\geq 0$.
We have $X' \subseteq \bigcap_{i\geq 0}(\m^i)^{\oplus m}=0$.
\end{proof}
\section{Silting theory and compatibility}\label{section-locally-finite}
In this section we study compatible elements in the case where the number of torsion classes of $\Fl\L_\p$ is finite for each $\p$.
Throughout this section, let $(R, \L)$ be a Noetherian algebra.
\subsection{Silting objects and bar operation}\label{subsection-silting-bar}
We study basic properties of torsion classes generated by silting modules and apply them to give a classification of torsion classes.
\begin{propdef}\label{propdef-stors}
For each silting $\L$-module $M$, $\gen M$ is a torsion class of $\mod\L$. We call a torsion class of this form a \emph{silting torsion class}.
Let
\begin{align*}
\stors\L:=\{\Fac M\mid M\in\siltm\L\}\ \mbox{ and }\ \stors(\Fl\L):=\{\cT \cap \Fl\L \mid \cT \in\stors\L \}.
\end{align*}
Clearly $(\stors\L,\subseteq)$ is a poset.
\end{propdef}
\begin{proof}
This is \cite[Lemma 3.1]{Kimura}. Note that $R$ is assumed to be complete local in Section 3 loc.cit., it is not used in the proof.
\end{proof}
We denote by $\ftors\L$ the set of all functorially finite torsion classes of $\mod\L$.
\begin{prop}\label{prop-silting-bijection}
For a Noetherian algebra $(R, \L)$, the following statements hold.
\begin{enumerate}[{\rm (a)}]
	\item For each presilting $\L$-module $M$, we have $\gen M\in \stors\L$.
	\item We have $\stors\L \subseteq \ftors\L$. The equality holds if $R$ is complete local.
	\item We have isomorphisms of posets
\[
\twosilt\L \longrightarrow \siltm\L \longrightarrow \stors\L\]
given by $P \mapsto H^0(P)$ and $M \mapsto \Fac M$ for $P\in\twosilt\L$, $M \in \siltm\L$.
\end{enumerate}
\end{prop}
\begin{proof}
(a)
Let $P\in\twopsilt\L$ such that $M=H^0(P)$, and let $f : \L\to P'$ be a left $(\add P)$-approximation of $\L$ in $\sK^{\rm b}(\proj\L)$. Then $Q:=P\oplus C(f)\in\twosilt\L$ holds by \cite[Lemma 2.4]{Kimura}. Since $H^0(P')\to H^0(C(f))$ is surjective, we have $\gen M = \gen H^0(Q)\in\stors\L$.

(b)
Any $\cT\in\tors\L$ is contravariantly finite by Proposition \ref{prop-tors-perp}. Also, for any $M\in\mod\L$, $\gen M$ is covariantly finite by \cite[Theorem 4.5, Proposition 4.6]{Auslander-Smalo}.
The latter assertion is \cite[Theorem 3.8]{Kimura}.

(c)
We shall prove this in Proposition \ref{prop-ap-twosilt-silt-stors}.
\end{proof}
Now we consider a triangle functor $(R/\mfm)\otimes_{R}(-) : \sK^{\rm b}(\proj\L) \to \sK^{\rm b}(\proj(\L/\mfm\L))$.
The following result  was proved in \cite[Theorem 4.3]{Kimura} for the special case where $R$ is complete local.
\begin{prop}\label{prop-silting-ftors}
Let $(R, \L)$ be a Noetherian algebra such that $(R, \m)$ is local.
Then we have the commutative diagram below such that vertical maps are isomorphisms of posets and three horizontal maps are embeddings of posets.
Moreover they are isomorphisms of posets if $\L$ is a semi-perfect ring.
\begin{align}\label{diagram-twosilt-sttilt-ftors}
\begin{tikzpicture}[baseline=0]
\node(twosiltL)at(0,2){$\twosilt\L$};
\node(twosiltLI)at(8,2){$\twosilt\L(\m)$};
\node(sttiltL)at(0,0){$\siltm\L$};
\node(sttiltLI)at(8,0){$\siltm\L(\m)$};
\node(ftorsL)at(0,-2){$\stors\L$};
\node(ftorsLI)at(8,-2){$\stors\L(\m)=\ftors\L(\m)$.};
\draw[thick, ->] (sttiltL)--(sttiltLI) node[midway, above]{$(R/\mfm)\otimes_{R}(-)$};
\draw[thick, ->] (ftorsL)--(ftorsLI) node[midway, above]{$(-)\cap\mod(\L/\mfm\L)$};
\draw[thick, ->] (twosiltL)--(twosiltLI) node[midway, above]{$(R/\mfm)\otimes_{R}(-)$};
\draw[thick, ->] (sttiltL)--(ftorsL) node[midway, left]{$\Fac(-)$} node[midway, right]{$\wr$};
\draw[thick, ->] (sttiltLI)--(ftorsLI) node[midway, left]{$\Fac(-)$} node[midway, right]{$\wr$};
\draw[thick, ->] (twosiltL)--(sttiltL) node[midway, left]{$H^{0}(-)$} node[midway, right]{$\wr$};
\draw[thick, ->] (twosiltLI)--(sttiltLI) node[midway, left]{$H^{0}(-)$} node[midway, right]{$\wr$};
\end{tikzpicture}
\end{align}
\end{prop}
\begin{proof}
The top horizontal map is well-defined by Lemma \ref{lem-twosilt-fac}.
The well-definedness of the middle horizontal map and the commutativity of the upper square follow from right-exactness of the functor $(R/\m)\otimes_R(-)$. The well-definedness of the bottom horizontal map and the commutativity of the lower square follow from $(\Fac M)\cap\mod\L(\m)=\Fac((R/\m)\otimes_RM)$.
The vertical maps are isomorphisms by Proposition \ref{prop-silting-bijection}(c), and the right bottom equality holds by Proposition \ref{prop-silting-bijection}(b).

The top horizontal map is an embedding of posets by \cite[Theorem 4.3(a)]{Kimura}.
The last assertion is \cite[Theorem 4.3(b)]{Kimura}. Note that $R$ is assumed to be complete local there, it is not used in the proof.
\end{proof}
\begin{rem}\label{rem-not-sp}
If $\L$ is not semi-perfect, then the horizontal maps in (\ref{diagram-twosilt-sttilt-ftors}) are not surjective in general.
Let $(R, \L)$ be a Noetherian algebra such that $(R, \m)$ is local.
Assume that $\L$ is ring indecomposable commutative and there are two maximal ideals $\m_1$, $\m_2$ of $\L$ containing $\m \L$ and $\L/\m\L \simeq \L/\m_1 \times \L/\m_2$ (for example $R=\bZ_{(5)}$ and $\L=\bZ[\sqrt{-1}]_{(5)}$).
By Proposition \ref{prop-comm-silt}, we have $\siltm \L = \{0, \L\}$.
On the other hand, $\ftors(\L/\m\L)=\tors(\L/\m \L)=\{\mod(\L/\m\L),\mod(\L/\m_1),\mod(\L/\m_2),0\}$.
\end{rem}
We give the following closure operation to obtain a torsion class.
\begin{propdef}\label{lem-Serre-extension}
Let $\cA$ be an abelian category, $\cS$ a subcategory of $\cA$ which is closed under subobjects and factor objects in $\cA$. Thus $\cS$ is an abelian category, and we denote by $\tors\cS$ the set of torsion classes of $\cS$.
For $\cC\in\tors\cS$ such that $\cC={}^{\perp_{\cS}}(\cC^{\perp_{\cS}})$, let
\[
\psi(\cC):=\{X\in\cA\ \mid \Fac X\cap\cS \subseteq \cC\}.
\]
Then we have the following statements.
\begin{enumerate}[{\rm (a)}]
	\item $\psi(\cC)\in\tors\cA$ holds. It is the maximum among torsion classes $\cT$ of $\cA$ satisfying $\cT\cap\cS \subseteq \cC$.
	\item $\psi(\cC)= {}^{\perp_{\cA}}\left( \cC^{\perp_{\cS}}  \right)$ and $\psi(\cC) \cap\cS = \cC$ hold.
\end{enumerate}
\end{propdef}
\begin{proof}
(a)
We show that $\psi(\cC)$ is a torsion class of $\cA$.
It is clear that $\psi(\cC)$ is closed under factors.
Let $0 \to X \to Y \to Z \to 0$ be an exact sequence in $\cA$ with $X,Z\in\psi(\cC)$.
We show that each factor $Y/Y^{\prime}$ belongs to $\cC$ if it belongs to $\cS$.
We have a short exact sequence $0 \to (X + Y^{\prime})/Y^{\prime} \to Y/Y^{\prime} \to Y/(X+Y^{\prime}) \to 0$.
Since the first term belongs to $\Fac X \cap \cS$ and the third term belongs to $\Fac Z \cap \cS$, the middle term belongs to $\cC$.
By the definition of $\psi(\cC)$, the maximum property holds.

(b)
We show the equality $\psi(\cC)={}^{\perp_{\cA}}\left( \cC^{\perp_{\cS}} \right)$.
By assumption, $\cS \cap {}^{\perp_{\cA}}\left( \cC^{\perp_{\cS}} \right) = \cC$ holds.
By the maximum property, we have $\psi(\cC)\supseteq {}^{\perp_{\cA}}\left( \cC^{\perp_{\cS}} \right)$.
Let $X\in\psi(\cC)$, $Y\in \cC^{\perp_{\cA}} \cap \cS$ and $f : X \to Y$.
We have $\Im(f) \in \Fac X \cap \cS \subseteq \cC$.
So the inclusion map $\Im(f) \to Y$ is zero, and $f=0$.
This implies that $X\in{}^{\perp_{\cA}}\left( \cC^{\perp_{\cS}} \right)$.
We complete the proof.
\end{proof}
Let $\cA=\mod\L$ and $\cS=\Fl\L$.
By Proposition \ref{prop-tors-perp}, each $\cC\in\tors(\Fl\L)$ satisfies $\cC={}^{\perp_{\cS}}(\cC^{\perp_{\cS}})$.
So by applying Definition-Proposition \ref{lem-Serre-extension}, we obtain an injective map
\[\psi : \tors(\Fl\L)\to\tors\L,\ \cC\mapsto \psi(\cC).\]
We leave the reader to check that this is bijective if and only if $\dim R=0$.

If $\cT\in\stors(\Fl\L)$, then $\psi(\cT)$ can be described explicitly as follows.
\begin{thm}\label{thm-tors-closure-stable}
Let $(R, \L)$ be a Noetherian algebra.
\begin{enumerate}[{\rm (a)}]
\item For each presilting $\L$-module $M$, we have \[\gen M = \psi(\gen M \cap \Fl\L).\]
\item The map $(-) \cap \Fl\L : \stors\L \to \stors(\Fl\L)$ is an isomorphism of posets whose inverse is given by $\psi : \stors(\Fl\L) \to \stors \L$.
\end{enumerate}
\end{thm}
We need the following lemma to prove Theorem \ref{thm-tors-closure-stable}.
\begin{lem}\cite{Guralnick}\label{lem-lifting-number}
Let $(R, \m)$ be a commutative local Noetherian ring, $(R, \L)$ a Noetherian algebra.
Then for each $M, N\in\mod\L$, there  exists an integer $e\geq 0$, called a lifting number of $(M, N)$, such that for any $i>0$ and any
$f\in\Hom_\L(M/\mfm^{e+i}M, N/\mfm^{e+i}N)$, there exists $g\in\Hom_\L(M, N)$ such that $f$ and $g$ induce the same morphism from $M/\mfm^iM$ to $N/\mfm^iN$.	
\end{lem}
We are ready to prove the theorem.
\begin{proof}[Proof of Theorem \ref{thm-tors-closure-stable}]
The assertion (b) immediately follows from (a) and Definition-Proposition \ref{lem-Serre-extension}(b).
By Proposition \ref{prop-silting-bijection}(a), we may assume that $M$ is a silting module.
Clearly $\gen M \subseteq \psi(\gen M \cap\Fl\L)$ holds.

(i) Let $\cT=\Fac M$ and $\cC=\cT \cap \Fl\L$. We first prove $\psi(\cC) \subseteq \cT$ in the case where $(R, \m)$ is a local ring.
Let $X\in\psi(\cC)$.
There exists  a surjective morphism $a : M_{0} \to X/\mfm X$ with $M_{0}$ in $\add M$.
We show that for any $i>0$, there exists a surjective morphism $M_{0} \to X/\mfm^{i} X$.
Take a short exact sequence
\[
0 \to \mfm X/ \mfm^{i}X \to X/\mfm^{i}X \xto{\pi} X/\mfm X \to 0.
\]
We have $\mfm X/ \mfm^{i}X \in \Fac X\,\cap\,\Fl\L \subseteq \cC$.
Since $M$ is a Ext-projective module in $\cT$, there exists a morphism $b : M_{0} \to X/\mfm^{i}X$ with $\pi b =a$.
This implies $\Im b + \mfm(X/\mfm^{i}X) = X/\mfm^{i}X$.
By Nakayama's lemma, $b$ is surjective.

Let $e$ be a lifting number of $(M_{0}, X)$.
By the above argument, there exists a surjective morphism $f : M_{0} \to X/\mfm^{e+1}X$.
Since $e$ is a lifting number, there exists a morphism $g : M_{0} \to X$ such that $f$ and $g$ induce the same morphism $\ov{f} = \ov{g} : M_{0}/\mfm M_{0} \to X/\mfm X$.
This means that $\Im g+\mfm X = X$.
By Nakayama's lemma, $g$ is surjective.
Therefore $X$ belongs to $\cT$ and we have $\psi(\cC) \subseteq \cT$.

(ii) We go back to the case where $R$ is not necessary local.
Let $X\in\psi(\gen M \cap \Fl\L)$.
So $\Fac X \cap \Fl\L \subseteq \Fac M$ holds.
For each maximal ideal $\m$ of $R$, $\Fac X_\m \cap \Fl\L_\m \subseteq \Fac M_\m$ holds.
Since $M_\m$ is a silting $\L_\m$-module, we have $X_\m\in\Fac M_\m$ by case (i).
By Lemma \ref{lem-local-gen}(b), $X\in\gen M$ holds.
\end{proof}
Now we apply Theorem \ref{thm-tors-closure-stable} to study properties of ${\rm r}_{\p,\q} : \tors(\Fl\L_\p) \to \tors(\Fl\L_\q)$.
We can calculate ${\rm r}_{\p,\,\q}$ for elements in $\stors(\Fl\L_\p)$ as follows. The following result will be used in Subsection \ref{subsection-comp-silt-finite}, and (b) improves Proposition \ref{prop-comp-rpq}(b).
\begin{cor}\label{cor-psilt-rpq}
Let $(R, \L)$ be a Noetherian algebra.
\begin{enumerate}[\rm(a)]
\item Let $M$ be a presilting $\L$-module.
For prime ideals $\p \supseteq \q$, we have
\[
{\rm r}_{\p,\q}(\gen M_{\p} \cap \Fl\L_\p)=\gen M_{\q} \cap \Fl\L_\q.
\]
\item
For each $\cT\in\stors(\Fl\L_\p)$, we have ${\rm r}_{\q,\mathfrak{r}} \circ {\rm r}_{\p,\q}(\cT) = {\rm r}_{\p, \mathfrak{r}}(\cT)$.
\end{enumerate}
\end{cor}
\begin{proof}
We have $\gen M_\p = \psi(\gen M_\p \cap \Fl\L_\p)$ since $\gen M_\p \in \stors\L_\p$ and by Theorem \ref{thm-tors-closure-stable}.
By a direct calculation, we have $(\gen M_\p)_\q=\gen M_\q$.
Thus the assertion holds.
\end{proof}
\subsection{Silting finiteness}\label{subsection-locally-finite}
In this subsection we study finiteness of two-term silting complexes of Noetherian algebras.
We start with recalling the following basic result in tilting theory.
\begin{prop}\label{thm-DIJ}\cite{DIJ}
Let $\L$ be a finite dimensional algebra over a field.
Then following statements are equivalent.
\begin{enumerate}[{\rm (i)}]
\item
The set $\siltm \L$ is finite.
\item
The set $\ftors\L$ of functorially finite torsion classes of $\mod\L$ is finite.
\item
The set $\twosilt\L$ is finite.
\item
The set $\tors\L$ is finite.
\item
$\tors\L = \ftors\L$ holds.
\end{enumerate}
In this case we call $\L$ \emph{$\tau$-tilting finite}.
\end{prop}
\begin{proof}
By results of \cite{Adachi-Iyama-Reiten}, (i),(ii),(iii) are equivalent, and by \cite[Theorem 3.8]{DIJ}, (v) is also equivalent. Thus (ii) (and (v)) implies (iv), and the converse is clear.
\end{proof}
The main theorem of this subsection is a generalization below of Proposition \ref{thm-DIJ} for Noetherian algebras.
In fact, if $R$ is a field, then Theorem \ref{thm-silting-finite} is specialized to Proposition \ref{thm-DIJ}.
\begin{thm}\label{thm-silting-finite}
Let $(R, \L)$ be a Noetherian algebra.
We consider the following statements. \begin{enumerate}[{\rm (i)}]
\item
The set $\siltm\L$ of additive equivalence classes of silting $\L$-modules is finite.
\item
The set $\stors\L$ is finite (or equivalently, $\stors(\Fl\L)$ is finite by Theorem \ref{thm-tors-closure-stable}).
\item The set $\twosilt\L$ is finite.
\item
The set $\tors(\Fl\L)$ of torsion classes of $\Fl\L$ is finite.
\item
$\supp_R\L \cap \MSpec R$ is finite, and 
$\L(\m)$ is $\tau$-tilting finite for any $\m\in\MSpec R$.
\item
$\tors(\Fl\L) = \stors(\Fl\L)$ holds.
\item
$\stors\L = \{ \cT\in\tors \L \mid \psi(\cT\cap \Fl\L) = \cT \}$ holds.
\end{enumerate}
Then we have the following implications: ${\rm (i)} \Leftrightarrow {\rm (ii)} \Leftrightarrow {\rm (iii)}  \Leftarrow {\rm (iv)} \Leftrightarrow {\rm (v)} \Leftarrow {\rm (vi)} \Leftrightarrow {\rm (vii)}$.
Moreover if $\L$ is semi-perfect, then all statements are equivalent.
\end{thm}
\begin{rem}
The implications ${\rm (iii)} \Rightarrow {\rm (iv)}$  and ${\rm (v)} \Rightarrow {\rm (vi)}$ of Theorem \ref{thm-silting-finite} do not hold in general.
For instance, if $R$ is ring indecomposable and $\Spec R$ is infinite (e.g.\ $R=\bZ$), then the $R$-algebra $R$ satisfies (iii) by Proposition \ref{prop-comm-silt}, but does not satisfy (iv). Also the Noetherian algebra in Remark \ref{rem-not-sp} satisfies (v), but does not satisfy (vi).

Also notice that the following conditions appearing in Proposition \ref{thm-DIJ} are rarely satisfied.
\begin{enumerate}
\item[(viii)]
The set $\tors\L$ is finite.
\item[(ix)]
$\tors\L = \ftors\L$ holds.
\end{enumerate}
In fact, if $\Spec R$ is infinite, then the $R$-algebra $R$ satisfies all conditions (i)--(vii) by Proposition \ref{prop-comm-silt}, but does not satisfy (viii)--(ix) by Corollary \ref{cor-local-sp-closed}.
\end{rem}
\begin{proof}[Proof of Theorem \ref{thm-silting-finite}]
Let ($\ast$) be one of (i) to (vii).
If $\L \simeq \L_1 \times \L_2$ as $R$-algebras, then $\L$ satisfies ($\ast$) if and only if $\L_1$ and $\L_2$ satisfy ($\ast$).
Thus we can assume that $\L$ is ring indecomposable.
In particular, $R$ is ring indecomposable too.
By replacing $R$ by $R/\ann_R\L$, we can assume that $\L$ is a faithful $R$-module.

(A) We prove the former statement.

(i)$\Leftrightarrow$(ii)$\Leftrightarrow$(iii) Immediate from Proposition \ref{prop-silting-bijection}.

(iv)$\Rightarrow$(ii) This is clear.

(iv)$\Leftrightarrow$(v) Immediate from Proposition \ref{prop-torsion-fl}.

(vi)$\Rightarrow$(v)
Since $R$ is ring indecomposable and $\L$ is a faithful $R$-module, $R$ is a local ring by Proposition \ref{local} below.
In particular, $\supp_R\L \cap \MSpec R$ is a finite set.

To show the latter assertion, let $\m\in\MSpec R$ and $\cT\in\tors\L(\m)$.
By Proposition \ref{prop-torsion-fl}, there exists $\cU\in\tors(\Fl\L)$ such that $\cT=\cU\cap\mod\L(\m)$. By (vi), there exists $M\in\siltm\L$ such that $\cU=\Fac M\cap\Fl\L$. Then $\cT=\Fac M\cap\mod\L(\m)=\Fac(M/\m M)\in\ftors\L(\m)$.

(vi)$\Rightarrow$(vii)
By Theorem \ref{thm-tors-closure-stable}, $\psi(\Fac M\cap\Fl\L)=\Fac M$ holds for any $M\in\siltm\L$.
Let $\cT\in\tors\L$ such that  $\psi(\cT\cap\Fl\L) = \cT$.
By (vi), there exists $M\in\siltm \L$ such that $\Fac M \cap \Fl\L = \cT \cap \Fl\L$.
Then we have $\Fac M  = \psi(\Fac M \cap \Fl\L) = \psi(\cT \cap \Fl\L)  =\cT$. Thus (vii) holds.

(vii)$\Rightarrow$(vi) For each $\cU \in \tors(\Fl\L)$, we have $\psi(\cU) \cap \Fl\L = \cU$ and $\psi(\psi(\cU)\cap\Fl\L) = \psi(\cU)$.
Thus there exists $M\in\siltm\L$ such that $\Fac M = \psi(\cU)$ by (vii).
We have $\Fac M \cap \Fl\L = \psi(\cU) \cap \Fl\L = \cU$ as desired.

(B) To prove the latter statement, we assume that $\L$ is semi-perfect.
By Proposition \ref{prop-sp-local} and since $R$ is ring indecomposable, $R$ is a local ring.
Let $\m$ be the maximal ideal of $R$.

(i)$\Rightarrow$(v) By Proposition \ref{prop-silting-ftors}, $\siltm\L\simeq\siltm\L(\m)$ holds. Thus the assertion is immediate.

(v)$\Rightarrow$(vi) By Proposition \ref{thm-DIJ}, we have $\tors\L(\m)=\stors\L(\m)$.
Let $\cT\in\tors(\Fl\L)$.
By Proposition \ref{prop-silting-ftors}, there is $M\in\siltm\L$ such that $\Fac M \cap \mod\L(\m) = \Fac(M/\m M)=\cT\cap\mod\L(\m)$ holds.
By Propositions \ref{prop-torsion-fl}, $\Fac M \cap \Fl\L = \cT$ holds.
Thus (vi) holds.
\end{proof}
We complete the proof of Theorem \ref{thm-silting-finite} by showing the following result.
\begin{prop}\label{local}
Let $(R, \L)$ be a Noetherian algebra such that $R$ is ring indecomposable and $\L$ is a faithful $R$-module.
If $\tors(\Fl\L)=\stors(\Fl\L)$ holds, then $R$ is a local ring.
\end{prop}
\begin{proof}
Let $\m\in\MSpec R$.
Since $\Fl\L_\m$ is a torsion class of $\Fl\L$, there exists $M\in\siltm\L$ such that $\Fac M \cap \Fl\L = \Fl\L_\m$.
We have $\Fac M_\m = \psi(\Fac M_\m \cap \Fl\L_\m) = \psi(\Fl\L_\m) = \mod\L_\m$ by Theorem \ref{thm-tors-closure-stable}.
In particular, $M_\m$ is a faithful $R_\m$-module.
Thus we have $\supp_R M \supseteq \Spec R_\m$ holds.
For each $\n \in\MSpec R\setminus\{\m\}$, $\gen M \cap \Fl\L_\n=0$ and thus $M_\n=0$.
Therefore $\supp_R M = \Spec R_\m$ holds.
By applying Lemma \ref{lem-Spec-conn}(b) to $\Spec R_\m$, we have $\Spec R_\m = \Spec R$.
Namely, $R$ is a local ring.
\end{proof}
\subsection{Compatible elements for silting finite case}\label{subsection-comp-silt-finite}
For two posets $X, Y$, the set $\Hom_{\rm poset}(X, Y)$ of morphisms of posets forms a poset again: for two morphisms $f, g : X\to Y$, we write $f\leq g$ if $f(x) \leq g(x)$ holds for any $x\in X$.

In this subsection, we show the following theorem.
\begin{thm}\label{thm-r-isom}
Let $(R, \L)$ be a Noetherian algebra such that $R$ is ring indecomposable and the following conditions are satisfied.
\begin{itemize}
	\item[{\rm (a)}] $(-)_\p : \twosilt\L \to \twosilt \L_\p$ is an isomorphism of posets for any $\p\in\Spec R$.
	\item[{\rm (b)}] $\stors(\Fl\L_\p)=\tors(\Fl\L_\p)$ holds for any $\p\in\Spec R$.
\end{itemize}
Then $(R, \L)$ is compatible, and we have an isomorphism of posets
\[
\tors\L \simeq \Hom_{\rm poset}(\Spec R, \twosilt\L).
\]
\end{thm}
We use the following observations.
\begin{prop}\label{prop-element-image}
Let $(R, \L)$ be a Noetherian algebra and $\cX=(\cX^\p)_\p\in\bT_R(\L)$.
Then $\cX$ belongs to $\Im\Phi_{\rm t}$ if it satisfies the following condition: for any $\q\in\Spec R$, there exists $M\in\mod\L$ such that $\cX^\q \subseteq \gen M_\q$ and $\gen M_{\p} \cap \Fl\L_{\p} \subseteq \cX^{\p}$ for any $\p\in V(\q)\setminus\{\q\}$.
\end{prop}
\begin{proof}
Let $\cU:=\Psi_{\rm t}(\cX)$.
It is enough to show that $\cX^\q \subseteq \cU_\q$ holds for any $\q\in\Spec R$.
By Lemma \ref{lem-local-global}(b), any module in $\cX^\q$ can be written as $X_\q$ for some $X\in\mod\L$ such that $\supp_R X\subseteq V(\q)$.
Let $f : M' \to X$ be a right $(\add M)$-approximation of $X$ and $Y:=\Im f$.
Then $f_\q\circ -: \Hom_{\L_\q}(M_\q, M'_\q) \to \Hom_{\L_\q}(M_\q, X_\q)$ is surjective by \eqref{iso-localization}, and thus $f_\q$ is a right $(\add M_\q)$-approximation of $X_\q$. Since $X_\q\in \cX^\q \subseteq \gen M_\q$, we have $Y_\q = X_\q$.
By replacing $X$ by $Y$, we may assume $X\in\gen M$.
Then for each $\p\in V(\q)\setminus\{\q\}$, $X_{\p}\in\gen M_{\p}\subseteq \psi_{\p}(\cX^{\p})$ holds.
This implies $X\in \cU$ and $X_\q\in\cU_\q$.
\end{proof}
\begin{prop}\label{prop-local-iso}
Assume that $R$ is ring indecomposable.
Then for a Noetherian algebra $(R, \L)$, the following {\rm (a)} and {\rm (a$'$)} are equivalent.
\begin{enumerate}
	\item[{\rm (a)}] $(-)_\p : \twosilt\L \to \twosilt \L_\p$ is an isomorphism of posets for any $\p\in\Spec R$.
	\item[{\rm (a$'$)}] The following statements hold.
		\begin{enumerate}[{\rm (i)}]
		\item $(-)_\p : \twosilt\L \to \twosilt\L_\p$ is surjective for any $\p\in\Spec R$.
		\item $(-)_\q : \twosilt\L_\p \to \twosilt\L_\q $ is an isomorphism of posets for any pair $\p \supseteq \q$ in $\Spec R$.
		\end{enumerate}
\end{enumerate}
\end{prop}
For prime ideals $\p \supseteq \q$, we denote by $(-)^\p_\q$ the functor $(-)_\q : \mod\L_\p \to \mod\L_\q$ when we need to clarify the domain.
\begin{proof}
By Proposition \ref{prop-presilt-local}, we have well-defined morphisms $(-)_\p:\twosilt\L \to \twosilt\L_\p$ and $(-)^{\p}_\q : \twosilt\L_\p \to \twosilt\L_\q$ of posets.
Assume that (a) holds.
Then clearly, (a$'$)(i) holds.
Since $(-)_\q^\p \circ (-)_\p = (-)_\q$ holds as maps from $\twosilt\L$ to $\twosilt\L_\q$, (a$'$)(ii) holds.

Conversely, assume that (a$'$) holds.
Fix $\p\in\Spec R$ and $X, Y\in\twosilt\L$.
It suffices to show that $X_\p \geq Y_\p$ implies $X \geq Y$.
Let $\cS:=\{\q\in\Spec R\mid X_\q\geq Y_\q\}$. Then $\p\in\cS$, and both $\cS$ and $\cS^{\rm c}$ are specialization closed by (a$'$)(ii).
Then $\cS=\Spec R$ holds by Lemma \ref{lem-Spec-conn}(b). Thus $X\geq Y$ holds by Proposition \ref{prop-presilt-local}(b)(i).
\end{proof}
We are ready to show Theorem \ref{thm-r-isom}.
\begin{proof}[Proof of Theorem \ref{thm-r-isom}]
By Corollary \ref{cor-psilt-rpq}, for prime ideals $\p \supseteq \q$ the following diagram commutes:
\begin{equation}\label{diagram-twosilt-r}
\begin{tikzcd}
\twosilt\L \arrow[d, equal] \arrow[r, "(-)_\p"] & \twosilt \L_\p \arrow[d, "(-)_\q^\p"'] \arrow[rrr, "\Fac H^0(-) \cap \Fl\L_\p"] &&& \tors(\Fl \L_\p) \arrow[d, "{\rm r}_{\p, \q}"]   \\
\twosilt\L \arrow[r, "(-)_\q"] & \twosilt \L_\q \arrow[rrr, "\Fac H^0(-) \cap \Fl\L_\q"] &&& \tors(\Fl \L_\q)
\end{tikzcd}
\end{equation}
By Proposition \ref{prop-local-iso}, the left horizontal maps and the vertical map $(-)_\q^\p$ are isomorphisms of posets.
So are the right horizontal maps since they are compositions
\[\twosilt\L_\p\xrightarrow{{\rm\ref{prop-silting-bijection}(c)}}\stors\L_\p\xrightarrow{{\rm\ref{thm-tors-closure-stable}(b)}}\stors(\Fl\L_\p)\stackrel{{\rm(b)}}{=}\tors(\Fl\L_\p)\]
of isomorphisms of posets.
Therefore all maps in \eqref{diagram-twosilt-r}) are isomorphisms of posets.

We denote by $g_\p:\twosilt\L\simeq\tors(\Fl\L_\p)$ the upper isomorphism of posets in \eqref{diagram-twosilt-r}. By the commutativity of \eqref{diagram-twosilt-r}, it is clear that an isomorphism $\bT^{\rm c}_R(\L)\simeq\Hom_{\mathsf{poset}}(\Spec R, \twosilt\L)$ of posets is given by $\bT^{\rm c}_R(\L)\ni(\cX^\p)_\p\mapsto (g_\p^{-1}(\cX^\p))_{\p}$.

It remains to prove that $(R,\L)$ is compatible. Fix any $\cX=(\cX^\p)_\p\in\bT^{\rm c}_R(\L)$.
To prove $\cX\in\Im\Phi_{\rm t}$, it suffices to check the condition in Proposition \ref{prop-element-image}.
For each $\q\in\Spec R$, there exists 
 $M\in\siltm\L$ such that $\cX^\q=\gen M_\q\cap\Fl\L_\q$ by our assumptions.
For any $\p\in V(\q)\setminus\{\q\}$, since $\cX$ is compatible, we have
\[{\rm r}_{\p,\q}(\gen M_{\p}\cap\Fl\L_{\p})\stackrel{\rm\ref{cor-psilt-rpq}}{=}\gen M_{\q}\cap\Fl\L_{\q}=\cX^\q\subseteq{\rm r}_{\p,\q}(\cX^{\p}).\]
Since ${\rm r}_{\p,\q}$ is an isomorphism, $\gen M_{\p}\cap\Fl\L_{\p}\subseteq\cX^{\p}$ holds. Thus the condition in Proposition \ref{prop-element-image} holds.
\end{proof}
In the rest of this section, we give a family of examples of $(R,\L)$ satisfying the conditions in Theorem \ref{thm-r-isom}.
We start with giving the following lemma.
\begin{lem}\label{lem-semi-perfect}\
Assume that $(R, \m)$ is local.
For $\L$-module $P\in\proj\L$, $\End_{\L}(P)$ is local if and only if $\End_{\L}(P/\mfm P)$ is local.
\end{lem}
\begin{proof}
Clearly $\End_\L(P/\mfm P)\simeq \End_\L(P)/\m \End_\L(P)$ holds (see \cite[Lemma 4.1]{Kimura}).
Since $\End_\L(P)$ is a Noetherian $R$-algebra, $\m\End_\L(P)\subseteq \rad\End_\L(P)$ holds by \cite[(5.22) Proposition]{Curtis-Reiner}.
Therefore we have the assertion.
\end{proof}
Let $A$ be a finite dimensional algebra over a field $k$.
A simple $A$-module $S$ is said to be \emph{$k$-simple} if $\End_A(S) \simeq k$ holds.
This is equivalent to that $S_K=K\otimes_k S$ is a simple $A_K$-module for any field extension $K$ of $k$ (see \cite[(3.43) Theorem]{Curtis-Reiner}).
For instance, if $k$ is an algebraically closed field, or $A=kQ/I$ for a finite quiver $Q$ and an admissible ideal $I$, then all simple $A$-modules are $k$-simple.
\begin{prop}\label{prop-ftors-bij}
Let $A$ be a finite dimensional $k$-algebra over a field $k$, $R$ a commutative Noetherian ring which contains $k$, and $\L:=R\otimes_k A$.
Assume that any simple $A$-module is $k$-simple.
\begin{enumerate}[\rm(a)]
\item $\L_\p$ is a semi-perfect ring for any $\p\in\Spec R$.
\item For any $\p \supseteq \q$ in $\Spec R$, the following maps are isomorphisms of posets: 
\[
\twosilt A \xto{R\otimes_k (-)} \twosilt \L \xto{(-)_\p} \twosilt \L_\p \xto{(-)_\q} \twosilt\L_\q.
\]
\end{enumerate}
\end{prop}
\begin{proof}
(a) Let $A=\bigoplus_{i=1}^{\ell}P_i$ be an indecomposable decomposition of $A$ as a left $A$-module.
By Lemma \ref{lem-k-K-proj}, each $\kappa(\p)\otimes_kP_i$ is indecomposable. By Lemma \ref{lem-semi-perfect}, $\L_\p$ is a semi-perfect ring.

(b) For each prime ideal $\p$ of $R$, by Theorem \ref{prop-base-silt} and Propositions \ref{prop-presilt-local} and \ref{prop-quo-silt}, we have the following morphisms of posets:
\begin{align}\label{seq-R-p-kp}
\twosilt A \xto{R\otimes_k (-)} \twosilt \L \xto{(-)_\p}\twosilt \L_\p \xto{\kappa(\p) \otimes_{R_\p}(-)} \twosilt\L(\p).
\end{align}
Since $\L_\p$ is a semi-perfect ring by (a), the last map is an isomorphism by Proposition \ref{prop-silting-ftors}.
Since $\L(\p)=\kappa(\p)\otimes_kA$, Theorem \ref{thm-twosilt-k-K} implies that the composite of these three maps is an isomorphism.
Thus the composite $(-)_\p \circ (R\otimes_k(-))=R_\p \otimes_k (-):\twosilt A\to\twosilt\L_\p$ of the left map and the middle map is an isomorphism.
In particular, $(-)_\p:\twosilt\L\to\twosilt\L_\p$ is surjective.
For prime ideals $\p \supseteq \q$, we have the following commutative diagram
\[
\begin{tikzcd}
\twosilt A \arrow[d, equal] \arrow[rr, "R_\p \otimes_k(-)"] && \twosilt \L_\p \arrow[d, "(-)_\q"] \\
\twosilt A \arrow[rr, "R_\q \otimes_k(-)"] && \twosilt \L_\q
\end{tikzcd}
\]
where two horizontal maps are isomorphisms.
Thus the right vertical map $(-)_\q$ is also an isomorphism.
Applying Proposition \ref{prop-local-iso}(a$'$)$\Rightarrow$(a), the map $(-)_\p$ in \eqref{seq-R-p-kp} is an isomorphism, and so is $R\otimes_k(-)$.
\end{proof}

\begin{cor}\label{cor-r-isom}
Let $A$ be a finite dimensional algebra over a field $k$.
Assume that $A$ is $\tau$-tilting finite and any simple $A$-module is $k$-simple.
Then for arbitrary commutative Noetherian ring $R$ which contains the field $k$, we have an isomorphism of posets
\[
\tors(R\otimes_k A) \simeq \Hom_{\rm poset}(\Spec R, \tors A).
\]
\end{cor}
\begin{proof}
First, we assume that $R$ is ring indecomposable.
Let $\L=R\otimes_k A$.
We see that $\L$ satisfies the assumptions of Theorem \ref{thm-r-isom}.
By Proposition \ref{prop-ftors-bij}, $\L_\p$ is semi-perfect and the following two maps are isomorphisms of posets for any $\p\in\Spec R$:
\[
\twosilt A \xto{R\otimes_k (-)} \twosilt \L \xto{(-)_\p} \twosilt \L_\p.
\]
In particular, $\L$ satisfies (a) in Theorem \ref{thm-r-isom}.
Since $A$ is $\tau$-tilting finite, $\twosilt\L_\p \simeq \twosilt A$ is a finite set.
By applying Theorem \ref{thm-silting-finite}(iii) $\Rightarrow$ (vi) for the semi-perfect ring $\L_\p$, we have $\stors(\Fl\L_\p)=\tors(\Fl\L_\p)$, that is,  $\L$ satisfies (b) in Theorem \ref{thm-r-isom}.
The desired isomorphism is a conclusion of Theorem \ref{thm-r-isom}.

Now we consider general $R$. Take a decomposition $R=\prod_{i=1}^{\ell}R_i$ as a ring, where $R_i$ is ring indecomposable. Then $R\otimes_k A \simeq \prod_{i=1}^{\ell}R_i\otimes_k A$, and we have isomorphisms of posets
\begin{align*}
\tors(R\otimes_k A) &= \prod_{i=1}^{\ell}\tors(R_i\otimes_k A) \simeq \prod_{i=1}^{\ell}\Hom_{\rm poset}(\Spec R_i, \tors A)\\
& \simeq \Hom_{\rm poset}(\bigsqcup_{i=1}^{\ell}\Spec R_i, \tors A) = \Hom_{\rm poset}(\Spec R, \tors A).\qedhere
\end{align*}
\end{proof}
\begin{exa}\label{example-Dynkin}
Let $Q$ be a Dynkin quiver, and $R$ a commutative Noetherian ring which contains a field $k$.
It is known that the set of torsion classes $\tors(kQ)$ is isomorphic to the Cambrian lattice $\mathfrak{C}_Q$ of $Q$ by \cite[Theorem 4.3]{Ingalls-Thomas} and \cite{Reading}.
Since $RQ \simeq R\otimes_k kQ$, Corollary \ref{cor-r-isom} gives an isomorphism of posets
\[
\tors RQ \simeq \Hom_{\rm poset}(\Spec R, \mathfrak{C}_Q).
\]
\end{exa}
\appendix
\section{Characterizations of two-term silting complexes and silting modules}\label{subsection-silting-ch}
The notion of silting modules was introduced by \cite{AMV}.
In their definition, silting modules are not necessarily finitely presented, while our silting modules defined in Definition \ref{dfn-two-term-silting} are always finitely presented by definition.
In this section, we give several characterizations of silting modules in our sense.
One of equivalent conditions shows that finitely presented silting modules in their sense is precisely silting modules in our sense, see Theorem \ref{thm-ap-i-iv}.
Moreover if $\L$ is a finite dimensional algebra over a field, then our silting modules are precisely support $\tau$-tilting modules introduced by Adachi-Iyama-Reiten \cite{Adachi-Iyama-Reiten}, see Theorem \ref{thm-ap-tau}.

Let $\L$ be a ring with $1_\L$, and $\mod\L$ the category of finitely generated $\L$-modules, and $\fp\L$ the category of finitely presented $\L$-modules.
For $M\in\mod\L$, we denote by $\Add M$ the subcategory of $\Mod\L$ consisting of direct summands of (possibly infinite) direct sums of copies of $M$.
Let $\Gen M$ be the subcategory of $\Mod\L$ consisting of factors of modules in $\Add M$. Then $\gen M = \Gen M \cap \mod\L$ holds, where $\gen M$ is defined in Section \ref{subsection-notation}.
We call an object $P=(P^i,d^i)$ in $\sK(\Proj\L)$ \emph{two-term} if $P^i=0$ for $i\neq0,-1$. (We do not assume $P^i\in\proj\L$ in this subsection.)
In this case, let
\begin{align}
\notag\cT_P & = \{X \in \Mod\L \mid \Hom_\L(f, X) \, \mbox{is surjective}\}\\
& =\{ X \in \Mod\L \mid \Hom_{\sD(\L)}(P, X[1])=0 \}. \label{Tp-2}
\end{align}
\begin{lem}\label{lem-Tp-torsion}
Let $P=(P^{-1} \xto{f} P^0)\in\sK(\Proj\L)$ be a two-term complex, and $M=H^0(P)$.
\begin{enumerate}[\rm(a)]
\item $\cT_P$ is closed under factor modules, extensions and products. It is also closed under coproducts if $P\in\sK^{\rm b}(\proj\L)$.
\item $\Ext^1_\L(M,\cT_P)=0$ holds.
\end{enumerate}
\end{lem}
\begin{proof}
(a) The first assertion follows from the equality (\ref{Tp-2}) and that $P^{-1}$ and $P^0$ belong to $\Proj\L$. The second one follows from compactness of $P$.

(b)
Since $P$ is a projective presentation of $M$ and $\Hom_\L(P^0, N) \to \Hom_\L(P^{-1}, N)$ is surjective for any $N\in\cT_P$, we obtain the assertion.
\end{proof}
\begin{prop}\label{prop-M-Tp}
For a two-term complex $P,Q\in\sK(\Proj\L)$, $M=H^0(P)$ and $N=H^0(Q)$, the following statements hold.
\begin{enumerate}[\rm (a)]
	\item $\Hom_{\sD(\L)}(P,Q[1])=0$ holds if and only if $N\in\cT_P$ holds if and only if $\Gen N\subseteq\cT_P$ holds. In this case $\Ext_{\L}^1(M, \Gen N)=0$ holds.
	\item Assume $P\in\sK^{\rm b}(\proj\L)$. Then $P$ is presilting if and only if $M\in\cT_P$ holds if and only if $\Gen M\subseteq\cT_P$ holds. In this case $\Ext_{\L}^1(M, \Gen M)=0$ holds.
\end{enumerate}
\end{prop}
\begin{proof}
(a) Let $N'=H^{-1}(Q)$.
We have a triangle $N'[1] \to Q \to N \to N'[2]$ in $\sD(\L)$.
By applying the functor $\Hom_{\sD(\L)}(P[-1],-)$ and using $\Hom_{\sD(\L)}(P[-1],N[i])=0$ for $i\ge1$, we obtain $\Hom_{\sD(\L)}(P[-1], Q) \simeq \Hom_{\sD(\L)}(P[-1], N)$.
Thus the first assertion follows. 

If $N\in\cT_P$, then $\Gen N\subseteq\cT_P$ by Lemma \ref{lem-Tp-torsion}(a) and $\Ext^1_\L(M, \Gen N)=0$ by Lemma \ref{lem-Tp-torsion}(b).

(b) This is the case $P=Q$ in (a).
\end{proof}
The following theorem gives various characterizations of silting modules.
\begin{thm}\label{thm-ap-i-iv}
Let $\L$ be a ring and $M\in\fp\L$. Then the following statements are equivalent.
\begin{enumerate}[{\rm (i)}]
	\item $M$ is a silting module, that is, there is $P\in\twosilt\L$ such that $\add H^0(P)=\add M$.
	\item There is $P\in\twopsilt\L$ such that $\add H^0(P)=\add M$ and there exists an exact sequence $\L \xto{g} M^0 \to M^1 \to 0$, where $g$ is a left $\cT_P$-approximation and $M^0, M^1\in\add M$.
	\item The condition {\rm(ii)}, where $\cT_P$ is replaced by $\Gen M$.
	\item The condition {\rm(ii)}, where $\cT_P$ is replaced by $\add M$.
	\item[{\rm (v)}] There is a two-term complex $P\in\sK^{\rm b}(\proj\L)$ such that $\add H^0(P)=\add M$ and $\cT_P=\Gen M$.
\end{enumerate}
\end{thm}

The condition (iv) is a variation of the definition of $\tau$-rigid pairs in \cite[Definition 1.3]{Iyama-Jorgensen-Yang}, and a characterization of support $\tau$-tilting modules \cite[Proposition 2.14]{Jasso} in the case where $\L$ is an Artinian algebra.
The condition (v) (respectively, (ii)) is a finitely presented version of the definition (respectively, a characterization) of large silting modules in \cite{AMV} (respectively, \cite[Proposition 3.11]{AMV}).
\begin{rem}
We can replace $\cT_P$ (respectively, $\Gen M$, $\add M$) in (ii) (respectively, (iii), (iv)) by $\cT_P\cap\mod\L$ or $\cT_P\cap \fp\L$ (respectively, $\gen M$, $\Add M$).
\end{rem}
\begin{proof}[Proof of Theorem \ref{thm-ap-i-iv}]
(i) implies (ii): Take a triangle $\L \to Q^0 \to Q^1 \to \L[1]$ in $\sD(\L)$ with $Q^0, Q^1\in\add P$ \cite[Corollary 2.4]{Iyama-Jorgensen-Yang}.
By applying $H^0$, we have an exact sequence $\L \xto{g} M^0 \to M^1 \to 0$ where $M^i=H^0(Q^i)\in\add M$.
Applying $\Hom_{\sD(\L)}(-,T)$ for each $X\in\cT_{P}$ to the triangle above, we have a commutative diagram of exact sequences:
\[\xymatrix{
\Hom_{\sD(\L)}(Q^0,X)\ar[r]\ar[d]^{H^0}&\Hom_{\sD(\L)}(\L,X)\ar[r]\ar@{=}[d]&\Hom_{\sD(\L)}(Q^1[-1],X)=0\\
\Hom_\L(M^0,X)\ar[r]^{-\circ g}&\Hom_\L(\L,X).
}\]
Thus the lower map is surjective, and hence $g$ is a left $\cT_{P}$-approximation.

(ii) implies (iii):
Since $P\in\twopsilt\L$, Proposition \ref{prop-M-Tp}(b) shows $\Gen M \subseteq \cT_P$. Thus the claim follows.

(iii) implies (iv):
Clear from $\add M \subseteq \Gen M$.

(iv) implies (i):
Since $\add H^0(P)=\add M$, there is a morphism $\tilde{g} : \L \to Q^0$ in $\sK^{\rm b}(\proj\L)$ such that $Q^0\in\add P$ and $H^0(\tilde{g})=g\oplus[0\to N^0]$ for some $N^0$.
Take a triangle $\L \xto{\tilde{g}} Q^0 \to C \to \L[1]$ in $\sK^{\rm b}(\proj\L)$.
By the following commutative diagram, $\tilde{g}$ is a left $(\add P)$-approximation:
\[
\begin{tikzcd}
\Hom_{\sD(\L)}(Q^0, P) \arrow[r, "-\circ\tilde{g}"] \arrow[d, twoheadrightarrow, "H^0"] & \Hom_{\sD(\L)}(\L, P) \arrow[d, "\simeq"] \\
\Hom_\L(M^0\oplus N^0, M) \arrow[r, twoheadrightarrow, "-\circ H^0(\tilde{g})"] & \Hom_\L(\L, M)
\end{tikzcd}
\]
Therefore $P\oplus C \in \twosilt\L$ is a co-Bongartz completion of $P$ (e.g.\ \cite[Lemma 4.2]{Iyama-Jorgensen-Yang}) and $\add H^0(P\oplus C)=\add (M\oplus M^1) = \add M$.

(ii) implies (v):
It is enough to show that $\cT_P\subseteq \Gen M$ holds.
Let $X\in \cT_P$.
Take  a surjective morphism $\L^{\oplus I} \to X$ with an index set $I$.
Then this morphism factors through $g^{\oplus I}$, where $g$ is a morphism in Theorem \ref{thm-ap-i-iv} (ii).
Thus $X\in\Gen M$ holds.

(v) implies (iii):
The module $M$ is silting in the sense of \cite[Definition 3.7]{AMV}.
Thus by \cite[Proposition 3.10]{AMV}, $M$ is a finendo quasitilting module.
By \cite[Proposition 3.2]{AMV}, there is an exact sequence $\L \xto{f} M^0 \xto{g} M^1 \to 0$ such that $M^0, M^1\in\Add M$, $f$ is a left $(\Gen M)$-approximation.
Let $M^0\oplus N^0=M^{\oplus I}$ for an index set $I$.
By replacing $g$ to $g\oplus 1_{N^0} : M^0\oplus N^0 \to M^1\oplus N^0$, we may assume that $M^0=M^{\oplus I}$.
Let $J\subseteq I$ be a finite subset such that $f(1_{\L})\in M^{\oplus J}$.
We have the following commutative diagram, where $C$ is the cokernel of $\L \to M^{\oplus J}$.
\[
\begin{tikzcd}
\L \arrow[r, "f"] \arrow[d, equal] & M^{\oplus I} \arrow[r] & M^1 \arrow[r] & 0 \\
\L \arrow[r] & M^{\oplus J} \arrow[u, hookrightarrow] \arrow[r] & C \arrow[u] \arrow[r] &0
\end{tikzcd}
\]
We have $M^{\oplus J},C\in\mod\L$ and $M^1\simeq C\oplus M^{\oplus I\setminus J}$, 
thus $C\in\Add M$. We have $C\in\add M$ by the claim below, and hence the bottom sequence in the above diagram gives Theorem \ref{thm-ap-i-iv} (iii).
\begin{enumerate}[\rm$\bullet$]
\item For modules $X, Y\in\Mod\L$ with $Y\in\Add X$, if $Y$ is finitely generated, then $Y\in\add X$ holds.
\end{enumerate}
In fact, there are morphisms $Y \xto{a} X^{\oplus I} \xto{b} Y$ with an index set $I$ such that $ba=1_Y$.
Since $Y$ is finitely generated, there is a finite subset $J\subseteq I$ such that $\Im a \subseteq X^{\oplus J}$ holds.
Then we have $(b|_{X^{\oplus J}})a=1_Y$.
\end{proof}
As an application, we give a simple proof of \cite[Theorem 4.9(1)(3)]{AMV} for the compact case.
\begin{cor}\label{2silt T=Gen}
Let $\L$ be a ring and $P\in\sK^{\rm b}(\proj\L)$ a two-term complex.
Then $P\in\twosilt\L$ holds if and only if $\cT_P=\Gen H^0(P)$ holds.
\end{cor}
\begin{proof}
The ``only if'' part follows from the proof of Theorem \ref{thm-ap-i-iv}(i)$\Rightarrow$(ii)$\Rightarrow$(v), where the complex $P$ is unchanged. 

We prove the ``if'' part. Since $H^0(P)\in\cT_P$, we have $P\in\twopsilt\L$ by Proposition \ref{prop-M-Tp}(b).
By the proof of Theorem \ref{thm-ap-i-iv}(iii)$\Rightarrow$(i), there exists $C\in\sK^{\rm b}(\proj\L)$ such that $P\oplus C\in\twosilt\L$.
We claim $C\in\add P$ which implies $P\in\twosilt\L$.
Since
\[
\Gen H^0(P) \subseteq \Gen H^0(P\oplus C) \stackrel{{\rm\ref{prop-M-Tp}}}{\subseteq} \cT_{P\oplus C} \subseteq \cT_P=\Gen H^0(P),
\]
all equalities hold. In particular, $H^0(C)\in\Gen H^0(P)$ and $\cT_P\subseteq\cT_C$ hold. Take a triangle
\begin{equation}\label{KPC}
K\to P^{\oplus I}\xrightarrow{f}C\xrightarrow{g} K[1],
\end{equation}
with $I=\Hom_{\sD(\L)}(P, C)$ and a canonical morphism $f$.  
Then $K\in\sK(\Proj\L)$ is concentrated in degrees $-1, 0, 1$. Also $H^0(f)$ is surjective since $H^0(C)\in\Gen H^0(P)$. Applying $H^0$ to \eqref{KPC}, we obtain $H^1(K)=0$.
Thus $K$ is isomorphic to a two-term complex in $\sK(\Proj\L)$.
Applying $\Hom_{\sD(\L)}(P, -)$ to \eqref{KPC}, we obtain $\Hom_{\sD(\L)}(P, K[1])=0$.
Since $\cT_P\subseteq\cT_C$, we obtain $\Hom_{\sD(\L)}(C, K[1])=0$ by Proposition \ref{prop-M-Tp}(a).
Thus $g=0$ and $C\in \Add P$ hold. Since $C$ is compact, $C\in\add P$ holds (cf.\ end of the proof of Theorem \ref{thm-ap-i-iv}(v)$\Rightarrow$(iii)).
\end{proof}
Let $\L$ be a ring. Now we give canonical isomorphisms between certain posets.
Let us recall relevant notions. 
Firstly, we denote by $\stors\L$ the set of subcategories of $\mod\L$ of the form $\gen M$ for some silting $\L$-module $M$. Clearly $(\stors\L,\subseteq)$ is a poset.
Secondly, for $P, Q\in\twosilt \L$, we write $Q\leq P$ if $\Hom(P, Q[1])=0$.
Then $(\twosilt\L, \leq)$ is a poset by \cite{AI}.
Finally, for silting modules $M, N\in\siltm\L$, we write $N\leq M$ if $\gen N \subseteq \gen M$.

Now we show that $(\siltm\L,\le)$ is also a poset, and these three posets are isomorphic to each other.
\begin{prop}\label{prop-ap-twosilt-silt-stors}
Let $\L$ be a ring.
Then $(\siltm\L, \leq)$ is a poset, and the maps $P \mapsto H^0(P)$ and $M \mapsto \gen M$ give isomorphisms of posets
	\[
	\twosilt\L \stackrel{H^0}{\longrightarrow} \siltm\L \stackrel{\gen}{\longrightarrow} \stors\L.
	\]
\end{prop}
\begin{proof}
The maps $H^0:\twosilt\L\to\siltm\L$ and $\gen:\siltm\L\to\stors\L$ are surjective by definition.

(i) We show that $\gen\circ H^0:\twosilt\L\to\stors\L$ is an isomorphism of posets. Let $P, Q\in\twosilt\L$. Then
\begin{align}\label{PleqQ}
P\leq Q \stackrel{{\rm\ref{prop-M-Tp}(a)}}{\Longleftrightarrow} H^0(P) \in \cT_Q \stackrel{{\rm\ref{2silt T=Gen}}}{=}\Gen H^0(Q) \stackrel{}{\Longleftrightarrow} \gen H^0(P) \subseteq \gen H^0(Q),
\end{align}
where the last equivalence holds since $H^0(P)$ is finitely generated. Thus the assertion follows.

(ii) We show that $(\siltm\L, \leq)$ is a poset and $\gen$ is an isomorphism of posets. Clearly $\leq$ satisfies reflexivity and transitivity. Assume that $M, N\in\siltm\L$ satisfy $\gen M = \gen N$. Take $P,Q\in\twosilt\L$ such that $M=H^0(P)$ and $N=H^0(Q)$. By \eqref{PleqQ}, $\add P =\add Q$ holds, and hence $\add M=\add H^0(P)=\add H^0(Q)=\add N$ holds. Thus the assertion follows.
\end{proof}
We give more characterizations of silting modules under certain assumptions on the ring $\L$.
We start with giving a simple proof of \cite[Lemma 5.2]{Iyama-Jorgensen-Yang}.
\begin{lem}\label{lem-ap-IJY}
Let $\L$ be a semi-perfect ring, $M\in\fp\L$ and $P=(P^{-1} \xto{d} P^0)\in\sK^{\rm b}(\proj\L)$ a minimal projective presentation of $M$.
Then the following statements are equivalent.
\begin{enumerate}[{\rm (i)}]
	\item $P$ is presilting.
	\item $\Ext^1_\L(M, \Gen M)=0$.
	\item $\Ext^1_\L(M, \gen M)=0$.
\end{enumerate}
\end{lem}
\begin{proof}
By Proposition \ref{prop-M-Tp} (i) implies (ii).
Clearly (ii) implies (iii).

We show that (iii) implies (i).
By Proposition \ref{prop-M-Tp}, it is enough to show $M\in\cT_P$.
Let $d': \Omega^2M \to P^{-1}$ be the kernel of $d$.
For any morphism $f : P^{-1}\to M$, let $\pi : M \to C:=\Cok(fd')$ be the canonical surjection.
We have $C\in\gen M$ and $\Ext^1_\L(M, C)=0$ by (iii).
Thus $\pi f d'=0$ implies that there is $a' : P^0 \to C$ such that $a'd=\pi f$.
Since $P^0$ is projective, there is $a : P^0 \to M$ such that $a'=\pi a$.
Thus $\pi(f-ad)=0$.
Since $P^{-1}$ is projective, there is $b : P^{-1} \to \Omega^2M$ such that $f-ad = fd'b$.
\[
\begin{tikzcd}
0 \arrow[r] & \Omega^2M \arrow[r, "d'"] \arrow[d, equal] & P^{-1} \arrow[r, "d"] \arrow[ld, "b"'] \arrow[d, "f"] & P^0 \arrow[r] \arrow[ld, "a"'] \arrow[d, "a' "] & M \arrow[r] & 0\\
& \Omega^2M \arrow[r, "fd' "] & M \arrow[r, "\pi"] & C \arrow[r] & 0.
\end{tikzcd}
\]
Consequently, $f=ad + fx$ holds for $x:=d'b$. Since $dx=0$, we have $fx=adx+fx^2=fx^2$ and $fx(1_{P^{-1}}-x)=0$. On the other hand, 
since $d$ is right minimal, $x\in\rad\End_\L(P^{-1})$ holds and $1_{P^{-1}}-x$ is an isomorphism.
Thus $fx=0$ and $f=ad$.
Thus $M\in\cT_P$ holds.
\end{proof}
\begin{thm}\label{thm-ap-semi-perfect}
Let $\L$ be a semi-perfect ring and $M\in\fp\L$.
Then the following conditions are equivalent to the conditions {\rm(i)--(v)} in Theorem \ref{thm-ap-i-iv}.
\begin{enumerate}
	\item[{\rm (iii$'$)}] $\Ext^1_\L(M, \Gen M)=0$ and there is an exact sequence $\L \xto{g} M^0 \to M^1 \to 0$, where $g$ is a left $(\Gen M)$-approximation and $M^0, M^1\in\add M$.
	\item[{\rm (iv$'$)}] The condition {\rm(iii$'$)}, where $\Gen M$ is replaced by $\add M$.
\end{enumerate}
\end{thm}
\begin{proof}
Let $P\in\sK^{\rm b}(\proj\L)$ be a minimal projective presentation of $M$. Using Lemma \ref{lem-ap-IJY}, we have (iii)$\Leftrightarrow$(iii$'$) and (iv)$\Leftrightarrow$(iv$'$).
\end{proof}
We denote by $\sK^2(\proj\L)$ the subcategory of $\sK^{\rm b}(\proj\L)$ consisting of two-term complexes.
It is easy to see that if $\sK^2(\proj\L)$ is a Krull-Schmidt category, then so is the category of finitely presented $\L$-modules $\fp\L$.
For an object $X$ of a Krull-Schmidt category, we denote by $|X|$ the number of isomorphism classes of indecomposable direct summands of $X$.
\begin{thm}\label{thm-ap-tau}
Let $\L$ be a ring such that $\sK^2(\proj\L)$ is Krull-Schmidt and $\Hom_{\sD(\L)}(X, Y)$ is a finitely generated $\End_{\sD(\L)}(Y)$-module for any complexes $X, Y\in\sK^2(\proj\L)$, and let $M\in\fp\L$. Then the following condition is equivalent to the conditions in Theorems \ref{thm-ap-i-iv} and \ref{thm-ap-semi-perfect}.
\begin{enumerate}
	\item[{\rm (vi)}] There is $P\in\twopsilt\L$ and an idempotent $e\in\L$ such that $H^0(P)=M$, $eM=0$ and $|M|=|\L/\ideal{e}|$ hold.
\end{enumerate}
\end{thm}
\begin{proof}
Theorem \ref{thm-ap-i-iv} (iv) is equivalent to (vi) by \cite[Theorem 4.6]{Iyama-Jorgensen-Yang}.
\end{proof}
\section*{Acknowledgements}
The second author would like to thank William Crawley-Boevey and Henning Krause for many supports and helpful comments.
The first author would like to thank the organizers of seminars of "Network on Silting Theory", where a part of result of this paper was presented.
The authors would like to thank Haruhisa Enomoto, Ryo Kanda, Hiroki Matsui, Tsutomu Nakamura, Shunya Saito and Ryo Takahashi for useful discussion.

The first author is supported by JSPS Grant-in-Aid for Scientific Research (B) 16H03923, (C) 18K03209.
The second author was partially supported by the Alexander von Humboldt Foundation in the framework of an Alexander von Humboldt Professorship endowed by the German Federal Ministry of Education and Research.

\end{document}